\documentclass{article}

\usepackage{amsmath,amsfonts,bm}

\def\1{\bm{1}}

\def\mA{{\bm{A}}}

\def\mD{{\bm{D}}}

\def\mI{{\bm{I}}}

\def\mL{{\bm{L}}}
\def\mM{{\bm{M}}}

\def\mO{{\bm{O}}}

\def\mQ{{\bm{Q}}}
\def\mR{{\bm{R}}}
\def\mS{{\bm{S}}}
\def\mT{{\bm{T}}}

\def\mV{{\bm{V}}}
\def\mW{{\bm{W}}}

\DeclareMathAlphabet{\mathsfit}{\encodingdefault}{\sfdefault}{m}{sl}
\SetMathAlphabet{\mathsfit}{bold}{\encodingdefault}{\sfdefault}{bx}{n}

\def\sR{{\mathbb{R}}}

\usepackage{multirow}
\usepackage{calligra}
\usepackage{layout}

\usepackage{amsmath}
\usepackage{amsthm}
\usepackage{amssymb}
\usepackage{amsfonts}
\usepackage{mathtools}

\usepackage{url}
\usepackage{color}
\usepackage{graphicx}
\usepackage{algorithmic}
\usepackage{algorithm}
\usepackage{verbatim}
\usepackage{thmtools} 
\usepackage{thm-restate}

\usepackage[dvipsinames]{xcolor,colortbl}
\definecolor{midnightblue}{HTML}{0059b3}

\definecolor{noonblue}{HTML}{e5eef7}
\definecolor{chromered}{HTML}{f14233}
\definecolor{olivedrab}{HTML}{6b8e23}
\definecolor{darkgreen}{HTML}{3d5114}

\usepackage[hidelinks]{hyperref}
\usepackage{nicefrac}
\usepackage{adjustbox}

\hypersetup{ 
colorlinks=true,
linkcolor = midnightblue,
urlcolor= black,
citecolor =  darkgreen          
}
\usepackage{cleveref}

\newcommand{\norm}[1]{\left\| #1 \right\|}

\newcommand{\cC}{\mathcal{C}}

\newcommand{\cS}{\mathcal{S}}
\newcommand{\cT}{\mathcal{T}}

\newcommand{\mtL}{\tilde{\mL}}

\newcommand{\mbL}{\bar{\mL}}

\newcommand{\del}[1]{}

\newcommand{\R}{\mathbb{R}} 
\newcommand{\N}{\mathbb{N}} 

\newcommand{\eqdef}{:=}

\newcommand{\Exp}[1]{{\mathbb E}\left[#1\right]}

\DeclareMathOperator{\tr}{tr}           
\DeclareMathOperator{\Diag}{Diag}       
\DeclareMathOperator{\diag}{diag}       

\newtheorem{assumption}{Assumption}
\newtheorem{lemma}{Lemma}

\newtheorem{theorem}{Theorem}
\newtheorem{proposition}{Proposition}

\newtheorem{corollary}{Corollary}

\theoremstyle{plain}

\theoremstyle{definition}

\newcommand{\brr}[1]{\left( #1 \right)}   
\newcommand{\brc}[1]{\left\{ #1 \right\}} 

\usepackage[colorinlistoftodos,bordercolor=orange,backgroundcolor=orange!20,linecolor=orange,textsize=scriptsize]{todonotes}

\newcommand{\inner}[2]{\left\langle #1, #2 \right\rangle}

\newcommand{\bbS}{\mathbb{S}}
\newcommand{\nlay}{\ell} 
\newcommand{\Romannumeral}[1]{\uppercase\expandafter{\romannumeral#1}}

\usepackage[preprint]{neurips_2023}

\usepackage[utf8]{inputenc} 
\usepackage[T1]{fontenc}    
\usepackage{hyperref}       
\usepackage{url}            
\usepackage{booktabs}       
\usepackage{amsfonts}       
\usepackage{nicefrac}       
\usepackage{microtype}      
\usepackage{xcolor}         
\usepackage{subfigure}

\title{Det-CGD: Compressed Gradient Descent with Matrix Stepsizes for Non-Convex Optimization}

\author{
  Hanmin Li \\
  AI Initiative\\
  KAUST, Thuwal, Saudi Arabia \\
  \texttt{hanmin.li@kaust.edu.sa}\\
  \And
  Avetik Karagulyan \\
  AI Initiative\\
  KAUST, Thuwal, Saudi Arabia \\
  \texttt{avetik.karagulyan@kaust.edu.sa}\\
  \And
  Peter Richtárik \\
  AI Initiative\\
  KAUST, Thuwal, Saudi Arabia \\
  \texttt{peter.richtarik@kaust.edu.sa}\\
}

\usepackage[titletoc,toc,page]{appendix}
\usepackage{minitoc}

\begin{document}

\maketitle

\addtocontents{toc}{\protect\setcounter{tocdepth}{2}}

\begin{abstract}
  
This paper introduces a new method for minimizing matrix-smooth non-convex objectives through the use of novel Compressed Gradient Descent (CGD) algorithms enhanced with a matrix-valued stepsize. 
The proposed algorithms are theoretically analyzed first in the single-node and subsequently in the distributed settings. 
Our theoretical results reveal that the matrix stepsize in CGD can capture the objective's structure and lead to faster convergence compared to a scalar stepsize. 
As a byproduct of our general results, we emphasize the importance of selecting the compression mechanism and the matrix stepsize in a layer-wise manner, taking advantage of model structure. 
Moreover, we provide theoretical guarantees for free compression, by designing specific layer-wise compressors for the non-convex matrix smooth objectives. Our findings are supported with empirical evidence.\footnote{This work was supported by the KAUST Baseline Research Funding Scheme.}

\end{abstract}

\addtocontents{toc}{\protect\setcounter{tocdepth}{0}}

\section{Introduction}

The minimization of smooth and non-convex functions is a fundamental problem in various domains of applied mathematics.
Most machine learning algorithms rely on solving optimization problems for training and inference, often with structural constraints or non-convex objectives to accurately capture the learning and prediction problems in high-dimensional or non-linear spaces. 
However, non-convex problems are typically NP-hard to solve, leading to the popular approach of relaxing them to convex problems and using traditional methods. 
Direct approaches to non-convex optimization have shown success but their convergence and properties are not well understood, making them challenging for large scale optimization. 
While its convex alternative has been extensively studied and is generally an easier problem, the non-convex setting is of greater practical interest often being the computational bottleneck in many applications.
 
In this paper, we consider the general minimization problem:
\begin{equation}
\min_{x\in \R^d} f(x),
\end{equation}
where $f:\R^d\to \R$ is a differentiable function. 
In order for this problem to have a finite solution we will assume throughout the paper that $f$ is bounded from below. 
\begin{assumption}
  \label{ass:finf} 
  There exists $f^{\inf}\in \R$ such that $f(x)\geq f^{\inf}$ for all $x\in \R^d$. 
\end{assumption} 

The stochastic gradient descent (SGD) algorithm \citep{moulines2011non, bubeck2015convex, gower2019sgd} is one of the most common algorithms to solve this problem. In its most general form, it can be written as
\begin{equation}
x^{k+1} = x^k - \gamma g(x^k),
\end{equation}
where $g(x^k)$ is a stochastic estimator of $\nabla f(x^k)$ and $\gamma > 0$ is a positive scalar stepsize. A particular case of interest is the compressed gradient descent (CGD) algorithm \citep{khirirat2018distributed}, where the estimator $g$ is taken as a compressed alternative of the initial gradient:
\begin{equation}
g(x^k) = \cC(\nabla f(x^k)),
\end{equation}
and the compressor $\cC$ is chosen to be a "sparser" estimator that aims to reduce the communication overhead in distributed or federated settings. 
This is crucial, as highlighted in the seminal paper by \cite{konevcny2016federated}, which showed that the bottleneck of distributed optimization algorithms is the communication complexity.
In order to deal with the limited resources of current devices, there are various compression objectives that are practical to achieve. 
These include also compressing the model broadcasted from server to clients for local training, and reducing the computational burden of local training. 
These objectives are mostly complementary, but compressing gradients has the potential for the greatest practical impact due to slower upload speeds of client connections and the benefits of averaging \cite{kairouz2021advances}. In this paper we will focus on this latter problem.

An important subclass of compressors are the sketches. Sketches are linear operators defined on $\R^d$, i.e., $\cC(y) = \mS y$ for every $y \in \R^d$, where $\mS$ is a random matrix. 
A standard example of such a compressor is the Rand-$k$ compressor, which randomly chooses $k$ entries of its argument and scales them with a scalar multiplier to make the estimator unbiased. 
Instead of communicating all $d$ coordinates of the gradient, one communicates only a subset of size $k$, thus reducing the number of communicated bits  by a factor of $d/k$.
Formally, Rand-$k$ is defined as follows:
$\mS = \sum_{j = 1}^k \frac{d}{k} e_{i_j}^{\top}e_{i_j}^{\top} $, where $i_j$ are the selected coordinates of the input vector.
We refer the reader to \citep{safaryan2022uncertainty} for an overview on compressions.

Besides the assumption that function $f$ is bounded from below, we also assume that it is $\mL$ matrix smooth, as we are trying to take advantage of the entire information contained in the smoothness matrix $\mL$ and the stepsize matrix $\mD$.  
\begin{assumption}[Matrix smoothness] 
  \label{ass:matrix_L} 
  There exists $\mL \in \bbS^{d}_{+}$ such that 
  \begin{equation}
    \label{eq:L-matrix-smoothness}f(x) \leq f(y) + \inner{\nabla f(y)}{x-y} + \frac{1}{2} \inner{\mL (x-y)}{x-y}
  \end{equation}
  holds for all $x,y\in \R^d$. 
\end{assumption}

The assumption of matrix smoothness, which is a generalization of scalar smoothness, has been shown to be a more powerful tool for improving supervised model training. 
In \cite{safaryan2021smoothness}, the authors proposed using smoothness matrices and suggested a novel communication sparsification strategy to reduce communication complexity in distributed optimization. 
The technique was adapted to three distributed optimization algorithms in the convex setting, resulting in significant communication complexity savings and consistently outperforming the baselines. 
The results of this study demonstrate the efficacy of the matrix smoothness assumption in improving distributed optimization algorithms.

The case of block-diagonal smoothness matrices is particularly relevant in various applications, such as neural networks (NN). In this setting, each block corresponds to a layer of the network, and we characterize the smoothness with respect to nodes in the $i$-th layer by a corresponding matrix $\mL_i$. Unlike in the scalar setting, we favor the similarity of certain entries of the argument over the others. This is because the information carried by the layers becomes more complex, while the nodes in the same layers are similar. This phenomenon has been observed visually in various studies, such as those by \cite{yosinski2015understanding} and \cite{zintgraf2017visualizing}. 

Another motivation for using a layer-dependent stepsize has its roots in physics. In nature, the propagation speed of light in media of different densities varies due to frequency variations. Similarly, different layers in neural networks carry different information, metric systems, and scaling. Thus, the stepsizes need to be picked accordingly to achieve optimal convergence.

We study two matrix stepsized CGD-type algorithms and analyze their convergence properties for non-convex matrix-smooth functions. As mentioned earlier, we put special emphasis on the block-diagonal case. We design our sketches and stepsizes in a way that leverages this structure, and we show that in certain cases, we can achieve compression without losing in the overall communication complexity.

\subsection{Related work}

Many successful convex optimization techniques have been adapted for use in the non-convex setting. 
Here is a non-exhaustive list: adaptivity \citep{dvinskikh2019adaptive, zhang2020adaptive}, variance reduction \citep{jreddi2016proximal, li2021page}, and acceleration \citep{guminov2019accelerated}. A paper of particular importance for our work is that of \cite{khaled2020better}, which proposes a unified scheme for analyzing stochastic gradient descent in the non-convex regime. A comprehensive overview of non-convex optimization can be found in \citep{jain2017non, danilova2022recent}.

A classical example of a matrix stepsized method is Newton's method. This method has been popular in the optimization community for a long time \citep{gragg1974optimal, miel1980majorizing, yamamoto1987convergence}. However, computing the stepsize as the inverse Hessian of the current iteration results in significant computational complexity. Instead, quasi-Newton methods use an easily computable estimator to replace the inverse of the Hessian \citep{broyden1965class, dennis1977quasi, al2007overview, al2014broyden}. An example is the Newton-Star algorithm \citep{islamov2021distributed}, which we discuss in \Cref{sec:main_algos}.

\cite{gower2015randomized} analyzed sketched gradient descent by making the compressors unbiased with a sketch-and-project trick. They provided an analysis of the resulting algorithm for the linear feasibility problem. Later, \cite{hanzely2018sega} proposed a variance-reduced version of this method.

Leveraging the layer-wise structure of neural networks has been widely studied for optimizing the training loss function. For example, \citep{zheng2019layer} propose SGD with different scalar stepsizes for each layer, \citep{yu2017block, ginsburg2019stochastic} propose layer-wise normalization for Stochastic Normalized Gradient Descent, and \citep{dutta2020discrepancy, wang2022theoretically} propose layer-wise compression in the distributed setting.

DCGD, proposed by \cite{khirirat2018distributed}, has since been improved in various ways, such as in \citep{horvath2019natural, li2020acceleration}. There is also a large body of literature on other federated learning algorithms with unbiased compressors \citep{alistarh2017qsgd, mishchenko2019distributed, gorbunov2021marina,pmlr-v162-mishchenko22b,maranjyan2022gradskip,horvath2023stochastic}.

\subsection{Contributions}

Our paper contributes in the following ways:
\begin{itemize}
\item We propose two novel matrix stepsize sketch CGD algorithms in \Cref{sec:main_algos}, which, to the best of our knowledge, are the first attempts to analyze a fixed matrix stepsize for non-convex optimization. 
We present a unified theorem in \Cref{sec:main-results} that guarantees stationarity for minimizing matrix-smooth non-convex functions. 
The results shows that taking our algorithms improve on their scalar alternatives. 
The complexities are summarized in \Cref{Table:comm-complex-single-node} for some particular cases.
\item We design our algorithms' sketches and stepsize to take advantage of the layer-wise structure of neural networks, assuming that the smoothness matrix is block-diagonal. 
In \Cref{sec:block}, we prove that our algorithms achieve better convergence than classical methods.
\item Assuming the that the server-to-client communication is less expensive 
\cite{konevcny2016federated,kairouz2021advances}, we propose distributed versions of our algorithms in \Cref{sec:dist-main}, following the standard FL scheme, and prove weighted stationarity guarantees. 
Our theorem recovers the result for DCGD in the scalar case and improves it in general.
\item We validate our theoretical results with experiments. The plots and framework are provided in the Appendix.
\end{itemize}

\subsection{Preliminaries}

The usual Euclidean norm on $\R^d$ is defined as $\norm{\cdot}$. 
We use bold capital letters to denote matrices. By $\mI_d$ we denote the $d\times d$ identity matrix, and by $\mO_d$ we denote the 
$d\times d$ zero matrix.
Let $\bbS^{d}_{++}$ (resp.\  $\bbS^{d}_{+}$) be the set of $d\times d$ symmetric positive definite (resp.\ semi-definite) matrices. Given  $\mQ \in \bbS^{d}_{++}$ and $x\in \R^d$, we write $\norm{x}_{\mQ}\eqdef \sqrt{\inner{\mQ x}{x}} ,$ 
where $\inner{\cdot}{\cdot}$ is the standard Euclidean inner product on $\R^d$.
For a matrix $\mA \in \bbS_{++}^d$, we define by  $\lambda_{\max}(\mA)$ (resp. $\lambda_{\min}(\mA)$) the largest (resp. smallest) eigenvalue of the matrix $\mA$. 
Let $\mA_i \in \R^{d_i \times d_i}$ and $d = d_1 + \ldots + d_\nlay$. 
Then the matrix $\mA = \Diag(\mA_1,\ldots,\mA_\nlay)$ is defined as a block diagonal $d\times d$ matrix where the $i$-th block is equal to $\mA_i$. 
We will use $\diag(\mA) \in \R^{d\times d}$ to denote the diagonal of any matrix $\mA\in\R^{d\times d}$.
Given a function $f : \R^d \rightarrow \R$, its gradient and its Hessian at point $x \in \R^d$ are respectively denoted as $\nabla f(x)$ and $\nabla^2 f(x)$.

\section{The algorithms}\label{sec:main_algos}

Below we define our two main algorithms:
\begin{equation} 
  x^{k+1} = x^k - \mD \mS^k \nabla f(x^k), \tag{det-CGD1}
  \label{eq:alg1}
\end{equation}
and 
\begin{equation}
  x^{k+1} = x^k - \mT^k\mD\nabla f(x^k). \tag{det-CGD2}
  \label{eq:alg2}
\end{equation}
Here, $\mD \in \bbS_{++}^d$ is the fixed stepsize matrix. The sequences of random matrices  $\mS^k$ and $\mT^k$  satisfy the next assumption.
\begin{assumption}\label{ass:sketch}
  We will assume that the random sketches that appear in our algorithms are i.i.d., unbiased, symmetric and positive semi-definite for each algorithm. That is 
  \begin{align}
    &\mS^k, \mT^k \in \mathbb{S}^d_{+}, 
    \quad \mS^k \overset{iid}{\sim} \cS \quad \text{and} \quad
    \mT^k \overset{iid}{\sim} \cT \notag\\
    &\Exp{\mS^k}= \Exp{\mT^k} = \mI_d, \quad \text{for every} \quad k\in \N\notag.
  \end{align}
\end{assumption}

A simple instance of \ref{eq:alg1} and \ref{eq:alg2} is the vanilla GD. 
Indeed, if $\mS^k = \mT^k = \mI_d$ and $\mD = \gamma \mI_d$, then $x^{k+1} = x^k - \gamma \nabla f(x^k)$.
In general, one may view these algorithms as   Newton-type methods. 
In particular, our setting includes the Newton Star (NS) algorithm by \cite{islamov2021distributed}:
\begin{equation} 
  x^{k+1} = x^k - \brr{\nabla^2 f(x^{\inf})}^{-1} \nabla f(x^k). \tag{NS}
\end{equation} 
The authors prove that in the convex case it converges to the unique solution $x^{\inf}$   locally quadratically, provided certain assumptions are met. 
However, it is not a practical method as it requires knowledge of the Hessian at the optimal point. 
This method, nevertheless, hints that constant matrix stepsize can yield fast convergence guarantees. 
Our results allow us to choose the $\mD$ depending on the smoothness matrix $\mL$. 
The latter can be seen as a uniform upper bound on the Hessian.

The difference between  \ref{eq:alg1} and \ref{eq:alg2} is the update rule.
In particular, the order of the sketch and the stepsize is interchanged. 
When the sketch $\mS$ and the stepsize $\mD$ are commutative w.r.t. matrix product, the algorithms become equivalent. 
In general, a simple calculation shows that if we take
\begin{equation}\label{eq:DT-relation}
    \mT^k = \mD\mS^k\mD^{-1},
\end{equation}
then \ref{eq:alg1} and \ref{eq:alg2} are the same.
Defining $\mT^k$ according to \eqref{eq:DT-relation}, we recover the unbiasedness condition:
\begin{equation}\label{eq:TS-formula}
  \Exp{\mT^k} = \mD\Exp{\mS^k}\mD^{-1} = \mI_d.
\end{equation} 
However, in general  $\mD\Exp{\mS^k}\mD^{-1}$ is not necessarily symmetric, which contradicts to \Cref{ass:sketch}. 
Thus, \ref{eq:alg1} and \ref{eq:alg2} are not equivalent for our purposes.

\section{Main results}\label{sec:main-results}
Before we state the main result, we present a stepsize condition for \ref{eq:alg1} 
and \ref{eq:alg2}, respectively:
\begin{equation}
  \label{eq:ineq_D} 
  \begin{aligned}
    \Exp{\mS^k \mD \mL \mD \mS^k} \preceq \mD,
  \end{aligned}
\end{equation}
and
\begin{equation}
      \label{eq:ineq_D-var}
  \begin{aligned}
  \Exp{\mD \mT^k \mL \mT^k \mD} \preceq \mD.
  \end{aligned}
\end{equation}
In the case of vanilla GD \eqref{eq:ineq_D} and \eqref{eq:ineq_D-var} become $\gamma < L^{-1}$, which is the standard condition for convergence.

Below is the main convergence theorem for both algorithms in the single-node regime.
\begin{theorem}\label{thm:main-D}
  Suppose that Assumptions \ref{ass:finf}-\ref{ass:sketch} are satisfied. 
  Then,  for each $k\geq 0$
  \begin{equation}
    \label{eq:main-D} 
    \frac{1}{K} \sum_{k=0}^{K-1} \Exp{\norm{\nabla f(x^k)}_{\mD}^2} \leq \frac{2 (f(x^0) - f^{\inf})}{K},
  \end{equation}
  if one of the below conditions is true:
  \begin{itemize}
    \item[i)] The vectors $x^k$ are the iterates of \ref{eq:alg1} and $\mD$ satisfies \eqref{eq:ineq_D}$;$
    \item[ii)] The vectors $x^k$ are the iterates of \ref{eq:alg2} and $\mD$ satisfies \eqref{eq:ineq_D-var}.
  \end{itemize}
  
\end{theorem}

It is important to note that \Cref{thm:main-D} yields the same convergence rate for any $\mD \in \bbS^d_{++}$, despite the fact that the matrix norms on the left-hand side cannot be compared for different weight matrices. To ensure comparability of the right-hand side of \eqref{eq:main-D}, it is necessary to normalize the weight matrix $\mD$ that is used to measure the gradient norm. We propose using determinant normalization, which involves dividing both sides of \eqref{eq:main-D} by $\det(\mD)^{1/d}$, yielding the following:  
\begin{equation}
    \label{eq:main-D-normalized-det_1} 
    \frac{1}{K} \sum_{k=0}^{K-1} \Exp{\norm{\nabla f(x^k)}_{ \frac{\mD}{\det(\mD)^{1/d}}}^2} \leq \frac{2(f(x^0) - f^{\inf})}{\det(\mD)^{1/d} K}.
  \end{equation}
This normalization is meaningful because adjusting the weight matrix to $\frac{\mD}{\det(\mD)^{1/d}}$ allows its determinant to be $1$, making the norm on the left-hand side comparable to the standard Euclidean norm. It is important to note that the volume of the normalized ellipsoid $\big\{ x\in \R^d \;:\; \norm{x}_{ {\mD}/{\det(\mD)^{1/d}}}^2 \leq 1 \big\}$ does not depend on the choice of $\mD\in \bbS^d_{++}$. Therefore, the results of \eqref{eq:main-D} are comparable across different $\mD$ in the sense that the right-hand side of \eqref{eq:main-D} measures the volume of the ellipsoid containing the gradient.

\subsection{Optimal matrix stepsize}

In this section, we describe how to choose the optimal stepsize that minimizes the iteration complexity.
The problem is easier for \ref{eq:alg2}. 
We notice that \eqref{eq:ineq_D-var} can be explicitly solved. 
Specifically, it is equivalent to
\begin{equation}\label{eq:opt-D-2}
\mD \preceq \left(\Exp{\mT^k\mL\mT^k}\right)^{-1}.
\end{equation}
We want to emphasize that the RHS matrix is invertible despite the sketches not being so. Indeed. The map $h: \mT \rightarrow \mT \mL\mT$ is convex on $\bbS_{+}^d$. Therefore, Jensen's inequality implies
\begin{equation*}
\Exp{\mT^k\mL\mT^k} \succeq \Exp{\mT^k}\mL\Exp{\mT^k} = \mL \succ \mO_d.
\end{equation*}
This explicit condition on $\mD$ can assist in determining the optimal stepsize.
 Since both $\mD$ and $(\mT^k \mL \mT^k)^{-1}$ are positive definite, then the right-hand side of \eqref{eq:main-D-normalized-det_1} is minimized exactly when 
 \begin{equation}\label{eq:opt-D-alg2}
   \mD = (\mT^k \mL \mT^k)^{-1}.
 \end{equation}

The situation is different for \ref{eq:alg1}. 
According to \eqref{eq:main-D-normalized-det_1}, the optimal $\mD$ is defined as the solution of the following constrained optimization problem:
\begin{eqnarray}
  \label{prob:log-det}
  \text{minimize} && \log \det (\mD^{-1}) \notag \\
  \text{subject to} && \Exp{\mS^k \mD \mL \mD \mS^k} \preceq \mD  \\
  && \mD \in \bbS^d_{++}. \notag
\end{eqnarray}
\begin{proposition}
  \label{thm:opt-D}
  The optimization problem \eqref{prob:log-det} with respect to stepsize matrix $\mD \in \bbS^d_{++}$, is a convex optimization problem with convex constraint.
\end{proposition}
The proof of this proposition can be found in the Appendix. 
It is based on the reformulation of the constraint to its equivalent quadratic form inequality. 
Using the trace trick, we can prove that for every vector chosen in the quadratic form, it is convex. 
Since the intersection of convex sets is convex, we conclude the proof.

One could consider using the {\tt CVXPY} \citep{diamond2016cvxpy} package to solve \eqref{prob:log-det}, provided that it is first transformed into a Disciplined Convex Programming (DCP) form \citep{grant2006disciplined}. 
Nevertheless, \eqref{eq:ineq_D} is not recognized as a DCP constraint in the general case. 
To make {\tt CVXPY} applicable, additional steps tailored to the problem at hand must be taken.

\begin{table*}[ht]\tiny
  \centering
  \renewcommand{\arraystretch}{1.5}
    \caption{Summary of communication complexities of \ref{eq:alg1} and \ref{eq:alg2} with different sketches and stepsize matrices. The $\mD_i$ here for \ref{eq:alg1} is $\mW_i$ with the optimal scaling determined using \Cref{thm:block_diag}, for \ref{eq:alg2} it is the optimal stepsize matrix defined in \eqref{eq:opt-D-2}. The constant $2(f(x^0) - f^{\inf})/\varepsilon^2$ is hidden, $\nlay$ is the number of layers, $k_i$ is the mini-batch size for the $i$-th layer if we use the rand-$k$ sketch. The notation $\mtL_{i, k}$ is defined as $\frac{d-k}{d-1}\diag(\mL_i)+\frac{k-1}{d-1}\mL_i$.}
  \vspace{.3cm}
    \begin{tabular}{p{0.01\textwidth}p{0.065\textwidth}p{0.19\textwidth}p{0.36\textwidth}p{0.22\textwidth}}
      \toprule
      No. & The method & $\brr{\mS_i^k, \mD_i}$ & $l\geq 1$, $d_i$, $k_i$, $\sum_{i=1}^{\nlay}k_i = k$, layer structure &  $l=1$, $k_i = k$, general structure\\ 
      \midrule
      1. & \ref{eq:alg1} & $\brr{\mI_d,\gamma\mL_i^{-1}}$ & $d \cdot \det(\mL)^{1/d}$ 
      & $d \cdot \det(\mL)^{1/d}$  \\
      2. & \ref{eq:alg1} & $\brr{\mI_d,\gamma\diag^{-1}(\mL_i)}$ & $d \cdot \det\big(\diag(\mL)\big)^{1/d}$ & $d \cdot \det\big(\diag(\mL)\big)^{1/d}$ \\ 
      3. & \ref{eq:alg1} & $\brr{\mI_d,\gamma\mI_{d_i}}$ & $d \cdot \left(\prod_{i=1}^{l}\lambda_{\max}^{d_i}(\mL_i)\right)^{1/d}$ & $d \cdot \lambda_{\max}(\mL)$ \\
      4. & \ref{eq:alg1} & $\brr{\text{rand-}1,\gamma\mI_{d_i}}$ & $\nlay \cdot \left(\prod_{i=1}^l d_i^{d_i}\left(\max_j(\mL_i)_{jj}\right)^{d_i}\right)^{1/d}$  & $d \cdot \max_j(\mL_{jj})$ \\
      5. & \ref{eq:alg1} &  $\brr{\text{rand-}1, \gamma\mL_i^{-1}}$ & $\nlay \cdot \left(\frac{\prod_{i=1}^{l}d_i^{d_i}\lambda_{\max}^{d_i}\left(\mL_i^{\frac{1}{2}}\diag(\mL_i^{-1})\mL_i^{\frac{1}{2}}\right)}{\prod_{i=1}^{l}\det(\mL_i^{-1})}\right)^{{1/d}}$ & $\frac{d\lambda_{\max}\left(\mL^\frac{1}{2}\diag\brr{\mL^{-1}}\mL^\frac{1}{2}\right)}{\det\left(\mL^{-1}\right)^{{1/d}}}$  \\
      6. & \ref{eq:alg1} &  $\brr{\text{rand-}1,\gamma\mL_i^{-1/2}}$ & $\nlay \cdot \left(\frac{\prod_{i=1}^l d_i^{d_i}\lambda_{\max}^{d_i}(\mL_i^{1/2})}{\prod_{i=1}^{l}\det (\mL_i^{-1/2})}\right)^{1/d}$ & $d \cdot \lambda_{\max}^{1/2}(\mL)\det(\mL)^{{1}/\brr{2d}}$ \\
      7. & \ref{eq:alg1} &  $\brr{\text{rand-}1,   \gamma\diag^{-1}(\mL_i)}$ & $\nlay \cdot \left(\frac{\prod_{i=1}^{l}d_i^{d_i}}{\prod_{j=1}^d(\mL_{jj}^{-1})}\right)^{{1/d}}$ & $d\cdot  \det\big(\diag(\mL)\big)^{1/d}$ \\
      8. & \ref{eq:alg1} &  $\brr{\text{rand-}k_i, \gamma\diag^{-1}(\mL_i)}$ & $k\cdot\left(\prod_{i=1}^{l}\left(\frac{d_i}{k_i}\right)^{d_i} \det\big(\diag(\mL)\big)\right)^{1/d}$ & $d \cdot   \det\big(\diag(\mL)\big)^{1/d}$\\
      \midrule
      9. & \ref{eq:alg2} & $\brr{\mI_d,\mL_i^{-1}}$ & $d \cdot \det(\mL)^{1/d}$ & $d \cdot \det(\mL)^{1/d}$ \\
      10. & \ref{eq:alg2} &  $\brr{\text{rand-}1, \frac{\diag^{-1}(\mL_i)}{d_i}}$ & $\nlay \cdot \left(\prod_{i=1}^{l}d_i^{d_i}\right)^{1/d}\det(\diag{\mL})^{1/d}$ & $d\cdot\det(\diag(\mL))^{1/d}$ \\
      11. &  \ref{eq:alg2} &  $\brr{\text{rand-}k, \frac{k_i}{d_i}\mtL_{i, k_i}^{-1}}$ & $k\cdot\left(\prod_{i=1}^{l}\left(\frac{d_i}{k_i}\right)^{\frac{d_i}{d}}\right)\left(\prod_{i=1}^{l}\det(\mtL_{i, k_i})\right)^{1/d}$  & $d\cdot\det(\mtL_{1,k})$\\
      12. & \ref{eq:alg2} & $\brr{\text{Bern-}q_i , q_i\mL_i^{-1}}$ & $\left(\sum_{i=1}^{l}q_id_i\right)\cdot\prod_{i=1}^{l}\left(\frac{1}{q_i}\right)^\frac{d_i}{d}\det(\mL)^{1/d}$ & $d\cdot \det(\mL)^{1/d}$\\
      \midrule
      13. & GD & $\brr{\mI_d,\lambda^{-1}_{\max}(\mL) \mI_d}$ & N/A & $d \cdot \lambda_{\max}(\mL)$ \\
      \bottomrule
    \end{tabular}
    \label{Table:comm-complex-single-node}
  \end{table*}

\section{Leveraging the layer-wise structure}\label{sec:block}

In this section we focus on the block-diagonal case of $\mL$ for both \ref{eq:alg1} and \ref{eq:alg2}.   
In particular, we propose hyper-parameters of \ref{eq:alg1} designed specifically for training NNs. 
Let us assume that $\mL = \Diag(\mL_1,\ldots,\mL_\nlay)$, where $\mL_i \in \bbS_{++}^{d_i}$. 
This setting is a generalization of the classical smoothness condition, as in the latter case $\mL_i = L \mI_{d_i}$ for all $i = 1, \ldots, \nlay$. 
Respectively, we choose both the sketches and the stepsize to be block diagonal:
$\mD = \Diag(\mD_1,\ldots,\mD_\nlay) $ and $\mS^k = \Diag(\mS^k_1,\ldots,\mS^k_\nlay)$, where 
$\mD_i,\mS^k_i \in \bbS_{++}^{d_i}$.

Let us notice that the left hand side of the inequality constraint in \eqref{prob:log-det} has quadratic dependence on $\mD$, while the right hand side is linear. 
Thus, for every matrix $\mW \in \bbS_{++}^d$, there exists $\gamma > 0$ such that 
\begin{equation*}
  \gamma^2 \lambda_{\max} \left(\Exp{\mS^k \mW \mL \mW \mS^k} \right) \leq \gamma \lambda_{\min} (\mW).
\end{equation*}
Therefore, for $\gamma \mW$ we deduce
\begin{equation}
\Exp{\mS^k (\gamma \mW) \mL  (\gamma \mW)  \mS^k} 
\preceq \gamma^2 \lambda_{\max} \left(\Exp{\mS^k \mW \mL \mW \mS^k} \right) \mI_d
\preceq \gamma \lambda_{\min} (\mW) \mI_d \preceq \gamma \mW.
\end{equation}
The following theorem is based on this simple fact applied to the corresponding blocks of the matrices $\mD,\mL,\mS^k$ for \ref{eq:alg1}.
\begin{theorem}\label{thm:block_diag} 
Let $f:\R^d\to \R$ satisfy Assumptions~\ref{ass:finf} and \ref{ass:matrix_L}, with $\mL$ admitting the layer-separable structure $\mL=\Diag(\mL_1,\dots,\mL_\nlay)$, where $\mL_1,\dots,\mL_\nlay \in \bbS^{d_i}_{++}$.   
Choose random matrices $\mS^k_1,\dots,\mS^k_\nlay \in \bbS^d_{+}$ to satisfy \Cref{ass:sketch} for all $i\in [\nlay]$, and let $\mS^k \eqdef \Diag(\mS^k_1,\dots,\mS^k_{\nlay})$.  Furthermore, choose matrices $\mW_1,\dots,\mW_\nlay \in \bbS^d_{++}$ and scalars $\gamma_1,\dots,\gamma_\nlay >0$ such that
\begin{equation}\label{eq:thm-gamma_i}
\gamma_i \leq {\lambda_{\max}^{-1}\left(\Exp{\mW_i^{-1/2} \mS^k_i \mW_i \mL_i \mW_i \mS^k_i\mW_i^{-1/2}}\right)} \qquad \forall i\in [\nlay] .\end{equation}
Letting  $\mW\eqdef \Diag(\mW_1,\dots,\mW_\nlay)$, $\Gamma\eqdef \Diag(\gamma_1\mI_{d_1},\dots,\gamma_\nlay\mI_{d_\nlay})$ and  $\mD\eqdef \Gamma \mW$, we get
\begin{equation}\label{eq:main-D-normalized-det} 
  \frac{1}{K} \sum_{k=0}^{K-1} 
  \Exp{\norm{\nabla f(x^k)}^2_{\frac{\Gamma \mW}{\det\left(\Gamma \mW\right)^{1/d}}}} \leq   \frac{2(f(x^0) - f^{\inf})}{ \det\left(\Gamma \mW\right)^{1/d}\, K}.
\end{equation}

\end{theorem}

In particular, if the scalars $\{\gamma_i\}$ are chosen to be equal to their maximum allowed values from  \eqref{eq:thm-gamma_i}, then the convergence factor  of \eqref{eq:main-D-normalized-det} is equal to
\begin{align*}
  {\det\left(\Gamma \mW\right)}^{-\frac{1}{d}} 
   =  \left[\prod_{i=1}^\nlay  \lambda^{d_i}_{\max}\left(\Exp{\mW_i^{-\frac12} \mS^k_i \mW_i \mL_i \mW_i \mS^k_i\mW_i^{-\frac12}}\right) \right] ^{\frac{1}{d}}  \det(\mW^{-1})^{\frac{1}{d}}.
 \end{align*} 

\Cref{Table:comm-complex-single-node} contains the (expected) communication complexities of \ref{eq:alg1}, \ref{eq:alg2} and GD for several choices of $\mW,\mD$ and $\mS^k$. Here are a few comments about the table.
We deduce that taking a matrix stepsize without compression (row 1) we improve GD (row 13).  
A careful analysis reveals that the result in row 5 is always worse than row 7 in terms of both communication and iteration complexity. However, the results in row 6 and row 7 are not comparable in general, meaning that neither of them is universally better.
More discussion on this table can be found in the Appendix.
 
\paragraph{Compression for free.} Now, let us focus on row 12, which corresponds to a sampling scheme where 
the $i$-th layer is independently selected with probability $q_i$.
Mathematically, it goes as follows: 
\begin{equation}\label{eq:def-bern-q}
  \mT^k_i = \frac{\eta_i}{q_i} \mI_{d_{i}}, \quad \text{where} \quad \eta_i \sim \text{Bernoulli}(q_i).
\end{equation}
Jensen's inequality implies that 
\begin{equation}
  \left(\sum_{i=1}^{l}q_i d_i\right)\cdot\prod_{i=1}^{l}
  \left(\frac{1}{q_i}\right)^\frac{d_i}{d} \geq d.
 \end{equation} 
 The equality is attained when $q_i = q$ for all $i \in [\nlay]$.
 The expected bits transferred per iteration of this algorithm is then equal to $k_{\rm exp} = qd$ and the communication complexity equals $d \det(\mL)^{1/d}$. 
 Comparing with the results for \ref{eq:alg2} with rand-$k_{\rm exp}$ on row 11 and using the fact that 
 $\det(\mL) \leq \det\brr{\diag(\mL)}$, we deduce that the Bernoulli scheme is better than the uniform sampling scheme. 
 Notice also, the communication complexity matches the one for the uncompressed \ref{eq:alg2} displayed on row 9. 
 This, in particular means that using the Bern-$q$ sketches  we can compress the gradients for free. 
 The latter means that we reduce the number of bits broadcasted at each iteration without losing in the total communication complexity. 
 In particular, when all the layers have the same width $d_i$,  the number of broadcasted bits for each iteration is reduced by a factor of $q$.

\section{Distributed setting}\label{sec:dist-main}
In this section we describe the distributed versions of our algorithms and present convergence guarantees for them.
Let us consider  an objective function that is sum decomposable:
\begin{equation*}
  f(x) \eqdef \frac{1}{n}\sum_{i=1}^n f_i(x),
\end{equation*}
where each $f_i: \mathbb{R}^d \rightarrow \mathbb{R}$ is a differentiable function. 
We assume that $f$ satisfies \Cref{ass:finf} and the component functions satisfy the below condition.
\begin{assumption}\label{ass:lower-bounded-fi}
  Each component function $f_i$ is $\mL_i$-smooth and is bounded from below:
  $f_i(x) \geq f_i^{\inf}$ for all $x \in \R^d$.
\end{assumption}
This assumption also implies that $f$ is of matrix smoothness with 
$\mbL \in \bbS^d_{++}$, where $\mbL = \frac{1}{n}\sum_{i=1}^{n} \mL_i$. 
Following the standard FL framework \citep{konevcny2016federated,mcmahan2017communication-efficient,khirirat2018distributed}, we assume that the $i$-th component function $f_i$ is stored on the $i$-th client. 
At each iteration, the clients in parallel compute and compress the local gradient $\nabla f_i$ and communicate it to the central server. 
The server, then aggregates the compressed gradients, computes the next iterate, and in parallel broadcasts it to the clients. 
See the algorithms below for the pseudo-codes.
\begin{center}

\begin{minipage}{0.49\textwidth}
  \begin{algorithm}[H]
  \caption{Distributed \ref{eq:alg1}}\label{alg:dist-alg1}
  \begin{algorithmic}[1]
  \STATE {\bfseries Input:} Starting point $x_0$, stepsize matrix $\mD$, number of iterations $K$
  \FOR {$k=0,1,2,\ldots,K-1$}
  \STATE \underline{The devices in parallel:}
  \STATE \hspace{0.01\algorithmicindent} sample $\mS_i^k \sim \cS$;
  \STATE \hspace{0.01\algorithmicindent} compute $\mS_i^k \nabla f_i(x_{k}) $;
  \STATE \hspace{0.01\algorithmicindent} broadcast $\mS_i^k \nabla f_i(x_{k}) $.
  \STATE \underline{The server:}
  \STATE \hspace{0.01\algorithmicindent} combines $g_k = \frac{\mD}{n} \sum_{i}^{n}\mS_i^k \nabla f_i(x_{k}) $;
  \STATE \hspace{0.01\algorithmicindent} computes $x_{k+1}=x_k- g_k$;
  \STATE \hspace{0.01\algorithmicindent} broadcasts $x_{k+1}$.
  \ENDFOR
  \STATE {\bfseries Return:} $x_K$
  \end{algorithmic}
  \end{algorithm}
  \end{minipage}
  \begin{minipage}{0.49\textwidth}
    \begin{algorithm}[H]
    \caption{Distributed \ref{eq:alg2}}\label{alg:dist-alg2}
    \begin{algorithmic}[1]
      \STATE {\bfseries Input:} Starting point $x_0$, stepsize matrix $\mD$, number of iterations $K$
      \FOR {$k=0,1,2,\ldots,K-1$}
      \STATE \underline{The devices in parallel:}
      \STATE \hspace{0.01\algorithmicindent} sample $\mT_i^k \sim \cT$;
      \STATE \hspace{0.01\algorithmicindent} compute $\mT_i^k \mD \nabla f_i(x_{k}) $;
      \STATE \hspace{0.01\algorithmicindent} broadcast $\mT_i^k \mD \nabla f_i(x_{k}) $.
      \STATE \underline{The server:}
      \STATE \hspace{0.01\algorithmicindent} combines $g_k = \frac{1}{n} \sum_{i}^{n}\mT_i^k \mD \nabla f_i(x_{k})$;
      \STATE \hspace{0.01\algorithmicindent} computes $x_{k+1}=x_k - g_k$;
      \STATE \hspace{0.01\algorithmicindent} broadcasts $x_{k+1}$.
      \ENDFOR
      \STATE {\bfseries Return:} $x_K$
    \end{algorithmic}
    \end{algorithm}
  \end{minipage}
\end{center}

\begin{theorem}
        \label{thm:dist-alg1}
        Let $f_i:\R^d\rightarrow\R$ satisfy \Cref{ass:lower-bounded-fi} and let $f$ satisfy \Cref{ass:finf} and \Cref{ass:matrix_L} with smoothness matrix $\mL$. 
        If the stepsize satisfies
          \begin{equation}
            \label{eq:dis-cond1-key}
            \mD\mL\mD \preceq \mD,
        \end{equation}
        then the following convergence bound  is true for the iteration of \Cref{alg:dist-alg1}:
        \begin{equation}\label{eq:thm-dist1}
            \min_{0\leq k\leq K-1}\Exp{\norm{\nabla f(x^k)}_{\frac{\mD}{\det(\mD)^{1/d}}}^2} \leq \frac{2(1 + \frac{\lambda_{\mD}}{n})^K\left(f(x^0) - f^{\inf}\right)}{\det(\mD)^{1/d}\, K} + \frac{2\lambda_{\mD} \Delta^{\inf}}{\det(\mD)^{1/d} \, n},
        \end{equation}
        where $\Delta^{\inf} := f^{\inf} - \frac{1}{n}\sum_{i=1}^{n}f_i^{\inf}$ and
        \begin{align*}
            \lambda_{\mD} &:= \max_i \left\{\lambda_{\max}\left(\Exp{\mL_i^\frac{1}{2}\left(\mS_i^k - \mI_d\right)\mD\mL\mD\left(\mS_i^k - \mI_d\right)\mL_i^\frac{1}{2}}\right)\right\}.
        \end{align*}
    \end{theorem}
    The same result is true for \Cref{alg:dist-alg2} with a different constant $\lambda_{\mD}$.
The proof of \Cref{thm:dist-alg1} and its analogue for \Cref{alg:dist-alg2} are presented in the Appendix. 
The analysis is largely inspired by \cite[Theorem 1]{khaled2020better}. 
Now, let us examine the right-hand side of \eqref{eq:thm-dist1}. 
We start by observing that the first term has exponential dependence in $K$. 
However, the term inside the brackets, $1+\lambda_{\mD}/n$, depends on the stepsize $\mD$. 
Furthermore, it has a second-order dependence on $\mD$, implying that $\lambda_{\alpha \mD} = \alpha^2 \lambda_{\mD}$, as opposed to $\det(\alpha\mD)^{1/d}$, which is linear in $\alpha$. 
Therefore, we can choose a small enough coefficient $\alpha$ to ensure that $\lambda_{\mD}$ is of order $n/K$. 
This means that for a fixed number of iterations $K$, we choose the matrix stepsize to be "small enough" to guarantee that the denominator of the first term is bounded. 
The following corollary summarizes these arguments, and its proof can be found in the Appendix.
  \begin{corollary}
        \label{cor:dist-cond-conv}
        We reach an error level of $\varepsilon^2$ in \eqref{eq:thm-dist1} if the following conditions are satisfied:
        \begin{equation}
            \label{eq:cond-dist-1}
            \mD\mL\mD \preceq \mD, \quad 
            \lambda_{\mD} \leq \min\brc{\frac{n}{K}, \frac{n\varepsilon^2}{4\Delta^{\inf}}\det(\mD)^{1/d}}, \quad
            K \geq \frac{12(f(x^0) - f^{\inf})}{\det(\mD)^{1/d}\, \varepsilon^2}.
        \end{equation}
    \end{corollary}
    \Cref{prop-dist-convex} in the Appendix proves that these conditions with respect to $\mD$ are convex.  
    In order to minimize the iteration complexity for getting $\varepsilon^2$ error, one needs to solve the following optimization problem
    \begin{eqnarray}
      \text{minimize} && \log \det (\mD^{-1}) \notag \\
      \text{subject to} && \mD \quad \text{satisfies} \quad \eqref{eq:cond-dist-1}.\notag
    \end{eqnarray}
  
    Choosing the optimal stepsize for \Cref{alg:dist-alg1} is analogous to solving \eqref{prob:log-det}. One can formulate the distributed counterpart of \Cref{thm:block_diag} and attempt to solve it for different sketches. Furthermore, this leads to a convex matrix minimization problem involving $\mD$. We provide a formal proof of this property in the Appendix.
    Similar to the single-node case, computational methods can be employed using the {\tt CVXPY} package. However, some additional effort is required to transform \eqref{eq:cond-dist-1} into the disciplined convex programming (DCP) format.

    The second term in \eqref{eq:thm-dist1} corresponds to the convergence neighborhood of the algorithm. It does not depend on the number of iteration, thus it remains unchanged, after we choose the stepsize. 
    Nevertheless, it depends on the number of clients $n$. 
    In general, the term $\Delta^{\inf}/n$ can be unbounded, when $n \rightarrow +\infty$. However, per \Cref{cor:dist-cond-conv}, we require $\lambda_{\mD}$ to be upper-bounded by $n/K$. 
    Thus, the neighborhood term will indeed converge to zero when 
    $K \rightarrow +\infty$, if we choose the stepsize accordingly.

    We compare our results with the existing results for DCGD.
    In particular we use the technique from \cite{khaled2020better} for the scalar smooth DCGD with scalar stepsizes.
    This means that the parameters of algorithms are $\mL_i = L_i \mI_d,  \mL = L \mI_d,  \mD = \gamma \mI_d,  \omega = \lambda_{\max}\brr{\Exp{\left(\mS_i^k\right)^{\top}\mS_i^k}} - 1$.
    One may check that \eqref{eq:cond-dist-1} reduces to 
    \begin{equation}
            \gamma \leq \min\brc{\frac{1}{{L}}, \sqrt{\frac{n}{KL_{\max}{L}\omega}},
            \frac{n\varepsilon^2}{4\Delta^{\inf} L_{\max}{L}\omega} } 
            \quad \text{and} \quad 
            K\gamma \geq \frac{12(f(x^0) - f^{\inf})}{\varepsilon^2}
    \end{equation}
    As expected, this coincides with the results from \cite[Corollary 1]{khaled2020better}. 
    See the Appendix for the details on the analysis of \cite{khaled2020better}.
    Finally, we back up our theoretical findings with experiments. See \Cref{fig:dist-comparison} for a simple experiment confirming that Algorithms \ref{alg:dist-alg1} and \ref{alg:dist-alg2} have better iteration and communication complexity compared to scalar stepsized DCGD. 
    For more details on the experiments we refer the reader to the corresponding section in the Appendix.

    \begin{figure}[t]
    \includegraphics[scale = 0.27]{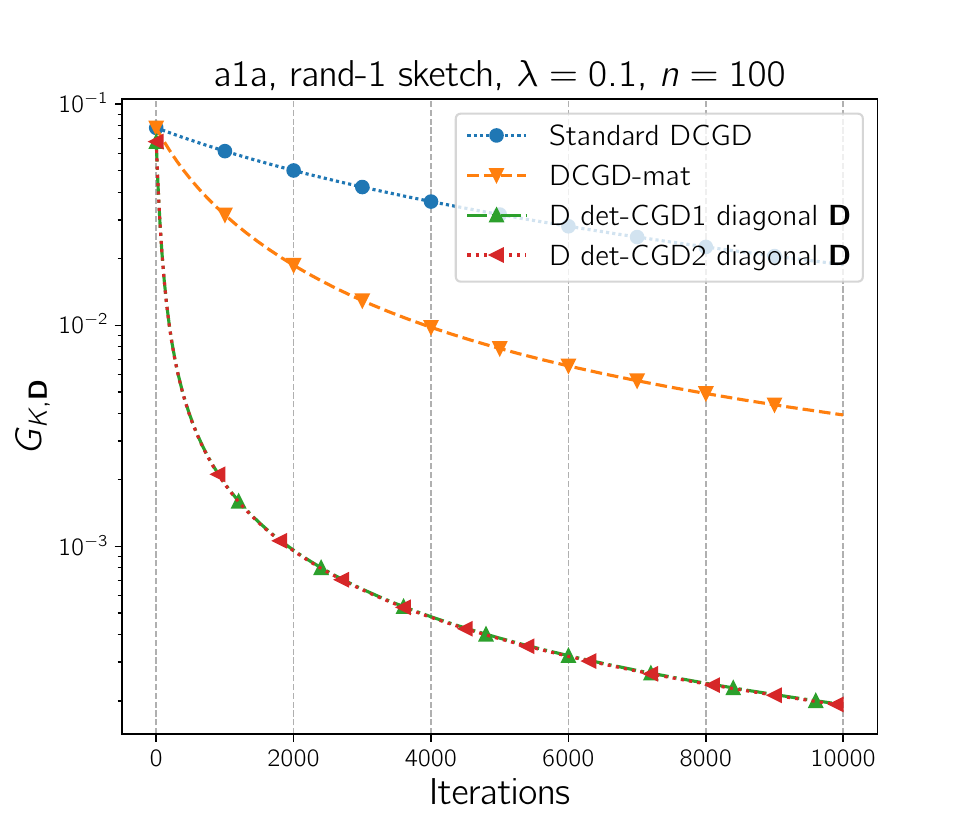}
    \includegraphics[scale = 0.27]{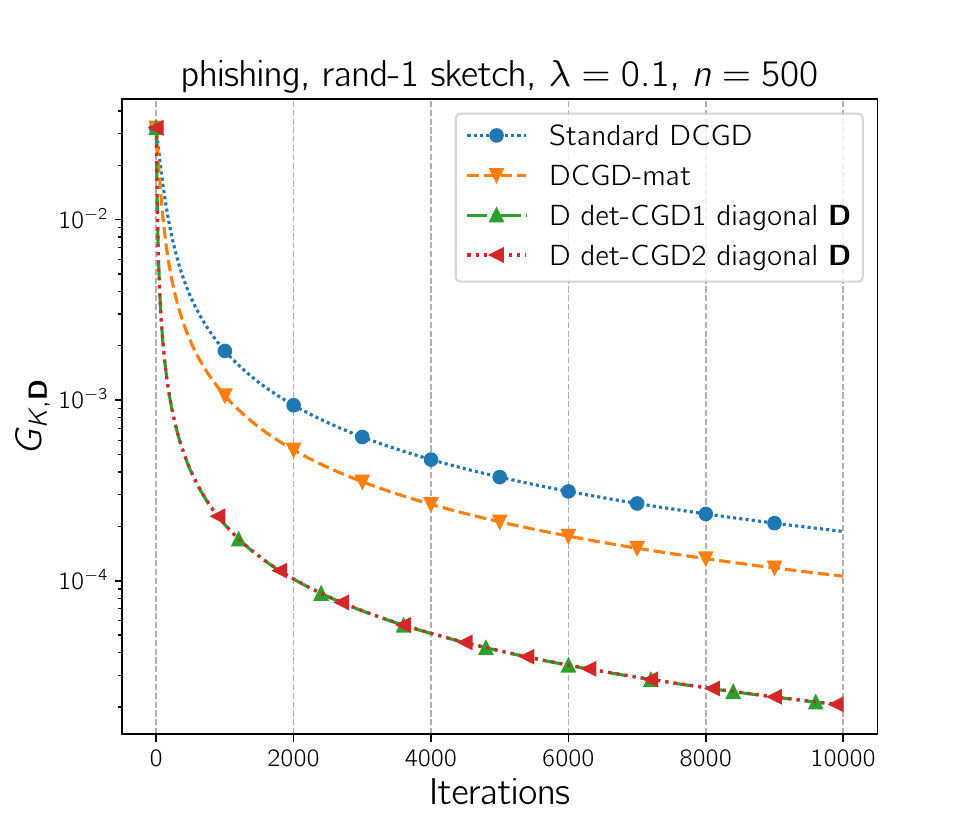}
     \includegraphics[scale = 0.27]{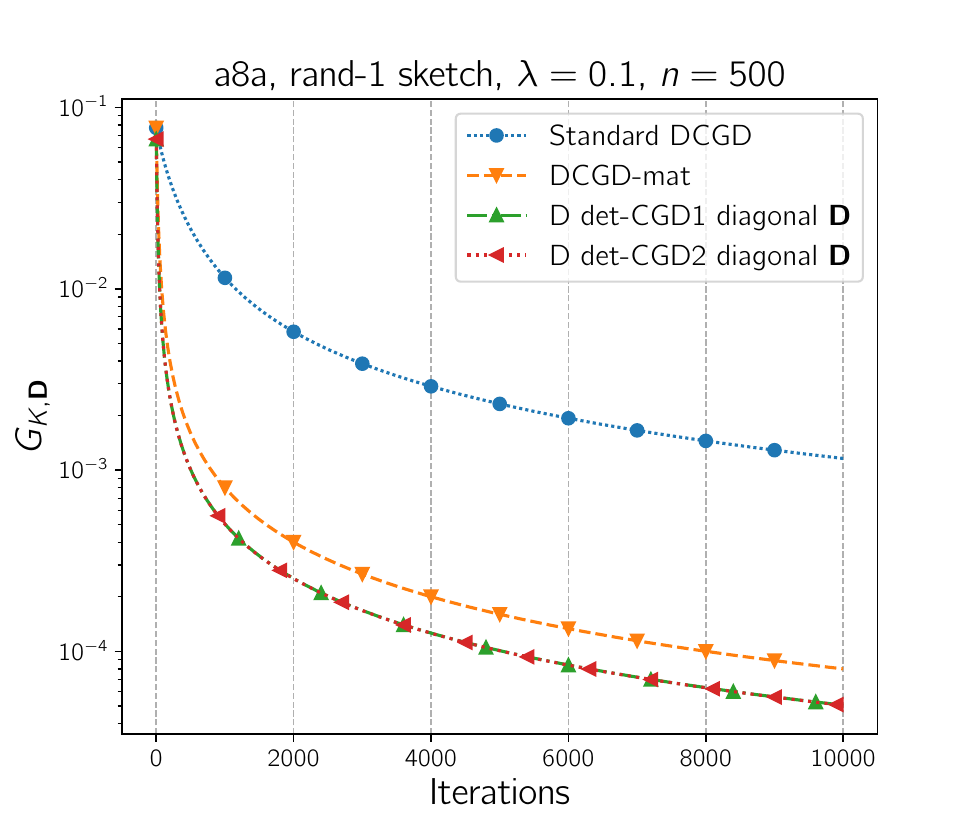}
    \caption{Comparison of standard DCGD, DCGD with matrix smoothness, D-\ref{eq:alg1} and D-\ref{eq:alg2} with optimal diagonal stepsizes under rand-$1$ sketch. The stepsize for standard DCGD is determined  using \cite[Proposition 4]{khaled2020better}, the stepsize for DCGD with matrix smoothness along with $\mD_1$, $\mD_2$ is determined using \Cref{cor:dist-cond-conv}, the error level is set to be $\varepsilon^2 = 0.0001$. Here $G_{K, \mD} := \frac{1}{K}\big(\sum_{k=0}^{K-1}\norm{\nabla f(x^k)}^2_{{\mD}/{\det(\mD)^{1/d}}}\big)$.}
    \label{fig:dist-comparison}
    \vspace{-.51cm}
    \end{figure}

\section{Conclusion}\label{sec:conclusion}

\subsection{Limitations}

It is worth noting that every point in $\R^d$ can be enclosed within some volume 1 ellipsoid. 
To see this, let $0\neq v\in \R^d$ and define $\mQ \eqdef \frac{\alpha}{\norm{v}^2} vv^\top + \beta \sum_{i=1}^{d} v_i v_i^\top$, where $v_1=\frac{v}{\norm{v}}, v_2,\dots,v_d$ form an orthonormal basis.
The eigenvalues of $\mQ$ are $\beta$ (with multiplicity $d-1$) and $\alpha$ (with multiplicity $1$), so we have $\det (\mQ) = \beta^{d-1} \alpha \leq 1$. 
Furthermore, we have $\norm{v}_{\mQ}^2 = v^\top \mQ v = \alpha \norm{v}^2$. 
By choosing $\alpha = \frac{1}{\norm{v}^2}$ and $\beta = \norm{v}^{2/(d-1)}$, we can obtain $\det (\mQ)=1$ while $\norm{v}_{\mQ}^2\leq 1 $. 
Therefore, having the average $\mD$-norm of the gradient bounded by a small number does not guarantee that the average Euclidean norm is small. 
This implies that the theory does not guarantee stationarity in the Euclidean sense.

\subsection{Future work}

Matrix stepsize gradient methods are still not well studied and require further analysis. 
Although many important algorithms have been proposed using scalar stepsizes and are known to have good performance, their matrix analogs have yet to be thoroughly examined. 
The distributed algorithms proposed in \Cref{sec:dist-main} follow the structure of DCGD by \cite{khirirat2018distributed}. However, other federated learning mechanisms such as MARINA, which has variance reduction \citep{gorbunov2021marina}, or EF21 by \cite{richtarik2021ef21}, which has powerful practical performance, should also be explored.

\bibliographystyle{apalike}
\bibliography{bibliography.bib}

\newpage

\addtocontents{toc}{\protect\setcounter{tocdepth}{3}}

\appendix

\tableofcontents

\part*{Appendix}
\section{Single node case}

\subsection{Proof of \Cref{thm:main-D}}

  \paragraph{i)}
  Using \Cref{ass:matrix_L} with $x=x^{k+1} = x^k - \mD \mS^k \nabla f(x^k)$ and $y =  x^k$, we get
  \begin{equation}
    \begin{aligned}
    \Exp{f(x^{k+1}) \mid x^k} 
    & \leq  \mathbb{E} \Big[ f(x^k) + \inner{\nabla f(x^k)}{- \mD \mS^k \nabla f(x^k)} \\ & \quad\quad+ \frac{1}{2} \inner{\mL (- \mD \mS^k \nabla f(x^k))}{- \mD \mS^k \nabla f(x^k)} \mid x^k\Big]\notag  \\
    & = f(x^k) - \inner{\nabla f(x^k)}{ \mD \Exp{\mS^k} \nabla f(x^k) } + \frac{1}{2}  \inner{\Exp{\mS^k \mD  \mL \mD \mS^k} \nabla f(x^k)}{  \nabla f(x^k)}. \notag 
  \end{aligned}
  \end{equation}
  From the unbiasedness of the sketch $\mS^k$
  \begin{eqnarray} \label{eq:same}
    \Exp{f(x^{k+1}) \mid x^k} 
    &\leq& f(x^k) - \inner{\nabla f(x^k)}{ \mD  \nabla f(x^k) } + \frac{1}{2}  \inner{\Exp{\mS^k \mD  \mL \mD \mS^k} \nabla f(x^k)}{  \nabla f(x^k)}\notag   \\
    &\overset{\eqref{eq:ineq_D}}{\leq}& f(x^k) - \inner{\nabla f(x^k)}{ \mD  \nabla f(x^k) } + \frac{1}{2}\inner{\mD \nabla f(x^k)}{  \nabla f(x^k)} \notag \\
    &=& f(x^k) - \frac{1}{2}\inner{\nabla f(x^k)}{ \mD  \nabla f(x^k) } \notag \\
    &=& f(x^k) - \frac{1}{2}\norm{ \nabla f(x^k) }_ \mD ^2.
    \label{eq:08yfdy8fd_)yfd}
  \end{eqnarray}
  Next, by subtracting $f^{\inf}$ from both sides of \eqref{eq:08yfdy8fd_)yfd}, taking expectation and applying the tower property, we get 
  \begin{eqnarray*}
    \Exp{f(x^{k+1}) } - f^{\inf} &=& \Exp{\Exp{f(x^{k+1})   \mid x^k}  } - f^{\inf} \\
    &\overset{\eqref{eq:08yfdy8fd_)yfd}}{\leq} & \Exp{f(x^{k})  - \frac{1}{2} \norm{\nabla f(x^k)}_{\mD}^2} - f^{\inf} \\
    &=&\Exp{f(x^k) }- f^{\inf}  -  \frac{1}{2} \Exp{\norm{ \nabla f(x^k) }_ \mD ^2 }. 
  \end{eqnarray*}
 Letting $\Delta^k \eqdef \Exp{f(x^k) }- f^{\inf}$, the last inequality can be written as $\Delta^{k+1} \leq \Delta^k   -  \frac{1}{2} \Exp{\norm{ \nabla f(x^k) }_ \mD ^2}.$
   Summing these inequalities for $k=0,1,\dots,K-1$, we get a telescoping effect leading to \[ \Delta^{K} \leq \Delta^0 - \frac{1}{2}\sum_{k=0}^{K-1} \Exp{\norm{\nabla f(x^k)}_\mD^2}.\] 
  It remains to rearrange the terms of this inequality, divide both sides by 
  $~ K\det(\mD)^{1/d}$ , and use the inequality $\Delta^K \geq 0$.

  \paragraph{ii)}
    Similar to the previous case, using matrix smoothness for $x=x^{k+1} = x^k - \mT^k\mD\nabla f(x^k)$ and $y = x^k $, we get
    \begin{eqnarray}
      \Exp{f(x^{k+1}) \mid x^k}
      &\leq& \mathbb{E} \Big[f(x^k) + \inner{\nabla f(x^k)}{- \mT^k\mD \nabla f(x^k)}
      \\ &&+ \frac{1}{2} \inner{\mL (- \mT^k\mD \nabla f(x^k))}{- \mT^k\mD \nabla f(x^k)} \mid x^k\Big] \notag  \\
      &=&f(x^k) - \inner{\nabla f(x^k)}{ \Exp{\mT^k} \mD  \nabla f(x^k) } 
      \\  &&+ \frac{1}{2}  \inner{\Exp{ \mD (\mT^k)^{\top} \mL \mT^k \mD} \nabla f(x^k)}{ \nabla f(x^k)}.\notag 
    \end{eqnarray}
    From \Cref{ass:sketch} and condition \eqref{eq:ineq_D-var} we deduce
    \begin{eqnarray}
        \Exp{f(x^{k+1}) \mid x^k} &\leq& f(x^k) - \inner{\nabla f(x^k)}{ \mD  \nabla f(x^k) } + \frac{1}{2}\inner{\mD \nabla f(x^k)}{  \nabla f(x^k)} \notag \\
        &=& f(x^k) - \frac{1}{2}\inner{\nabla f(x^k)}{ \mD  \nabla f(x^k) } \notag \\
        &=& f(x^k) - \frac{1}{2}\norm{ \nabla f(x^k) }_ \mD ^2. \label{eq:08yfdy8fd_)yfd_var}
    \end{eqnarray}
    Thus, we obtain the same upper bound on $\Exp{f(x^{k+1}) \mid x^k}$ as in \eqref{eq:same}.
   Following the steps from the first part, we conclude the proof.

\subsection{Proof of  \Cref{thm:opt-D}}

    Let us rewrite \eqref{eq:ineq_D} using quadratic forms. 
    That is for every non-zero $v \in \sR^d$, the following inequality must be true:
    \begin{equation*}
        v^{\top}\Exp{\mS^k\mD\mL\mD\mS^k}v \leq v^{\top}\mD v, \qquad \forall v \neq 0 
    \end{equation*}
    Notice that both sides of this inequality are real numbers, thus can be written equivalently as 
    \begin{equation*}
        \tr(v^{\top}\Exp{\mS^k\mD\mL\mD\mS^k}v) \leq \tr(v^{\top}\mD v), \qquad \forall v \neq 0
    \end{equation*}
    The LHS can be modified in the following way
    \begin{eqnarray*}
        \tr(v^{\top}\Exp{\mS^k\mD\mL\mD\mS^k}v) &\overset{\text{\Romannumeral1}}{=}& \tr\left(\Exp{v^{\top}\mS^k\mD\mL\mD\mS^kv}\right) \\
        &\overset{\text{\Romannumeral2}}{=}& \Exp{\tr(v^{\top}\mS^k\mD\mL\mD\mS^kv)} \\
        &\overset{\text{\Romannumeral3}}{=}& \Exp{\tr(\mL^\frac{1}{2}\mD\mS^k vv^{\top}\mS^k\mD\mL^\frac{1}{2})} \\
        &\overset{\text{\Romannumeral4}}{=}& \tr\left(\Exp{\mL^\frac{1}{2}\mD\mS^k vv^{\top}\mS^k\mD\mL^\frac{1}{2}}\right) \\
        &\overset{\text{\Romannumeral5}}{=}& \tr\left(\mL^\frac{1}{2}\mD\Exp{\mS^kvv^{\top}\mS^k}\mD\mL^\frac{1}{2}\right),
    \end{eqnarray*}
    where \Romannumeral1, \Romannumeral5 are due to the linearity of expectation, \Romannumeral2, \Romannumeral4 are due to the linearity of trace operator, \Romannumeral3 is obtained using the cyclic property of trace. Therefore, we can write the condition \eqref{eq:ineq_D} equivalently as 
    \begin{equation*}
        \tr\left(\mL^\frac{1}{2}\mD\Exp{\mS^kvv^{\top}\mS^k}\mD\mL^\frac{1}{2}\right) \leq \tr(vv^{\top}\mD ), \qquad \forall v \neq 0.
    \end{equation*}
    We then define function $g_v: \bbS^d_{++} \rightarrow \R$ for some fixed $v\neq 0$ as
    \begin{equation}
        g_v(\mD) := \tr\left(\mL^\frac{1}{2}\mD\Exp{\mS^kvv^{\top}\mS^k}\mD\mL^\frac{1}{2}\right) - \tr(vv^{\top}\mD).
        \label{eq:cvx-f-def}
    \end{equation}
    We want to show that for every fixed $v \neq 0$, $g$ is a convex function w.r.t $\mD$, so that in this case, the sub-level set $\{\mD \in \bbS^d_{++}\mid g_v(\mD) \leq 0\}$ is convex. 
    \begin{itemize}
        \item Notice that $vv^{\top}$ is a rank-$1$ matrix whose eigenvalues are all zero except one of them is $\|v\|^2 > 0$. We also have $(vv^{\top})^{\top} = (v^{\top})^{\top}v^{\top} = vv^{\top}$, so it is also a symmetric matrix. Thus we conclude that $vv^{\top} \in \bbS^d_{+}$ for every choice of $v$, we use $\mV = vv^{\top}$ to denote it.
        
        \item If $\mS^k = \mO_d$, then the first term is equal to $\mO_d$ and the function $g_v(\mD)$ is linear, thus, also convex. 
        Now, let us assume $\mS^k$ is nonzero. 
        Similarly $\mS^kvv^{\top}\mS^k = \mS^kv(\mS^kv)^T$ is also a symmetric positive semi-definite matrix whose eigenvalues are all $0$ except one of them is $\|\mS^kv\|^2$, this tells us that its expectation over $\mS^k$ is still a symmetric positive semi-definite matrix, we use $\mR = \Exp{\mS^kvv^{\top}\mS^k}$ to denote it.
    \end{itemize}
    Now we can write function $g_v$ as 
    \begin{equation*}
        g_v(\mD) = \tr(\mL^\frac{1}{2}\mD\mR\mD\mL^\frac{1}{2}) - \tr(\mV\mD).
    \end{equation*}
    We present the following lemma that guarantees the convexity of the first term.
  \begin{lemma}
    \label{lemma:convexity-func}
    For every matrix $\mR \in \bbS^d_{+}$, we define 
    \begin{equation}
        \label{eq:def-func-fv}
        f(\mD) = \tr(\mL^\frac{1}{2}\mD\mR\mD\mL^\frac{1}{2}),
    \end{equation}
    where $\mL, \mD \in \bbS^d_{++}$. Then function $f: \bbS^d_{++} \rightarrow \R$ is a convex function.
  \end{lemma}
  The proof can be found in \Cref{sec:lemma:convexity-func}.    
  According to \Cref{lemma:convexity-func}, the first term of $g_v(\mD)$ is a convex function, and we know that the second term is linear in $\mD$. 
    As a result, $g_v(\mD)$ is a convex function w.r.t. $\mD$ for every $v \neq 0$, thus the sub-level set $\{\mD\in\bbS^d_{++}\mid g_v(\mD) \leq 0\}$ is a convex set for every $v \neq 0$. 
    The intersection of all those convex sets corresponding to every $v \neq 0$ is still a convex set, which tells us the original condition \eqref{eq:ineq_D} is convex. 
    This concludes the proof of the proposition.

\section{Layer-wise case}
In this section, we provide interpretations about some of the results and conclusions we had in \Cref{sec:block}.

  \subsection{Proof of \Cref{thm:block_diag}}

  Note that $\Exp{\mS^k \mD \mL \mD \mS^k} = \Diag\big(\mQ^k_1,\ldots,\mQ^k_{\nlay} \big)$, where $\mQ^k_i := \gamma_{i}^2 \Exp{\mS_{i}^k \mW_{i} \mL_{i} \mD_{i}
   \mS_{i}^k}$. In other words,
    \begin{equation*}    
      \Exp{\mS^k \mD \mL \mD \mS^k} = \begin{pmatrix} 
      \mQ^k_1 & 0  & \cdots &  0 \\ 
      0 &  \mQ^k_2  & \cdots &  0 \\ 
      \vdots & \vdots  & \ddots &  \vdots \\ 
      0 & 0  & \cdots &  \mQ^k_{\nlay}  
      \end{pmatrix},
    \end{equation*}    
  which means that \eqref{eq:ineq_D} holds if and only if $\mQ^k_i \preceq \gamma_i \mW_i$ for all $i \in [\nlay]$, which holds if and only if \eqref{eq:thm-gamma_i} holds.
  Therefore,  Theorem~\ref{thm:main-D} applies, and we conclude that 
  \begin{equation}\label{eq:main-D_0} 
  \frac{1}{K} \sum_{k=0}^{K-1} \Exp{\norm{\nabla f(x^k)}_{\Gamma \mW}^2} \leq \frac{2 (f(x^0) - f^{\inf})}{K}. 
  \end{equation}
  To obtain \eqref{eq:main-D-normalized-det}, it remains to
   multiply both sides of \eqref{eq:main-D_0} by $\frac{1}{\det(\Gamma\mW)^{1/d}}$.

\subsection{Bernoulli-$q$ sketch for \ref{eq:alg2}}
The following corollary of \Cref{thm:block_diag} computes the communication complexity of \ref{eq:alg2} in the block diagonal setting with Bernoulli-$q$.
\begin{corollary}
    \label{cor:bern-sketch}
    Let  $\mT^k_i$ for the $i$-th layer in \ref{eq:alg2} be the Bern-$q_i$  sketch  which is defined as 
    \begin{equation}
        \label{eq:def-bern-sketch}
        \mT^k_i = \frac{\eta_i}{q_i} \mI_{d_{i}}, \quad \textnormal{where} \quad \eta_i \sim \textnormal{Bernoulli}(q_i).
    \end{equation} 
    Then, the communication complexity of \eqref{eq:alg2} is given by, 
    \begin{equation}
        \label{eq:comm-complex-layer}
        \frac{2(f(x^0) - f^{\inf})}{\epsilon^2}\left(\sum_{i=1}^{\nlay}q_id_i\right)\prod_{i=1}^{\nlay}\left(\frac{1}{q_i}\right)^\frac{d_i}{d}\det(\mL)^\frac{1}{d}.
    \end{equation}
    Furthermore, the communication complexity is minimized if the probabilities when $q_i = q, \quad \forall i \in [\nlay]$ and the minimum value is equal to 
    \begin{eqnarray}
        \label{eq:min-cc-alg2}
        \frac{2(f(x^0) - f^{\inf})\cdot dx~\det(\mL)^\frac{1}{d}}{\epsilon^2}  .
    \end{eqnarray}
\end{corollary}
\begin{proof}
    For \ref{eq:alg2}, its convergence requires \eqref{eq:opt-D-2}. We are using Bernoulli sketch here, so we deduce that
    \begin{eqnarray*}
        \Exp{\mT^k\mL\mT^k} &=& \Exp{\Diag(\mT^k_1\mL_1\mT^k_1, ..., \mT^k_\nlay\mL_\nlay\mT^k_\nlay)} \\
        &=&\Diag\left(\Exp{\mT^k_1\mL_1\mT^k_1}, ..., \Exp{\mT^k_\nlay\mL_\nlay\mT^k_\nlay}\right).
    \end{eqnarray*}
    Using the fact that for each block, we have 
    \begin{eqnarray*}
        \Exp{\mT^k_i\mL_i\mT^k_i} = (1-q_i)\mO_{d_i}\mL_i\mO_{d_i} + q_i\cdot\frac{1}{q_i^2}\mI_{d_i}\mL_i\mI_{d_i} = \frac{\mL_i}{q_i}, 
    \end{eqnarray*}
    we obtain 
    \begin{eqnarray*}
        \Exp{\mT^k\mL\mT^k} = \Diag\left(\frac{\mL_1}{q_1}, ..., \frac{\mL_\nlay}{q_\nlay}\right).
    \end{eqnarray*}
    Recalling \eqref{eq:opt-D-2}, the best stepsize possible is therefore given by
    \begin{eqnarray*}
        \mD &=& \left(\Exp{\mT^k\mL\mT^k}\right)^{-1} \\
        &=& \Diag^{-1}\left(\frac{\mL_1}{q_1}, ..., \frac{\mL_\nlay}{q_\nlay}\right) \\
        &=& \Diag\left(q_1\mL_1^{-1}, ..., q_\nlay\mL_\nlay^{-1}\right).
    \end{eqnarray*}
    From \eqref{eq:main-D-normalized-det_1}, we know that in order for \ref{eq:alg2} to converge to $\epsilon^2$ error level, we need 
    \begin{eqnarray*}
        \frac{2(f(x^0) - f^{\inf})}{\det(\mD)^\frac{1}{d}\, K} \leq \epsilon^2,
    \end{eqnarray*}
    which means that we need 
    \begin{eqnarray*}
        K \geq \frac{2(f(x^0) - f^{\inf})}{\det(\mD)^\frac{1}{d}\, \epsilon^2} = \frac{1}{\det(\mD)^\frac{1}{d}}\cdot \frac{2(f(x^0) - f^{\inf})}{\epsilon^2},
    \end{eqnarray*}
    iterations. For each iteration, the number of bits sent in expectation is equal to ~$\sum_{i=1}^{\nlay} q_id_i.$
    As a result, the communication complexity is given by, if we leave out the constant factor $2(f(x^0) - f^{\inf})/\epsilon^2$,
    \begin{eqnarray*}
        \left(\sum_{i=1}^{\nlay}q_id_i\right)\cdot\frac{1}{\det(\mD)^\frac{1}{d}} &=&  \left(\sum_{i=1}^{\nlay}q_id_i\right)\cdot\det(\mD^{-1})^\frac{1}{d}\\
        &=& \left(\sum_{i=1}^{\nlay}q_id_i\right)\cdot\left(\prod_{i=1}^{\nlay}\det(\frac{\mL_i}{q_i})\right)^\frac{1}{d} \\
        &=& \left(\sum_{i=1}^{\nlay}q_id_i\right)\cdot\prod_{i=1}^{\nlay}\left(\frac{1}{q_i}\right)^\frac{d_i}{d}\left(\prod_{i=1}^{l}\det(\mL_i)\right)^\frac{1}{d}\\
        &=&\left(\sum_{i=1}^{\nlay}q_id_i\right)\cdot\prod_{i=1}^{\nlay}\left(\frac{1}{q_i}\right)^\frac{d_i}{d}\cdot\det(\mL)^\frac{1}{d}. \\
    \end{eqnarray*}
    To obtain the optimal probability $q_i$, we can do the following transformation
    \begin{eqnarray*}
        \left(\sum_{i=1}^{\nlay}q_id_i\right)\cdot\frac{1}{\det(\mD)^\frac{1}{d}} &=& \left(\sum_{i=1}^{\nlay}q_i\frac{d_i}{d}\right)\cdot\prod_{i=1}^{\nlay}\left(\frac{1}{q_i}\right)^\frac{d_i}{d}\cdot d\det(\mL)^\frac{1}{d}.
    \end{eqnarray*}
    Therefore, it is equivalent to minimizing the coefficient
    \begin{eqnarray*}
        \left(\sum_{i=1}^{\nlay}q_i\frac{d_i}{d}\right)\cdot\prod_{i=1}^{\nlay}\left(\frac{1}{q_i}\right)^\frac{d_i}{d}.
    \end{eqnarray*}
    If we denote $\alpha_i = \frac{d_i}{d}$, then we know that $\alpha_i \in (0, 1]$ and $\sum_{i=1}^{\nlay} \alpha_i = 1$, the above coefficient turns into 
    \begin{eqnarray*}
        \left(\sum_{i=1}^{\nlay}\alpha_iq_i\right)\prod_{i=1}^{\nlay}\left(\frac{1}{q_i}\right)^{\alpha_i}.
    \end{eqnarray*}
    From the strict log-concavity of the $\log(\cdot)$ function and Jensen's inequality  we have 
    \begin{eqnarray*}
        \left(\sum_{i=1}^{\nlay}\alpha_iq_i\right) \geq \prod_{i=1}^{\nlay}q_i^{\alpha_i}.
    \end{eqnarray*}
    The  identity is obtained if and only if $q_i = q_j$, {for all} $i\neq j$.
    Thus, we get 
    \begin{eqnarray*}
        \left(\sum_{i=1}^{\nlay}\alpha_iq_i\right)\prod_{i=1}^{\nlay}\left(\frac{1}{q_i}\right)^{\alpha_i} \geq 1,
    \end{eqnarray*}
    which in its turn implies that the minimum of expected communication complexity is equal to 
        $d\cdot\det(\mL)^\frac{1}{d}$.
    The equality is achieved when the probabilities are equal. This concludes the proof.
\end{proof}
By utilizing the block diagonal structure, we are able to design special sketches that allow us to compress for free. 
This can be seen from row $12$, where the communication complexity of using Bernoulli compressor with equal probabilities for \ref{eq:alg2} in expectation is the same with GD, but the number of bits sent per iteration is reduced. 

\subsection{General cases for \ref{eq:alg1}}
The first part (row $1$ to row $8$) of \Cref{Table:comm-complex-single-node} records the communication complexities of \ref{eq:alg1} in the block diagonal setting and in the general setting. Depending on the types of sketches $\mS^k_i$ and matrices $\mW_i$ we are using, we can calculate the optimal scaling factor $\gamma_i$ using \Cref{thm:block_diag}. According to \eqref{eq:main-D-normalized-det_1}, in order to reach an error level of $\epsilon^2$, we need 
\begin{eqnarray}
    \label{eq:obtain-cc}
    K \geq \frac{1}{\det(\mD)^\frac{1}{d}}\cdot\frac{2(f(x^0) - f^{\inf})}{\epsilon^2},
\end{eqnarray}
where $K$ is the number of iterations in total. We can then obtain the communication complexity taking into account the number of bits transferred in each iteration in the block diagonal case. 
The same applies to the general case which can be viewed as a special case of the block diagonal setting where there is only $1$ block. 

\subsection{General cases for \ref{eq:alg2}}
The second part of \Cref{Table:comm-complex-single-node} (row $9$ to row $12$) records the communication complexities of \ref{eq:alg2}. 
Unlike \ref{eq:alg1}, we can always obtain the best stepsize matrix $\mD$ here if the sketch $\mS^k$ is given. The communication complexity can then be obtained in the same way as in the previous case using \eqref{eq:obtain-cc} combined with the number of bits sent per iteration. 

\subsection{Interpretations of \Cref{Table:comm-complex-single-node}}
The communication complexity of the Gradient Descent algorithm (row 13) in the general non-convex setting is equal to 
$d \lambda_{\max}(\mL)$, where $\lambda_{\max}(\mL)$ serves as the smoothness constant of the function.
Compared to the GD, \ref{eq:alg1} and \ref{eq:alg2} that use matrix stepsize without compression (row $1$ and $9$) are better in terms of both iteration and communication complexity. 
There are some results in the table that need careful analysis and we them present below.
In the remainder of the section we will omit the constant multiplier $2(f(x^0) - f^{\inf})/{\epsilon^2}$ in communication complexity,
as it appears for every setting in the table and thus is redundant for comparison purposes.

\subsubsection{Comparison of row $5$ and $7$}
\label{sec:worse}
Here we show that the communication complexity given in row $5$ is always worse than that of row $7$. This can be seen from the following proposition.
\begin{proposition}
    \label{cor:comp-table-result}
    For any matrix $\mL \in \bbS^d_{++}$, the following inequality holds
    \begin{eqnarray*}
        \lambda_{\max}\left(\mL^{\frac{1}{2}}\diag(\mL^{-1})\mL^{\frac{1}{2}}\right)\cdot\det(\mL)^\frac{1}{d} \geq \det(\diag(\mL))^\frac{1}{d}.
    \end{eqnarray*}
\end{proposition}
\begin{proof}
    The inequality given in \Cref{cor:comp-table-result} can be reformulated as 
    \begin{eqnarray*}
        \lambda_{\max}(\mL\diag(\mL^{-1})) \geq \det(\mL^{-1}\diag(\mL))^\frac{1}{d}.
    \end{eqnarray*}
    We use the notation 
    \begin{eqnarray*}
        \mM_1 = \mL\diag(\mL^{-1}), \quad \mM_2 = \mL^{-1}\diag(\mL),
    \end{eqnarray*}
    and notice that for any $i \in [d]$, we have
    \begin{eqnarray*}
        (\mM_1)_{ii} = (\mL)_{ii} \cdot (\mL^{-1})_{ii} = (\mM_2)_{ii}.
    \end{eqnarray*}
    Here the notation $(\mA)_{ij}$ refers to the entry $(i,j)$ of matrix $\mA$.
    As a result
    \begin{eqnarray*}
        \lambda_{\max}(\mM_1) &\geq& \left(\prod_{i=1}^{d}(\mM_1)_{ii}\right)^\frac{1}{d} = \left(\prod_{i=1}^{d}(\mM_2)_{ii}\right)^\frac{1}{d} \geq \det(\mM_2)^\frac{1}{d},
    \end{eqnarray*}
    where the first inequality is due to the fact that each diagonal element is upper-bounded by the maximum eigenvalue value, while the second one is obtained using the fact that the product of the diagonal elements is an upper bound of the determinant.
\end{proof}
From \Cref{cor:comp-table-result}, it immediately follows that the result in row $7$ is better than row $5$ in terms of both communication and iteration complexity.

\subsubsection{Comparison of row $6$ and $7$}
In this section we bring an examples of matrices $\mL$ which show that rows $6$ and $7$ are not comparable in general. 
Let $d = 2$ and $\mL \in \bbS^2_{++}$. If we pick 
\begin{equation*}
    \mL = \begin{pmatrix}
        16 & 0 \\
        0 & 1 \\
    \end{pmatrix},
\end{equation*}
then
\begin{eqnarray*}
    \det(\diag(\mL))^\frac{1}{d} &=& 4;\\
    \lambda_{\max}^\frac{1}{2}(\mL)\det(\mL)^\frac{1}{2d} &=& 8.
\end{eqnarray*}
However, if we pick 
\begin{equation*}
    \mL = \begin{pmatrix}
        16 & 3.9 \\
        3.9 & 1 \\
    \end{pmatrix},
\end{equation*}
then 
\begin{eqnarray*}
    \det(\diag(\mL))^\frac{1}{d} &=& 4;\\
    \lambda_{\max}^\frac{1}{2}(\mL)\det(\mL)^\frac{1}{2d} &\simeq& 3.88.
\end{eqnarray*}
From this example, we can see that the relation between the results in row $6$ and $7$ may vary depending on the value of $\mL$.

\section{Distributed case}

    \subsection{Proof of \Cref{thm:dist-alg1}}
        We first present some simple technical lemmas whose proofs are deferred to \Cref{sec:proofs-lemmas}.
        Let us recall that  $\mD \in \bbS^d_{++}$ is the stepsize matrix, $\mL, \mL_i \in \bbS^d_{++}$ are the smoothness matrices for $f$ and $f_i$, respectively.
        \begin{lemma}[Variance Decomposition]
            \label{lemma:var-decomp}
            For any random vector $x \in \R^d$, and any matrix $\mM \in \bbS^d_{+}$, the following identity holds 
            \begin{equation}
                \label{eq:d-norm-var-decomp}
                \Exp{\norm{x - \Exp{x}}_{\mM}^2} = \Exp{\norm{x}_{\mM}^2} - \norm{\Exp{x}}_{\mM}^2.
            \end{equation}
        \end{lemma}

        \begin{lemma}
            \label{lemma:var-sep}
            Assume $\{a_i\}_{i=1}^n$ is a set of independent random vectors in $\R^d$, which satisfy 
            \begin{equation*}
                \Exp{a_i} = 0, \quad \forall i \in [n].
            \end{equation*}
            Then, for any $\mM \in \bbS^d_{++}$, we have 
            \begin{equation}
                \label{eq:var-sep}
                \Exp{\norm{\frac{1}{n}\sum_{i=1}^{n}a_i}_{\mM}^2} = \frac{1}{n^2}\sum_{i=1}^n\Exp{\norm{a_i}^2_{\mM}}.
            \end{equation}
        \end{lemma}

        \begin{lemma}
            \label{lemma:property-of-sketch}
            For any vector $x \in \R^d$, and sketch matrix $\mS \in \bbS^d_{+}$ taken from some distribution $\cS$ over $\mS^d_{+}$, which satisfies
            \begin{equation*}
                \Exp{\mS} = \mI_d.
            \end{equation*}
            Then for any matrix $\mM \in \mS^d_{++}$, we have the following identity holds,
            \begin{equation}
                \label{eq:ineq-p-of-sketch}
                \Exp{\norm{\mS x - x}_{\mM}^2} = \norm{x}_{\Exp{\mS\mM\mS} - \mM}^2.
            \end{equation}
        \end{lemma}

        \begin{lemma} 
            \label{lemma:imp-smt-upd}
            If we have a differentiable function $f:\R^d\rightarrow \R$, that is $\mL$ matrix smooth and lower bounded by $f^{\inf}$, if we assume $\mL \in \bbS^d_{++}$, then the following inequality holds
            \begin{equation}
                \label{eq:imp-ineq-smooth}
                \inner{\nabla f(x)}{\mL^{-1}\nabla f(x)} \leq 2(f(x) - f^{\inf}).
            \end{equation}
        \end{lemma}
        Let the gradient estimator of our algorithm be defined as 
        \begin{eqnarray}
            \label{eq:def-g}
            && g(x) := \frac{1}{n}\sum_{i=1}^n\mS_i^k\nabla f_i(x),
        \end{eqnarray}
        as a result, \ref{eq:alg1} in the distributed case can then be written as
        \begin{eqnarray*}
            x^{k+1} = x^k - \mD g(x^k).
        \end{eqnarray*}
        Notice that we have 
        \begin{eqnarray}
            \label{eq:unb-g}
            && \Exp{g(x^k)\mid x^k} = \frac{1}{n}\sum_{i=1}^n\Exp{\mS^k_i}\nabla f_i(x^k) = \nabla f(x^k).
        \end{eqnarray}
        We start with applying the $\mL$-matrix smoothness of $f$:
        \begin{eqnarray*}
            f(x^{k+1}) &\leq& f(x^k) + \inner{\nabla f(x^k)}{x^{k+1} - x^k} + \frac{1}{2}\inner{\mL(x^{k+1} - x^k)}{x^{k+1} - x^k} \\
            &=& f(x^k) + \inner{\nabla f(x^k)}{-\mD g(x^k)} + \frac{1}{2}\inner{\mL\left(-\mD g(x^k)\right)}{-\mD g(x^k)} \\
            &=& f(x^k) - \inner{\nabla f(x^k)}{\mD g(x^k)} + \frac{1}{2}\inner{\mL\mD g(x^k)}{\mD g(x^k)}. \\
        \end{eqnarray*}
        Taking expectation conditioned on $x^k$, we get 
        \begin{align}
            \label{eq:dis-start}
             \Exp{f(x^{k+1})\mid x^k}
            &\leq f(x^k) - \inner{\nabla f(x^k)}{\mD\Exp{g(x^k) \mid x^k}} + \frac{1}{2}\Exp{\inner{\mL\mD g(x^k)}{\mD g(x^k)}\mid x^k} \notag\\
            &\overset{\eqref{eq:unb-g}}{=} f(x^k) - \inner{\nabla f(x^k)}{\mD\nabla f(x^k)} + \frac{1}{2}\Exp{\inner{\mL\mD g(x^k)}{\mD g(x^k)}\mid x^k} \notag\\
            &= f(x^k) - \norm{\nabla f(x^k)}_{\mD}^2 + \frac{1}{2}\underbrace{\Exp{\inner{\mL\mD g(x^k)}{\mD g(x^k)}\mid x^k}}_{:=T}.
        \end{align}
        Applying \Cref{lemma:var-decomp} to the term $T$ we obtain
        \begin{eqnarray*}
            T &=& \Exp{\norm{g(x^k)}_{\mD\mL\mD}^2\mid x^k} \notag\\
            &\overset{\eqref{eq:d-norm-var-decomp}}{=}& \Exp{\norm{g(x^k) - \Exp{g(x^k)\mid x^k}}_{\mD\mL\mD}^2\mid x^k} + \norm{\Exp{g(x^k)\mid x^k}}_{\mD\mL\mD}^2.       
        \end{eqnarray*}
        From the unbiasedness of the sketches, we have $\Exp{g(x^k) \mid x^k} = \nabla f(x^k)$, which yields
        \begin{eqnarray*}
            T &{=}& \Exp{\norm{g(x^k) - \nabla f(x^k)}_{\mD\mL\mD}^2\mid x^k} + \norm{\nabla f(x^k)}_{\mD\mL\mD}^2 \notag \\
            &=& \Exp{\norm{\frac{1}{n}\sum_{i=1}^{n}\left(\mS^k_i\nabla f_i(x^k) - \nabla f_i(x^k)\right)}_{\mD\mL\mD}^2\mid x^k} + \norm{\nabla f(x^k)}_{\mD\mL\mD}^2 \notag.
        \end{eqnarray*}
        Using \Cref{lemma:var-sep}, we have
        \begin{eqnarray}
            \label{eq:upper-on-t}
            T &{=}& \frac{1}{n^2}\sum_{i=1}^n\Exp{\norm{\mS^k_i\nabla f_i(x^k) - \nabla f_i(x^k)}^2_{\mD\mL\mD}\mid x^k}+ \norm{\nabla f(x^k)}_{\mD\mL\mD}^2 \notag \\
            &\leq& \frac{1}{n^2}\sum_{i=1}^n\Exp{\norm{\mS^k_i\nabla f_i(x^k) - \nabla f_i(x^k)}^2_{\mD\mL\mD}\mid x^k}+ \norm{\nabla f(x^k)}_{\mD}^2,
        \end{eqnarray}
        where the last inequality holds due to the inequality $\mD\mL\mD \preceq \mD$. 
        \begin{lemma}
            \label{lemma:5}
            Let $\mS$ be an unbiased $(\Exp{\mS} = \mI_d)$ sketch drawn randomly from some distribution $\cS$ over $\bbS^d_{+}$. The following bound holds for any $x \in \R^d$ and any matrix $\mA$, 
            \begin{equation}
                \label{eq:lemma-5}
                \Exp{\norm{\mS x - x}^2_{\mD\mL\mD}} \leq \lambda_{\max}\left(\mA^\frac{1}{2}\Exp{\left(\mS - \mI_d\right)\mD\mL\mD\left(\mS - \mI_d\right)}\mA^\frac{1}{2}\right)\cdot\norm{x}^2_{\mA^{-1}}.
            \end{equation}
        \end{lemma}
        Plugging \eqref{eq:upper-on-t} into \eqref{eq:dis-start} and applying Lemmas \ref{lemma:imp-smt-upd} and \ref{lemma:5} we deduce
        \begin{align*}
            \Exp{f(x^{k+1})\mid x^k} 
            &\leq f(x^k) - \frac{1}{2}\norm{\nabla f(x^k)}_{\mD}^2 \\ 
            &\qquad + \frac{1}{2n^2}\sum_{i=1}^n\Exp{\norm{\mS^k_i\nabla f_i(x^k) - \nabla f_i(x^k)}^2_{\mD\mL\mD}\mid x^k}. \\
            &\overset{\eqref{eq:lemma-5}}{\leq} f(x^k) - \frac{1}{2}\norm{\nabla f(x^k)}_{\mD}^2 \\
            & \qquad + \frac{1}{2n^2}\sum_{i=1}^{n} \lambda_{\max}\left(\Exp{\mL_i^\frac{1}{2}\left(\mS^k_i - \mI_d\right)\mD\mL\mD\left(\mS^k_i - \mI_d\right)\mL_i^\frac{1}{2}}\right)\norm{\nabla f_i(x^k)}^2_{\mL_i^{-1}}\\
            &\overset{\eqref{eq:imp-ineq-smooth}}{\leq} f(x^k) - \frac{1}{2}\norm{\nabla f(x^k)}_{\mD}^2 \\
            & \qquad + \frac{1}{n^2}\sum_{i=1}^{n} \lambda_{\max}\left(\Exp{\mL_i^\frac{1}{2}\left(\mS^k_i - \mI_d\right)\mD\mL\mD\left(\mS^k_i - \mI_d\right)\mL_i^\frac{1}{2}}\right)\left(f_i(x^k) - f^{\inf}_i\right).\\
        \end{align*}
        Recalling the definition of $\lambda_{\mD}$, we bound $f(x^{k+1})$ by
        \begin{eqnarray*}
            \Exp{f(x^{k+1})\mid x^k}
            &\leq& f(x^k) - \frac{1}{2}\norm{\nabla f(x^k)}_{\mD}^2 + \frac{1}{n^2}\sum_{i=1}^{n} \lambda_{\mD}\left(f_i(x^k) - f^{\inf}_i\right)\\
            &=&f(x^k) - \frac{1}{2}\norm{\nabla f(x^k)}_{\mD}^2 + \frac{\lambda_{\mD}}{n}\left(\frac{1}{n}\sum_{i=1}^{n}f_i(x^k) - \frac{1}{n}\sum_{i=1}^{n}f_i^{\inf}\right)\\
            &=& f(x^k) - \frac{1}{2}\norm{\nabla f(x^k)}_{\mD}^2 + \frac{\lambda_{\mD}}{n}(f(x^k) - f^{\inf}) + \frac{\lambda_{\mD}}{n}\left(f^{\inf} - \frac{1}{n}\sum_{i=1}^{n}f_i^{\inf}\right).\\
        \end{eqnarray*}
        Subtracting $f^{\inf}$ from both sides, we get 
        \begin{align*}
            \Exp{f(x^{k+1}) - f^{\inf}\mid x^k}
            \leq f(x^k) - f^{\inf}  & - \frac{1}{2}\norm{\nabla f(x^k)}_{\mD}^2 
             + \frac{\lambda_{\mD}}{n}(f(x^k) - f^{\inf}) \\ & + \frac{\lambda_{\mD}}{n}\left(f^{\inf} - \frac{1}{n}\sum_{i=1}^{n}f_i^{\inf}\right).
        \end{align*}
        Taking expectation, applying tower property and rearranging terms, we get 
        \begin{align}
            \label{eq:dist-recur}
           \Exp{f(x^{k+1}) - f^{\inf}}\notag 
            \leq \left(1 + \frac{\lambda_{\mD}}{n}\right)&\Exp{f(x^k)-f^{\inf}} - \frac{1}{2}\Exp{\norm{\nabla f(x^k)}^2_{\mD}} \\ &+ \frac{\lambda_{\mD}}{n}\left(f^{\inf}- \frac{1}{n}\sum_{i=1}^{n} f^{\inf}_i\right).
        \end{align}
        If we denote 
        \begin{equation*}
            \delta^k = \Exp{f(x^k) - f^{\inf}}, \quad r^k = \Exp{\norm{\nabla f(x^k)}_{\mD}^2}, \quad \Delta^{\inf} = f^{\inf} - \frac{1}{n}\sum_{i=1}^{n}f_i^{\inf},
        \end{equation*}
        then \eqref{eq:dist-recur} becomes 
        \begin{equation}
            \label{eq:recursion-dist}
            \frac{1}{2}r^k \leq \left(1 + \frac{\lambda_{\mD}}{n}\right)\delta^k - \delta^{k+1} + \frac{\lambda_{\mD}\Delta^{\inf}}{n}.
        \end{equation}
        In order to approach the final result, we now follow \cite{stich2019unified}, \cite{khaled2020better} and define an exponentially decaying weighting sequence $\{w_k\}_{k=-1}^K$, where $K$ is the total number of iterations. We fix $w_{-1} > 0$ and define 
        \begin{equation*}
            w_k = \frac{w_{k-1}}{1 + \lambda_{\mD}/n}, \qquad \text{ for all } \quad k \geq 0.
        \end{equation*}
        By multiplying both sides of the recursion \eqref{eq:recursion-dist} by $w_k$, we get
        \begin{eqnarray*}
            \frac{1}{2}w_kr^k \leq w_{k-1}\delta^k - w_k\delta^{k+1} + \frac{\lambda_{\mD}\Delta^{\inf}}{n}w_k.
        \end{eqnarray*}
        Summing up the inequalities from $k=0, ..., K-1$, we get 
        \begin{equation*}
            \frac{1}{2}\sum_{k=0}^{K-1}w_kr^k \leq w_{-1}\delta^0 - w_{K-1}\delta^K + \frac{\lambda_{\mD}\Delta^{\inf}}{n}\sum_{k=0}^{K-1}w^k.
        \end{equation*}
        Define $W_K = \sum_{k=0}^{K-1} w_k$, and divide both sides by $W_K$, we get
        \begin{eqnarray*}
            \frac{1}{2}\min_{0\leq k\leq K-1}r^k \leq \frac{1}{2}\frac{\sum_{k=0}^{K-1}w_kr^k}{W_K}r^k \leq \frac{w_{-1}}{W_K}\delta^0 + \frac{\lambda_{\mD}\Delta^{\inf}}{n}.
        \end{eqnarray*}
        Notice that from the definition of $w_k$, we know that the following inequality holds,
        \begin{equation*}
            \frac{w_{-1}}{W_K} \leq \frac{w_{-1}}{Kw_{K-1}} = \frac{(1+\frac{\lambda_{\mD}}{n})^K}{K}.
        \end{equation*}
        As a result, we have 
        \begin{equation*}
            \min_{0\leq k\leq K-1}r^k \leq \frac{2\left(1 + \frac{\lambda_{\mD}}{n}\right)^K}{K}\delta^0 + \frac{2\lambda_{\mD}\Delta^{\inf}}{n}.
        \end{equation*}
        Recalling the definition for $r^k$ and $\delta^k$, we get the following result,
        \begin{equation*}
            \min_{0\leq k\leq K-1}\Exp{\norm{\nabla f(x^k)}_{\mD}^2} \leq \frac{2(1 + \frac{\lambda_{\mD}}{n})^K\left(f(x^0) - f^{\inf}\right)}{K} + \frac{2\lambda_{\mD}\Delta^{\inf}}{n}.
        \end{equation*}
        Finally, we apply determinant normalization and get
        \begin{equation}
            \label{eq:TBD}
            \min_{0\leq k\leq K-1}\Exp{\norm{\nabla f(x^k)}_{\mD/\det(\mD)^{1/d}}^2} \leq \frac{2(1 + \frac{\lambda_{\mD}}{n})^K\left(f(x^0) - f^{\inf}\right)}{\det(\mD)^{1/d}K} + \frac{2\lambda_{\mD}\Delta^{\inf}}{\det(\mD)^{1/d}n}.
        \end{equation}
        This concludes the proof.

    \subsection{Convexity of the constraints}

    \begin{proposition}
        \label{prop-dist-convex}
        The set of matrices $\mD$ that satisfy \eqref{eq:cond-dist-1} is convex.
    \end{proposition}
    
    \begin{proof}
        The first inequality in \eqref{eq:cond-dist-1} can be reformulated into 
        \begin{eqnarray*}
            \mD \preceq \mL^{-1},
        \end{eqnarray*}
        which is linear in $\mD$ and, therefore, is  convex. For the second constraint in \eqref{eq:cond-dist-1}, 
        \begin{eqnarray}
            \label{eq:second-cond}
            \max_i \left\{\lambda_{\max}\left(\Exp{\mL_i^\frac{1}{2}\left(\mS^k_i - \mI_d\right)\mD\mL\mD\left(\mS^k_i - \mI_d\right)\mL_i^\frac{1}{2}}\right)\right\} \leq \frac{n}{K},
        \end{eqnarray}
        we can reformulate it into $n$ constraints, one for each client $i$:
        \begin{align*}
             \lambda_{\max}&\left(\Exp{\mL_i^\frac{1}{2}\left(\mS^k_i - \mI_d\right)\mD\mL\mD\left(\mS^k_i - \mI_d\right)\mL_i^\frac{1}{2}}\right) \leq \frac{n}{K}, \quad \forall i \in [\nlay]\\
            &\Leftrightarrow \Exp{\mL_i^\frac{1}{2}\left(\mS^k_i - \mI_d\right)\mD\mL\mD\left(\mS^k_i - \mI_d\right)\mL_i^\frac{1}{2}} \preceq \frac{n}{K}\mI_d, \quad \forall i \in [\nlay]\\
            &\Leftrightarrow \mL_i^\frac{1}{2}\Exp{\left(\mS^k_i - \mI_d\right)\mD\mL\mD\left(\mS^k_i - \mI_d\right)}\mL_i^\frac{1}{2} \preceq \frac{n}{K}\mI_d, \quad \forall i \in [\nlay] \\
            &\Leftrightarrow \Exp{\left(\mS^k_i - \mI_d\right)\mD\mL\mD\left(\mS^k_i - \mI_d\right)} \preceq \frac{n}{K}\mL_i^{-1}, \quad \forall i \in [\nlay]. \\
        \end{align*}
        We then look at the individual condition for one client $i$, 
        \begin{eqnarray}
            \label{eq:individual-cond}
            \Exp{\left(\mS^k_i - \mI_d\right)\mD\mL\mD\left(\mS^k_i - \mI_d\right)} \preceq \frac{n}{K}\mL_i^{-1},
        \end{eqnarray}
        that is for any vector  $u \in \R^d$, we require 
        \begin{align*}
            u^{\top}&\Exp{\left(\mS^k_i - \mI_d\right)\mD\mL\mD\left(\mS^k_i - \mI_d\right)}u \leq \frac{n}{K}u^{\top}\mL_i^{-1}u, \\
            &\Leftrightarrow \tr\left(u^{\top}\Exp{\left(\mS^k_i - \mI_d\right)\mD\mL\mD\left(\mS^k_i - \mI_d\right)}u\right) \leq \frac{n}{K}\tr(u^{\top}\mL_i^{-1}u),  \\
            &\Leftrightarrow\Exp{\tr(u^{\top}\left(\mS^k_i - \mI_d\right)\mD\mL\mD\left(\mS^k_i - \mI_d\right)u)} \leq \tr(u^{\top}\mL_i^{-1}u),   \\
            &\Leftrightarrow\tr(\mL^\frac{1}{2}\mD\Exp{\left(\mS^k_i - \mI_d\right)uu^{\top}\left(\mS^k_i - \mI_d\right)}\mD\mL^\frac{1}{2}) \leq \tr(u^{\top}\mL_i^{-1}u). \\
        \end{align*} 
        We now define function $g_u: \bbS^d_{++}\rightarrow \R$ for every fixed $u\neq 0$,
        \begin{equation}
            \label{eq:def-gD}
            g_u(\mD) = \tr(\mL^\frac{1}{2}\mD\Exp{\left(\mS^k_i - \mI_d\right)uu^{\top}\left(\mS^k_i - \mI_d\right)}\mD\mL^\frac{1}{2}),
        \end{equation}
        notice that $uu^{\top}$ is a rank-$1$ matrix that is positive semi-definite, so for every $y \in \R^d$, 
        \begin{eqnarray*}
            \left(\left(\mS^k_i - \mI_d\right)y\right)^{\top}uu^{\top}\left(\left(\mS^k_i - \mI_d\right)y\right) \geq 0,
        \end{eqnarray*}
        which means that $\left(\mS^k_i - \mI_d\right)uu^{\top}\left(\mS^k_i - \mI_d\right) \in \bbS^d_{+}$ , and thus $\mR := \Exp{\left(\mS^k_i - \mI_d\right)uu^{\top}\left(\mS^k_i - \mI_d\right)} \in \bbS^d_{+}$ as well. 
        Using \Cref{lemma:convexity-func}, we know that $g_u(\mD)$ is a convex function for every $0 \neq u \in \R^d$, thus its sub-level set $\{\mD \in \bbS^d_{++} \mid g_u(\mD) \leq \tr(u^{\top}\mL_i^{-1}u)\}$ is a convex set. 
        The intersection of those convex sets corresponding to the individual constraint \eqref{eq:individual-cond} of client $i$ is convex. 
        Again the intersection of those convex sets for each client $i$, which corresponds to \eqref{eq:second-cond}, is still convex.

        For the third constraint in \eqref{eq:cond-dist-1}, we can transform it using similar steps as we obtain \eqref{eq:second-cond} into 
        \begin{eqnarray}
            \label{eq:cond-3-pr}
            \Exp{\left(\mS_i^k - \mI_d\right)\mD\mL\mD\left(\mS_i^k - \mI_d\right)} \preceq \frac{n\varepsilon^2}{4\Delta^{\inf}}\det(\mD)^{1/d}\mL_i^{-1}, \quad \forall i.
        \end{eqnarray}
        If we look at each individual constraint, we can write in quadratic forms for any $0 \neq u \in \R^d$,
        \begin{eqnarray*}
            u^{\top}\Exp{\left(\mS_i^k - \mI_d\right)\mD\mL\mD\left(\mS_i^k - \mI_d\right)}u \leq \frac{n\varepsilon^2}{4\Delta^{\inf}}\det(\mD)^{1/d}\cdot u^{\top}\mL_i^{-1}u, \quad \forall u \neq 0.
        \end{eqnarray*}
        Using the linearity of expectation and the trace operator with the trace trick, we can transform the above condition into,
        \begin{eqnarray*}
            \tr(\mL^\frac{1}{2}\mD\Exp{\left(\mS^k_i - \mI_d\right)uu^{\top}\left(\mS^k_i - \mI_d\right)}\mD\mL^\frac{1}{2}) \leq \frac{n\varepsilon^2}{4\Delta^{\inf}}\det(\mD)^\frac{1}{d}\tr(u^{\top}\mL_i^{-1}u) \quad \forall u \neq 0.
        \end{eqnarray*}
        notice that we have already shown that $\mR = \Exp{\left(\mS^k_i - \mI_d\right)uu^{\top}\left(\mS^k_i - \mI_d\right)} \in \bbS^d_{+}$. Thus if we apply \Cref{lemma:convexity-func}, we know that the left-hand side of the previous inequality is convex w.r.t. $\mD$. 
        On the other hand we know that $\det(\mD)^\frac{1}{d}$ is a concave function for symmetric positive definite matrices $\mD$. 
        So the set of $\mD$ satisfying the constraint here for every $u \in \R^d$ is convex, thus their intersection is convex as well. 
        Which means that the set of $\mD$ satisfying the constraint for each client $i$ is convex. 
        Thus the intersection of those convex sets corresponding to different clients, which corresponds to \eqref{eq:cond-3-pr}, is still convex. 
        Now we know that the set of $\mD$ satisfying each of the three constraints in \eqref{eq:cond-dist-1} is convex, thus the intersection of them is convex as well. This concludes the proof. 
    \end{proof}

    \subsubsection{Proof of \Cref{cor:dist-cond-conv}}
       
        For the first term in the RHS of convergence bound \eqref{eq:thm-dist1} under condition \eqref{eq:cond-dist-1}, we know that 
        \begin{eqnarray*}
            2\brr{1 + \frac{\lambda_{\mD}}{n}}^K \leq 2\cdot\exp\brr{\lambda_{\mD}\cdot\frac{K}{n}} \leq 2\cdot\exp(1) \leq 6,
        \end{eqnarray*}
        thus 
        \begin{eqnarray*}
            \frac{2(1 + \frac{\lambda_{\mD}}{n})^K\left(f(x^0) - f^{\inf}\right)}{\det(\mD)^{1/d}\, K} &\leq& \frac{6\left(f(x^0) - f^{\inf}\right)}{\det(\mD)^{1/d}\, K}\\
            &\leq& \frac{6\left(f(x^0) - f^{\inf}\right)}{\det(\mD)^{1/d}} \cdot \frac{\varepsilon^2\det(\mD)^\frac{1}{d}}{12\left(f(x^0) - f^{\inf}\right)}\\
            &=& \frac{\varepsilon^2}{2}.
        \end{eqnarray*}
        While for the second term of RHS in \eqref{eq:thm-dist1}, we have 
        \begin{eqnarray*}
            \frac{2\lambda_{\mD} \Delta^{\inf}}{\det(\mD)^{1/d} \, n} \leq \frac{2\Delta^{\inf}}{\det(\mD)^{1/d}\, n} \cdot \frac{\varepsilon^2(\det(\mD))^{1/d}\, n}{4\Delta^{\inf}} 
            \leq \frac{\varepsilon^2}{2}.
        \end{eqnarray*}
        Thus we know that the left hand side of \eqref{eq:thm-dist1} is upper bounded by
        \begin{eqnarray*}
            \min_{0\leq k\leq K-1}\Exp{\norm{\nabla f(x^k)}_{\frac{\mD}{\det(\mD)^{1/d}}}^2} &\leq& \frac{\varepsilon^2}{2} + \frac{\varepsilon^2}{2} = \varepsilon^2.
        \end{eqnarray*}
        This concludes the proof.

    \subsection{Distributed \ref{eq:alg2}}
        We also extend \ref{eq:alg2} to the distributed case. Consider the method 
        \begin{equation}
            \label{eq:alg2-dis}
            x^{k+1} = x^k - \frac{1}{n}\sum_{i=1}^{n}\mT_i^k\mD\nabla f_i(x^k),
        \end{equation}
        where $\mD \in \bbS^d_{++}$ is the stepsize matrix, and each $\mT^k_i$ is a sequence of sketch matrices drawn randomly from some distribution $\cT$ over $\bbS^d_{+}$ independent of each other, satisfying
        \begin{equation}
            \Exp{\mT_i^k} = \mI_d.
            \label{eq:T-unbiased}
        \end{equation}

\subsubsection{Analysis of distributed \ref{eq:alg2}}
    In this section, we present the theory for \Cref{alg:dist-alg2}, which is an analogous to what we have seen for \Cref{alg:dist-alg1}. 
    We first present the following lemma which is necessary for our analysis.
    \begin{lemma}
        \label{lemma:7}
        For any sketch $\mT^k_i$ of client $i$ drawn randomly from some distribution $\cT$ over $\bbS^d_{+}$ which satisfies
        \begin{equation*}
            \Exp{\mT^k_i} = \mI_d,
        \end{equation*}
        the following inequality holds for any $x \in \R^d$ for each client $i$,
        \begin{equation}
            \label{eq:lemma-7}
            \Exp{\norm{\mT^k_i\mD x - \mD x}^2_{\mL}} \leq \lambda_{\max}\left(\mL_i^\frac{1}{2}\mD\Exp{\left(\mT^k_i - \mI_d\right)\mL\left(\mT^k_i - \mI_d\right)}\mD\mL_i^\frac{1}{2}\right)\cdot\norm{x}^2_{\mL_i^{-1}}.
        \end{equation}
    \end{lemma}
    
    \begin{theorem}
        \label{thm:conv-2}
        Let $f_i:\R^d\rightarrow\R$ satisfy \Cref{ass:lower-bounded-fi} and let $f$ satisfy Assumptions \ref{ass:finf} and \ref{ass:matrix_L} with a smoothness matrix $\mL$. 
        If the stepsize satisfies,
        \begin{equation}
            \label{eq:ineq-upperbound-D-alg2}
            \mD\mL\mD \preceq \mD,
        \end{equation}
        then the following convergence bound is true for the iteration of \Cref{alg:dist-alg2}
        \begin{equation}
            \label{eq:thm-conv-2}
            \min_{0\leq k\leq K-1}\Exp{\norm{\nabla f(x^k)}_{\frac{\mD}{\det(\mD)^{1/d}}}^2} \leq \frac{2(1 + \frac{\lambda_{\mD}^{\prime}}{n})^K\left(f(x^0) - f^{\inf}\right)}{\det(\mD)^{1/d}\, K} + \frac{2\lambda_{\mD}^{\prime}\Delta^{\inf}}{\det(\mD)^{1/d}\, n},
        \end{equation}
        where $\Delta^{\inf} := f^{\inf} - \frac{1}{n}\sum_{i=1}^{n}f_i^{\inf}$ and
        \begin{eqnarray*}
            \lambda_{\mD}^{\prime} &:=& \max_i \left\{\lambda_{\max}\left(\Exp{\mL_i^\frac{1}{2}\mD\left(\mT^k_i - \mI_d\right)\mL\left(\mT^k_i - \mI_d\right)\mD\mL_i^\frac{1}{2}}\right)\right\}. \\
        \end{eqnarray*}
    \end{theorem}
    \begin{proof}
        We first define function $g(x)$ as follows,
        \begin{equation*}
            g(x) = \frac{1}{n}\sum_{i=1}^{n}\mT_i^k\mD\nabla f_i(x^k).
        \end{equation*}
        As a result, \Cref{alg:dist-alg2} can be written as
        \begin{eqnarray*}
            x^{k+1} = x^k - g(x^k).
        \end{eqnarray*}
        Notice that 
        \begin{equation}
            \label{eq:unb-g-2}
            \Exp{g(x)} = \frac{1}{n}\sum_{i=1}^{n}\Exp{\mT_i^k}\mD\nabla f_i(x) = \mD\nabla f(x).
        \end{equation}
        We then start with the $\mL$ matrix smoothness of function $f$,
        \begin{eqnarray*}
            f(x^{k+1}) &\leq& f(x^k) + \inner{\nabla f(x^k)}{x^{k+1} - x^k} + \frac{1}{2}\inner{\mL(x^{k+1} - x^k)}{x^{k+1} - x^k} \\
            &=& f(x^k) + \inner{\nabla f(x^k)}{-g(x^k)} + \frac{1}{2}\inner{\mL\left(-g(x^k)\right)}{-g(x^k)} \\
            &=& f(x^k) - \inner{\nabla f(x^k)}{g(x^k)} + \frac{1}{2}\inner{\mL g(x^k)}{g(x^k)}. \\
        \end{eqnarray*}
        We then take expectation conditioned on $x^k$,
        \begin{align}
            \label{eq:proof-temp1}
            \Exp{f(x^{k+1}) \mid x^k} 
            &\leq f(x^k) - \inner{\nabla f(x^k)}{\Exp{g(x^k)\mid x^k}} + \frac{1}{2}\Exp{\inner{\mL g(x^k)}{g(x^k)}\mid x^k} \notag\\
            &= f(x^k) - \inner{\nabla f(x^k)}{\mD\nabla f(x^k)} + \frac{1}{2}\underbrace{\Exp{\inner{\mL g(x^k)}{g(x^k)}\mid x^k}}_{:= T}.
        \end{align}
        \Cref{lemma:var-decomp} yields
        \begin{eqnarray*}
            T &=& \Exp{\norm{g(x^k)}_{\mL}^2 \mid x^k} \\
            &\overset{\eqref{eq:d-norm-var-decomp}}{=}& \Exp{\norm{g(x^k) - \Exp{g(x^k) \mid x^k}^2_{\mL}}\mid x^k} + \norm{\Exp{g(x)\mid x^k}}^2_{\mL}.
        \end{eqnarray*}
        From \eqref{eq:unb-g-2} we deduce
        \begin{eqnarray*}
            T 
            &{=}& \Exp{\norm{g(x^k) - \mD\nabla f(x^k)}_{\mL}^2 \mid x^k} + \norm{\mD \nabla f(x^k)}_{\mL}^2 \\
            &=&\Exp{\norm{\frac{1}{n}\sum_{i=1}^n\mT_i^k\mD\nabla f_i(x^k) - \frac{\mD}{n}\sum_{i=1}^{n}\nabla f_i(x^k)}_{\mL}^2 \mid x^k} + \norm{\nabla f(x^k)}_{\mD\mL\mD}^2 \\
            &=&\Exp{\norm{\frac{1}{n}\sum_{i=1}\left(\mT_i^k\mD - \mD\right)\nabla f_i(x^k)}_{\mL}^2 \mid x^k} + \norm{\nabla f(x^k)}_{\mD\mL\mD}^2 .
        \end{eqnarray*}
        Recalling \Cref{lemma:var-sep} we obtain
        \begin{eqnarray*}
            T 
            & \overset{\eqref{eq:var-sep}}{=}& \frac{1}{n^2}\sum_{i=1}^{n}\Exp{\norm{\mT_i^k\mD\nabla f_i(x^k) - \mD\nabla f_i(x^k)}_{\mL}^2\mid x^k} + \norm{\nabla f(x^k)}_{\mD\mL\mD}^2 \\
            &\overset{\eqref{eq:ineq-upperbound-D-alg2}}{\leq }&\frac{1}{n^2}\sum_{i=1}^{n}\Exp{\norm{\mT_i^k\mD\nabla f_i(x^k) - \mD\nabla f_i(x^k)}_{\mL}^2 \mid x^k} + \norm{\nabla f(x^k)}_{\mD}^2.
        \end{eqnarray*}
        By applying \Cref{lemma:7}, we get 
        \begin{eqnarray*}
            T&{\leq}& \frac{1}{n^2}\sum_{i=1}^{n}\lambda_{\max}\left(\mL_i^\frac{1}{2}\mD\Exp{(\mT^k_i-\mI_d)\mL(\mT^k_i-\mI_d)}\mD\mL_i^\frac{1}{2}\right)\norm{\nabla f_i(x^k)}^2_{\mL_i^{-1}} + \norm{\nabla f(x^k)}_{\mD}^2\\
            &\overset{\eqref{eq:imp-ineq-smooth}}{\leq}& \lambda_{\mD}^{\prime}\cdot\frac{2}{n}\left(f(x^k) - \frac{1}{n}\sum_{i=1}^{n}f_i^{\inf}\right) + \norm{\nabla f(x^k)}_{\mD}^2.
        \end{eqnarray*}
        Then we plug the upper bound of $T$ back into \eqref{eq:proof-temp1}, we get 
        \begin{eqnarray*}
            && \Exp{f(x^{k+1}) \mid x^k} \\
            &\leq& f(x^k) - \frac{1}{2}\norm{\nabla f(x^k)}_{\mD}^2 + \frac{\lambda_{\mD}^{\prime}}{n}\left(f(x^k) - f^{\inf}\right) + \frac{\lambda_{\mD}^{\prime}}{n}(f^{\inf} - \frac{1}{n}\sum_{i=1}^{n}f_i^{\inf}).
        \end{eqnarray*}
        Taking expectation, subtracting $f^{\inf}$ from both sides, and using tower property, we get 
        \begin{align*}
            \mathbb{E}\Big[f(x^{k+1}) &- f^{\inf} \Big]\\
            &\leq \Exp{f(x^k) -f^{\inf}} - \frac{1}{2}\Exp{\norm{\nabla f(x^k)}_{\mD}^2} + \frac{\lambda_{\mD}^{\prime}}{n}\Exp{f(x^k) - f^{\inf}} + \frac{\lambda_{\mD}^{\prime}}{n}\Delta^{\inf}.
        \end{align*}
        Then following similar steps as in the proof of \Cref{thm:dist-alg1}, we are able to get
        \begin{equation*}
            \min_{0\leq k\leq K-1}\Exp{\norm{\nabla f(x^k)}_{\frac{\mD}{\det(\mD)^{1/d}}}^2} \leq \frac{2(1 + \frac{\lambda_{\mD}^{\prime}}{n})^K\left(f(x^0) - f^{\inf}\right)}{\det(\mD)^{1/d}\, K} + \frac{2\lambda_{\mD}^{\prime}\Delta^{\inf}}{\det(\mD)^{1/d}\, n}.
        \end{equation*}
        This concludes the proof.
    \end{proof}
    Similar to \Cref{alg:dist-alg1}, we can choose the parameters of the algorithm to avoid the exponential blow-up in convergence bound \eqref{eq:thm-conv-2}. 
    The following corollary sums up the convergence conditions for \Cref{alg:dist-alg2}.
    \begin{corollary}
        \label{cor:dist-cond-conv-2}
        We reach an error level of $\varepsilon^2$ in \eqref{eq:thm-conv-2} if the following conditions are satisfied:
        \begin{equation}
            \label{eq:cond-dist-2}
            \mD\mL\mD \preceq \mD, \quad 
            \lambda_{\mD}^{\prime} \leq \min\brc{\frac{n}{K}, \frac{n\varepsilon^2}{4\Delta^{\inf}}\det(\mD)^{1/d}}, \quad
            K \geq \frac{12(f(x^0) - f^{\inf})}{\det(\mD)^{1/d}\varepsilon^2}.
        \end{equation}
    \end{corollary}
    The proof of this corollary is exactly the same as for \Cref{cor:dist-cond-conv}.

    \subsubsection{Optimal stepsize}
        In order to minimize the iteration complexity for \Cref{alg:dist-alg2}, the following optimization problem needs to be solved  
        \begin{eqnarray*}
            \min && \log \det (\mD^{-1}) \\
            \text{subject to} && \mD \quad \text{satisfies} \quad \eqref{eq:cond-dist-2}
        \end{eqnarray*}
        Following similar techniques in the proof of \Cref{prop-dist-convex}, we are able to prove that the above optimization problem is still a convex optimization problem. 
        One simple way to find stepsize matrices is to follow the scheme suggested for solving \eqref{prob:log-det}. 
        That is we first fix $\mW \in \bbS^d_{++}$ and we find the optimal $0 < \gamma \in \R$, such that $\mD = \gamma\mW$ satisfies \eqref{eq:cond-dist-2}.
    
      \subsection{DCGD with constant stepsize}

      In this section we describe the convergence result for DCGD from \cite{khaled2020better}. 
      We assume that the component functions $f_i$ satisfy \Cref{ass:lower-bounded-fi} with 
      $\mL_i = L_i \mI_{d}$ and $f$ satisfies Assumption \ref{ass:finf} and \ref{ass:matrix_L} with 
      $\mL = L \mI_d$.
      \cite{khaled2020better} proposed a unified analysis for non-convex optimization algorithms based on a generic upper bound on the second moment of the gradient estimator $g(x^k)$: 
    \begin{equation}\label{eq:abc}
        \Exp{\norm{g(x^k)}^2} \leq 2A \brr{f(x^k) - f^{\inf}} + B \norm{\nabla f(x^k)}^2 + C,
      \end{equation}
      In our case the gradient estimator is defined as follows
      \begin{equation}
        g_{\sf DCGD}(x^k) = \frac{1}{n}\sum_{i=1}^n \mS^k_i \nabla f_i(x^k).
      \end{equation}
      Here each $\mS^k_i$ is the sketch matrix on the $i$-th client at the $k$-th iteration.
      One may check that $g_{\sf DCGD}$ satisfies \eqref{eq:abc} with the following constants:
      \begin{equation}
        A = \frac{\omega L_{\max}}{n}, \quad B = 1, \quad C =  \frac{2\omega L_{\max}}{n}\Delta^{\inf}.
      \end{equation}
      The constant $L_{\max}$ is defined as the maximum of all $L_i$ and $\omega = \lambda_{\max}\brr{\Exp{\left(\mS_i^k\right)^{\top}\mS_i^k}} - 1.$
      Applying Corollary 1 from \cite{khaled2020better}, we deduce the following.
      If 
      \begin{equation}
        \gamma \leq \min\brc{\frac1L, \frac{\sqrt{n}}{\sqrt{\omega L L_{\max} K}}, 
        \frac{n\epsilon^2 }{4L L_{\max} \omega\Delta^{\inf}}} \text{ and } \gamma K \geq \frac{12\left(f(x^0) - f^{\inf}\right)}{\varepsilon^2},
      \end{equation}
      then 
      \begin{equation}
        \min_{k = 0,\ldots,K-1}\Exp{\norm{\nabla f(x^k)}^2} \leq \epsilon^2.
      \end{equation}

  \section{Proofs of technical lemmas}\label{sec:proofs-lemmas}

    \subsection{Proof of \Cref{lemma:convexity-func}}\label{sec:lemma:convexity-func}
    Let us pick any two matrices $\mD_1, \mD_2 \in \bbS_{++}^d$, scalar $\alpha$ satisfying $0 \leq \alpha \leq 1$ and show that the following inequality holds regardless of the choice of $\mR$,
    \begin{equation}
        \label{eq:convex-eq}
        f(\alpha\mD_1 + (1-\alpha)\mD_2) \leq \alpha f(\mD_1)+(1-\alpha) f(\mD_2). 
    \end{equation}
    For the LHS, we have 
    \begin{eqnarray*}
        && f(\alpha\mD_1 + (1-\alpha)\mD_2) \\
        &=& \tr(\mL^{\frac{1}{2}}\left(\alpha\mD_1 + (1-\alpha)\mD_2\right)\mR(\alpha\mD_1 + (1-\alpha)\mD_2)\mL^\frac{1}{2})\\
        &=& \alpha^2\tr(\mL^\frac{1}{2}\mD_1\mR\mD_1\mL^\frac{1}{2}) + (1-\alpha)^2\tr(\mL^\frac{1}{2}\mD_2\mR\mD_2\mL^\frac{1}{2}) \\
        &&\qquad + \alpha(1-\alpha)\tr(\mL^\frac{1}{2}\mD_1\mR\mD_2\mL^\frac{1}{2}) + \alpha(1-\alpha)\tr(\mL^\frac{1}{2}\mD_2\mR\mD_1\mL^\frac{1}{2}).
    \end{eqnarray*}
    and for the RHS, we have
    \begin{eqnarray*}
        \alpha f(\mD_1) + (1-\alpha)f(\mD_2) 
        = \alpha\tr(\mL^\frac{1}{2}\mD_1\mR\mD_1\mL^\frac{1}{2})  + (1-\alpha)\tr(\mL^\frac{1}{2}\mD_2\mR\mD_2\mL^\frac{1}{2}) .
    \end{eqnarray*}
    Thus \eqref{eq:convex-eq} can be simplified to the following inequality after rearranging terms
    \begin{eqnarray*}
        &&\alpha(1-\alpha)\tr(\mL^\frac{1}{2}\mD_1\mR\mD_2\mL^\frac{1}{2}) + \alpha(1-\alpha)\tr(\mL^\frac{1}{2}\mD_2\mR\mD_1\mL^\frac{1}{2}) \\
        &\leq& \alpha(1-\alpha)\tr(\mL^\frac{1}{2}\mD_1\mR\mD_1\mL^\frac{1}{2}) + \alpha(1-\alpha)\tr(\mL^\frac{1}{2}\mD_2\mR\mD_2\mL^\frac{1}{2}).
    \end{eqnarray*}
    This is equivalent to  
    \begin{equation*}
        \tr(\mL^\frac{1}{2}\mD_1\mR\mD_1\mL^\frac{1}{2}) + \tr(\mL^\frac{1}{2}\mD_2\mR\mD_2\mL^\frac{1}{2}) - \tr(\mL^\frac{1}{2}\mD_1\mR\mD_2\mL^\frac{1}{2}) - \tr(\mL^\frac{1}{2}\mD_2\mR\mD_1\mL^\frac{1}{2}) \geq 0.
    \end{equation*}
    To show that the above inequality holds, we do the following transformation for the LHS
    \begin{eqnarray*}
        && \tr(\mL^\frac{1}{2}\mD_1\mR\mD_1\mL^\frac{1}{2}) + \tr(\mL^\frac{1}{2}\mD_2\mR\mD_2\mL^\frac{1}{2}) - \tr(\mL^\frac{1}{2}\mD_1\mR\mD_2\mL^\frac{1}{2}) - \tr(\mL^\frac{1}{2}\mD_2\mR\mD_1\mL^\frac{1}{2}) \\
        &=& \tr(\mL^\frac{1}{2}\mD_1\mR(\mD_1-\mD_2)\mL^\frac{1}{2}) + \tr(\mL^\frac{1}{2}\mD_2\mR(\mD_2-\mD_1)\mL^\frac{1}{2}) \\
        &=& \tr(\mL^\frac{1}{2}(\mD_1-\mD_2)\mR(\mD_1-\mD_2)\mL^\frac{1}{2}).
    \end{eqnarray*}
    Since $\mR \in \bbS^d_{+}$ and $\mD_1 - \mD_2, \mL$ are symmetric, for any vector $u \in \R^d$
    \begin{equation}
        \label{eq:pos-prod-mat}
        u^\top \mL^\frac{1}{2}(\mD_1-\mD_2)\mR(\mD_1-\mD_2)\mL^\frac{1}{2} u
        = \brr{(\mD_1-\mD_2)\mL^\frac{1}{2} u}^\top \mR\left((\mD_1-\mD_2)\mL^\frac{1}{2} u\right) \geq 0.
    \end{equation}
    Thus, $\mL^\frac{1}{2}(\mD_1-\mD_2)\mR(\mD_1-\mD_2)\mL^\frac{1}{2} \in \bbS_{+}^d$, which yields the positivity of its trace.
    Therefore, \eqref{eq:convex-eq} holds, thus $f(D)$ is a convex function. This concludes the proof.

    \subsection{Proof of \Cref{lemma:var-decomp}}
        We have 
        \begin{eqnarray*}
            \Exp{\norm{x - \Exp{x}}_{\mM}^2} &=& \Exp{\inner{x - \Exp{x}}{\mM\left(x - \Exp{x}\right)}} \\
            &=&\Exp{\left(x-\Exp{x}\right)^{\top}\mM\left(x - \Exp{x}\right)} \\
            &=&\Exp{x^{\top}\mM x - \Exp{x}^{\top}\mM x - x^{\top}\mM\Exp{x} + \Exp{x}^{\top}\mM\Exp{x}} \\
            &=&\Exp{x^{\top}\mM x} - 2\Exp{x}^{\top}\mM\Exp{x} + \Exp{x}^{\top}\mM\Exp{x} \\
            &=&\Exp{x^{\top}\mM x} - \Exp{x}^{\top}\mM\Exp{x} \\
            &=& \Exp{\norm{x}_{\mM}^2} - \norm{\Exp{x}}_{\mM}^2, \\
        \end{eqnarray*}
        which concludes the proof.

    \subsection{Proof of \Cref{lemma:var-sep}}
    \begin{proof}
        We have \
        \begin{eqnarray*}
            \Exp{\norm{\frac{1}{n}\sum_{i=1}^{n}a_i}_{\mM}^2} &=& \frac{1}{n^2}\sum_{i=1}^{n}\Exp{\inner{a_i}{\mM a_i}} + \frac{1}{n^2}\sum_{i\neq j}\Exp{\inner{a_i}{\mM a_j}} \\
            &=& \frac{1}{n^2}\sum_{i=1}^{n}\Exp{\norm{a_i}_{\mM}^2} + \frac{1}{n}\sum_{i\neq j}\inner{\Exp{a_i}}{\mM\Exp{a_j}} \\
            &=& \frac{1}{n^2}\sum_{i=1}^{n}\Exp{\norm{a_i}_{\mM}^2}.
        \end{eqnarray*}
        This concludes the proof.
    \end{proof}

    \subsection{Proof of \Cref{lemma:property-of-sketch}}

        Notice that 
        \begin{equation*}
            \Exp{\mS x} = \Exp{\mS}x = x.
        \end{equation*}
        We start with variance decomposition in the matrix norm, 
        \begin{eqnarray*}
            \Exp{\norm{\mS x - x}^2_{\mM}} &\overset{\eqref{eq:d-norm-var-decomp}}{=}& \Exp{\norm{\mS x}^2_{\mM}} - \norm{x}_{\mM}^2 \\
            &=&\Exp{\inner{\mS x}{\mM\mS x}} - \inner{x}{\mM x} \\
            &=& \inner{x}{\Exp{\mS\mM\mS}x} - \inner{x}{\mM x} \\
            &=& \inner{x}{\left(\Exp{\mS\mM\mS} - \mM\right)x}\\
            &=& \norm{x}_{\Exp{\mS\mM\mS} - \mM}^2.
        \end{eqnarray*}
        This concludes the proof.

    \subsection{Proof of \Cref{lemma:imp-smt-upd}}

        We follow the definition of $\mL$ matrix smoothness of function $f$, that for any $x^+, x \in \R^d$, we have 
        \begin{equation*}
            f(x^+) \leq f(x) + \inner{\nabla f(x)}{x^+ - x} + \frac{1}{2}\inner{x^+ - x}{\mL(x^+ - x)}.
        \end{equation*}
        We plug in $x^+ = x - \mL^{-1}\nabla f(x)$, and get 
        \begin{equation*}
            f^{\inf} \leq f(x^+) \leq f(x) - \inner{\nabla f(x)}{\mL^{-1}\nabla f(x)} + \frac{1}{2}\inner{\nabla f(x)}{\mL^{-1}\nabla f(x)}.
        \end{equation*}
        Rearranging terms we get 
        \begin{equation}
            \norm{\nabla f(x)}_{\mL^{-1}}^2 \leq 2\left(f(x) - f^{\inf}\right),
        \end{equation}
        which completes the proof.

    \subsection{Proof of \Cref{lemma:5}}
        \begin{eqnarray*}
            \Exp{\norm{\mS - x}^2_{\mD\mL\mD}} &=& \Exp{\inner{(\mS-\mI_d)x}{\mD\mL\mD(\mS - \mI_d)x}} \\
            &=& \Exp{x^{\top}(\mS-\mI_d)\mD\mL\mD(\mS-\mI_d)x} \\
            &=& x^{\top}\Exp{(\mS-\mI_d)\mD\mL\mD(\mS-\mI_d)}x \\
            &=& x^{\top}\mA^{-\frac{1}{2}}\left(\mA^\frac{1}{2}\Exp{(\mS-\mI_d)\mD\mL\mD(\mS-\mI_d)}\mA^\frac{1}{2}\right)\mA^{-\frac{1}{2}}x \\
            &\leq& \lambda_{\max}\left(\mA^\frac{1}{2}\Exp{(\mS-\mI_d)\mD\mL\mD(\mS-\mI_d)}\mA^\frac{1}{2}\right)\norm{\mA^{-\frac{1}{2}}x}^2 \\
            &=&\lambda_{\max}\left(\mA^\frac{1}{2}\Exp{(\mS-\mI_d)\mD\mL\mD(\mS-\mI_d)}\mA^\frac{1}{2}\right)\norm{x}^2_{\mA^{-1}}.\\
        \end{eqnarray*}
        This completes the proof.

    \subsection{Proof of \Cref{lemma:7}}

        \begin{eqnarray*}
            \Exp{\norm{\mT^k_i\mD x - \mD x}^2_{\mL}} &=& \Exp{\inner{(\mT^k_i-\mI_d)\mD x}{\mL(\mT^k_i - \mI_d)\mD x}} \\
            &=& \Exp{x^{\top}\mD(\mT^k_i-\mI_d)\mL(\mT^k_i-\mI_d)\mD x} \\
            &=& x^{\top}\mD\Exp{(\mT^k_i-\mI_d)\mL(\mT^k_i-\mI_d)}\mD x \\
            &=& x^{\top}\mL_i^{-\frac{1}{2}}\left(\mL_i^\frac{1}{2}\mD\Exp{(\mT^k_i-\mI_d)\mL(\mT^k_i-\mI_d)}\mD\mL_i^\frac{1}{2}\right)\mL_i^{-\frac{1}{2}}x \\
            &\leq& \lambda_{\max}\left(\mL_i^\frac{1}{2}\mD\Exp{(\mT^k_i-\mI_d)\mL(\mT^k_i-\mI_d)}\mD\mL_i^\frac{1}{2}\right)\norm{\mL_i^{-\frac{1}{2}}x}^2 \\
            &=&\lambda_{\max}\left(\mL_i^\frac{1}{2}\mD\Exp{(\mT^k_i-\mI_d)\mL(\mT^k_i-\mI_d)}\mD\mL_i^\frac{1}{2}\right)\norm{x}^2_{\mL_i^{-1}}.\\
        \end{eqnarray*}
        This completes the proof.

  \section{Experiments}

  In this section, we describe the settings and results of numerical experiments to demonstrate the effectiveness of our method. 
  We perform several experiments under single node case and distributed case. The code is available at \url{https://anonymous.4open.science/r/detCGD_Code-A87D/}.

  \subsection{Single node case}
  For the single node case, we study the logistic regression problem with a non-convex regularizer. The objective is given as 
  \begin{eqnarray*}
    f(x) = \frac{1}{n}\sum_{i=1}^{n}\log\left(1 + e^{-b_i\cdot\langle a_i, x\rangle}\right) + \lambda\cdot\sum_{j=1}^{d}\frac{x_j^2}{1 + x^2_j},
  \end{eqnarray*}
  where $x \in \R^d$ is the model, $(a_i, b_i) \in \R^d \times \left\{-1, +1\right\}$ is one data point in the dataset whose size is $n$. 
  The constant $\lambda > 0$ is a tunable hyperparameter associated with the regularizer. 
  We conduct numerical experiments using several datasets from the LibSVM repository 
  \citep{chang2011libsvm}. 
  We estimate the smoothness matrix of function $f$ here as 
  \begin{eqnarray*}
    \mL = \frac{1}{n}\sum_{i=1}^{n}\frac{a_ia_i^{\top}}{4} + 2\lambda\cdot\mI_d.
  \end{eqnarray*}

  \subsubsection{Comparison to CGD with scalar stepsize, scalar smoothness constant}
  The purpose of the first experiment is to show that by using matrix stepsize, \ref{eq:alg1} and \ref{eq:alg2} will have better iteration and communication complexities compared to standard CGD. 
  We run a CGD with scalar stepsize $\gamma$ and a scalar smoothness constant $L = \lambda_{\max}(\mL)$) and CGD with scalar stepsize $\gamma \cdot \mI_d$ and smoothness matrix $\mL$. 
  We use standard CGD to refer to the CGD with scalar stepsize, scalar smoothness constant, and CGD-mat to refer to CGD with scalar stepsize, smoothness matrix in \Cref{fig:experiment-1-sub-1}, \ref{fig:experiment-2-sub-1}. The notation $G_{K, \mD}$ appears in the label of y axis is defined as 
  \begin{eqnarray}
    \label{eq:def-G-var}
    G_{K, \mD} := \frac{1}{K}\left(\sum_{k=0}^{K-1}\norm{\nabla f(x^k)}^2_{\frac{\mD}{\det(\mD)^{1/d}}}\right),
  \end{eqnarray}
  it is the average matrix norm of the gradient of $f$ over the first $K-1$ iterations in log scale. The weight matrix here has determinant $1$, and thus it is comparable to the standard Euclidean norm. The result is meaningful in this sense.

  \begin{figure}
	\centering
    \subfigure{
	\begin{minipage}[b]{0.98\textwidth}
		\includegraphics[width=0.32\textwidth]{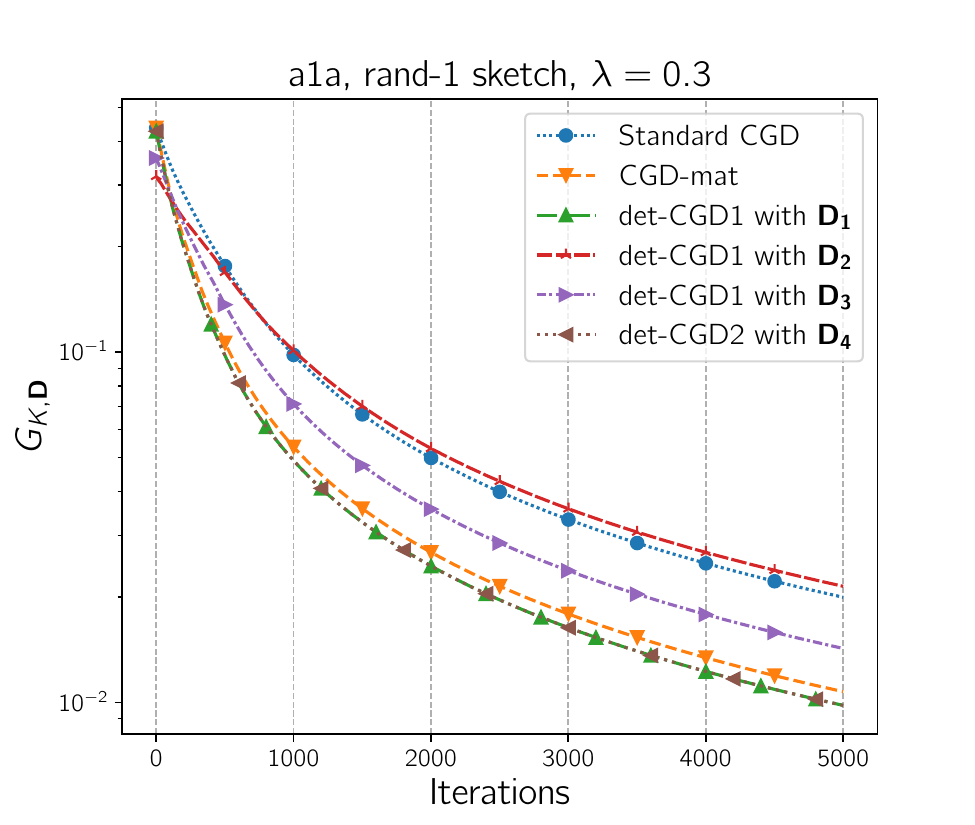} 
		\includegraphics[width=0.32\textwidth]{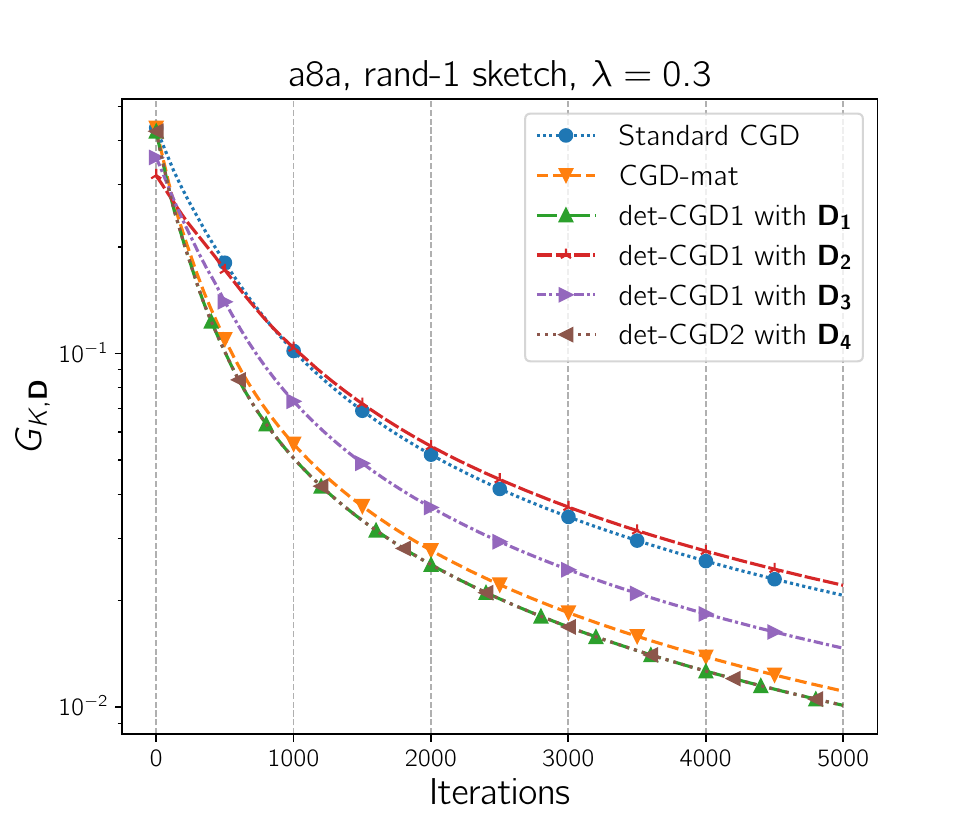}
        \includegraphics[width=0.32\textwidth]{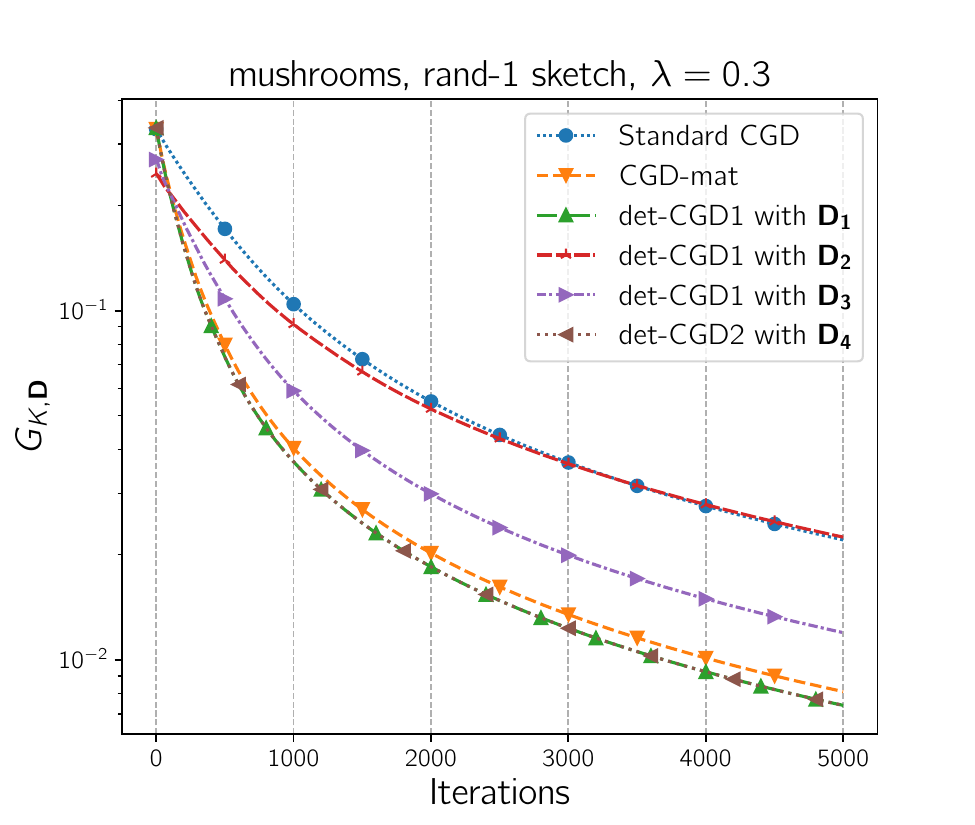}
	\end{minipage}
    }

    \subfigure{
	\begin{minipage}[b]{0.98\textwidth}
		\includegraphics[width=0.32\textwidth]{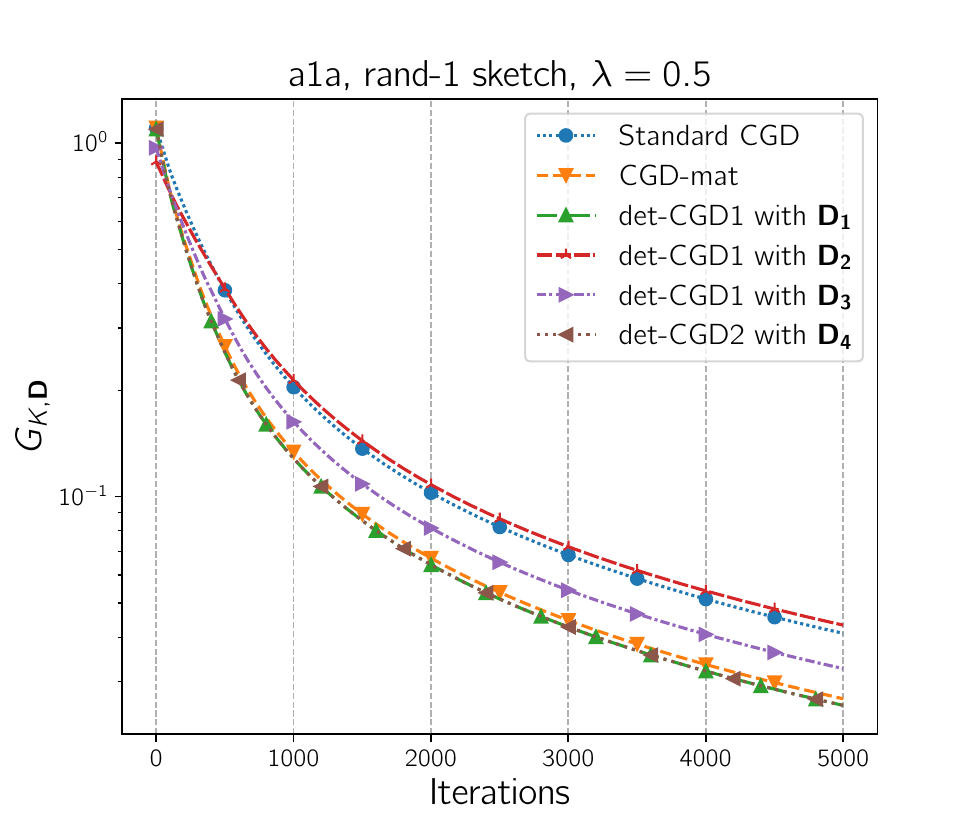} 
		\includegraphics[width=0.32\textwidth]{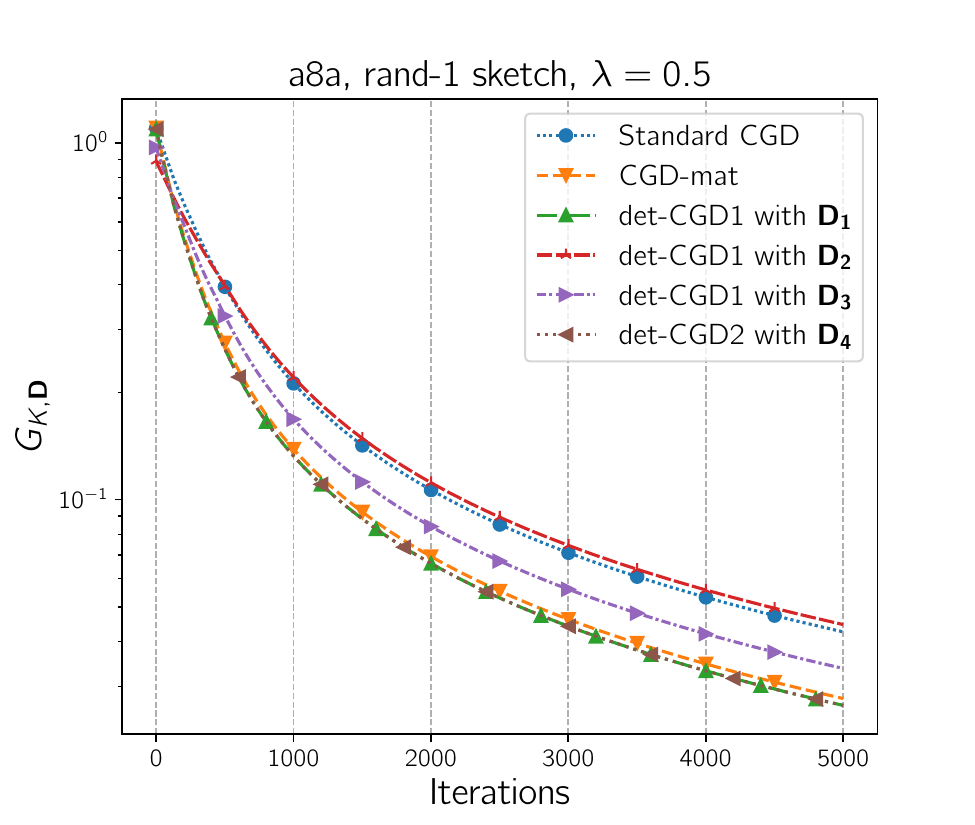}
        \includegraphics[width=0.32\textwidth]{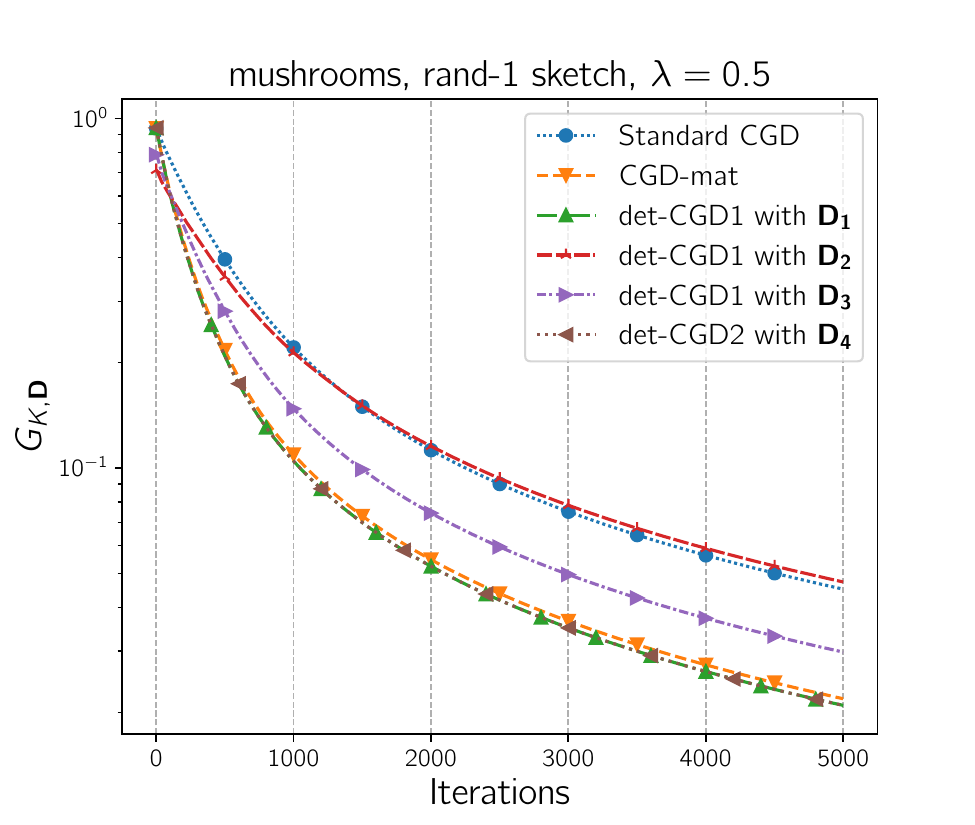}
	\end{minipage}
    }

	\caption{Comparison of standard CGD, CGD-mat, \ref{eq:alg1} with $\mD_1 = \gamma_1 \cdot \diag^{-1}(\mL)$, \ref{eq:alg1} with $\mD_2 = \gamma_2 \cdot \mL^{-1}$, \ref{eq:alg1} with $\mD_3 = \gamma_3 \cdot \mL^{-1/2}$ and \ref{eq:alg2} with $\mD_4 = \gamma_4 \cdot \diag^{-1}(\mL)$, where $\gamma_1, \gamma_2, \gamma_3$ are the optimal scaling factors for \ref{eq:alg1} in that case, $\mD_4$ is the optimal matrix stepsize for \ref{eq:alg2}. Rand-$1$ sketch is used in all the methods through out the experiments. The notation $G_{K, \mD}$ in the $y$-axis is defined in \eqref{eq:def-G-var}.}
	\label{fig:experiment-1-sub-1}
\end{figure}

The result presented in \Cref{fig:experiment-1-sub-1} suggests that compared to the standard CGD \citep{khirirat2018distributed}, CGD-mat performs better in terms of both iteration complexity and communication complexity. 
Furthermore,  \ref{eq:alg1} and \ref{eq:alg2} with the best diagonal matrix stepsizes outperform both CGD and CGD-mat which confirms our theory. 
The scaling factors $\gamma_1, \gamma_2, \gamma_3$ here for \ref{eq:alg1} are determined using \Cref{thm:block_diag} with $\nlay = 1$. The matrix stepsize for \ref{eq:alg2} is determined through \eqref{eq:opt-D-2}. \ref{eq:alg1} and \ref{eq:alg2} with diagonal matrix stepsizes perform very similarly in the experiment, this is expected since we are using rand-$1$ sketch, which means that the stepsize matrix and the sketch matrix are commutable since they are both diagonal. We also notice that \ref{eq:alg1} with $\mD_2 = \gamma_2 \cdot \mL^{-1}$ is always worse than $\mD_4 = \gamma_4 \cdot \diag^{-1}(\mL)$, this is also expected since we mentioned in \Cref{sec:worse} that the result row $5$ (corresponding to $\mD_2$) in \Cref{Table:comm-complex-single-node} is always worse than row $7$ (corresponding to $\mD_4$).

\subsubsection{Comparison of the two algorithms under the same stepsize}
The purpose of the second experiment is to compare the performance of \ref{eq:alg1} and \ref{eq:alg2} in terms of iteration complexity and communication complexity. 
We know the conditions for \ref{eq:alg1} and \ref{eq:alg2} to converge are given by \eqref{eq:ineq_D} and \eqref{eq:ineq_D-var} respectively.
As a result, we are able to obtain the optimal matrix stepsize for \ref{eq:alg2} if we are using rand-$\tau$ sparsification. It is given by 
\begin{eqnarray*}
    \mD_2^* = \frac{\tau}{d}\left(\frac{d - \tau}{d - 1}\diag(\mL) + \frac{\tau - 1}{d - 1}\mL\right)^{-1},
\end{eqnarray*}
according to \eqref{eq:opt-D-2}. The definition of $G_{K, \mD}$ is given in \eqref{eq:def-G-var}. 
Parameter $\tau$ here for random sparsification is set to be an the integer part $\{\frac{d}{4}, \frac{d}{2}, \frac{3d}{4}\}$, where $d$ is the dimension of the model.

\begin{figure}
	\centering
    \subfigure{
	\begin{minipage}[b]{0.98\textwidth}
		\includegraphics[width=0.32\textwidth]{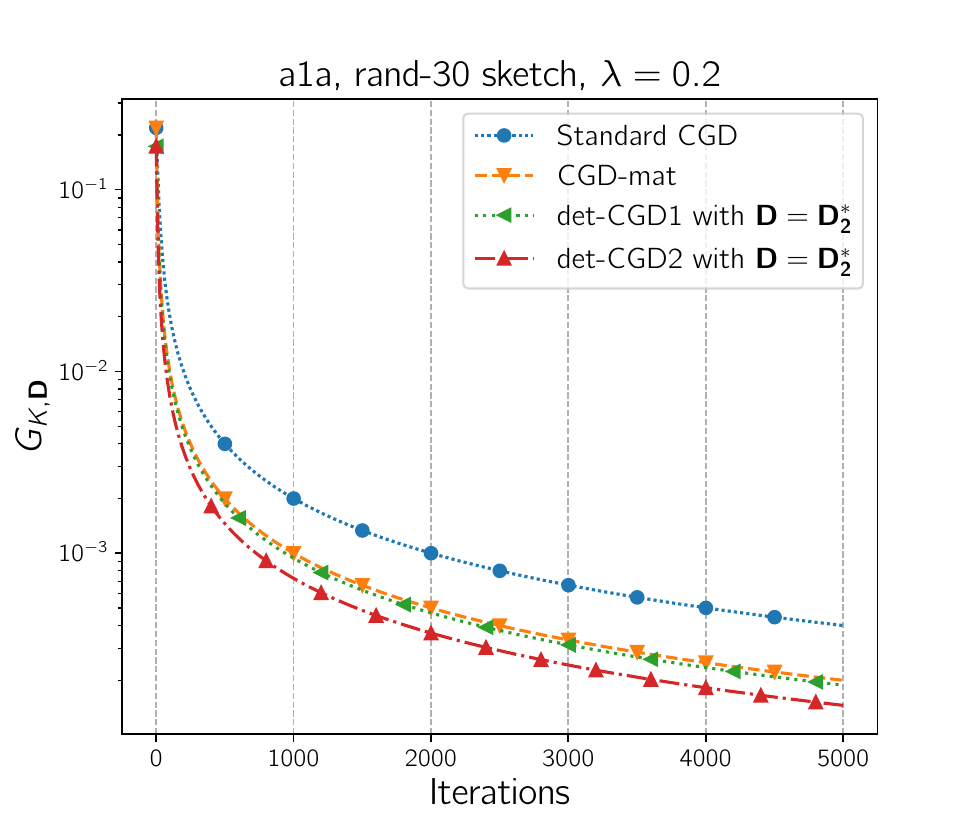} 
		\includegraphics[width=0.32\textwidth]{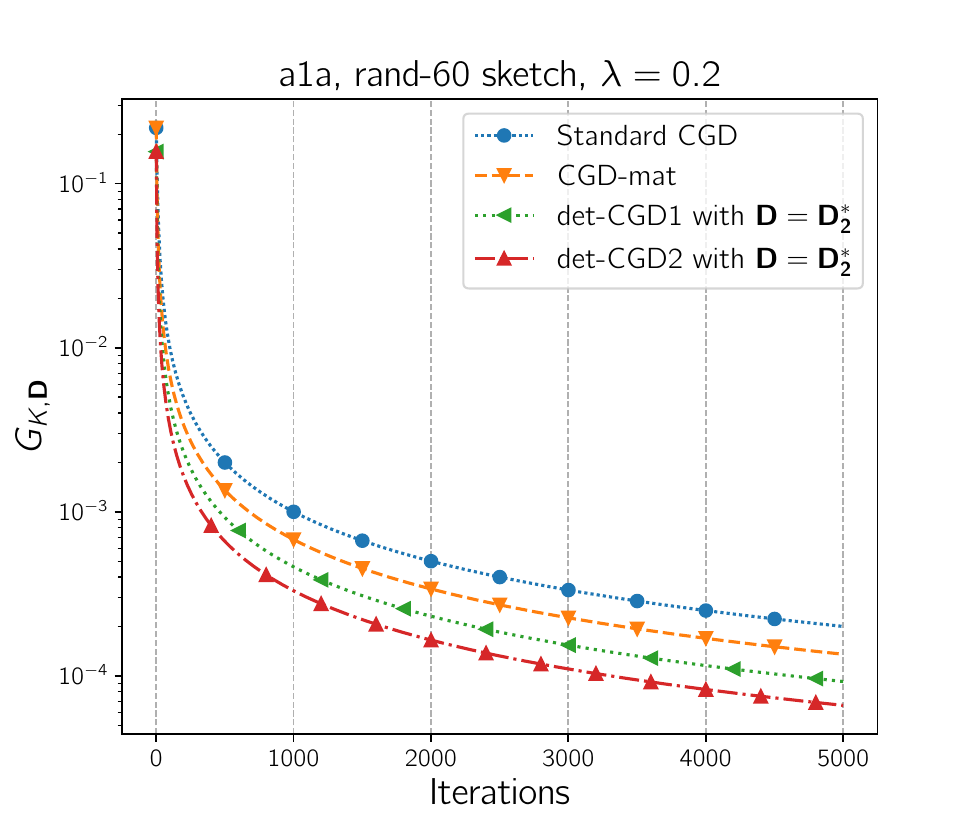}
        \includegraphics[width=0.32\textwidth]{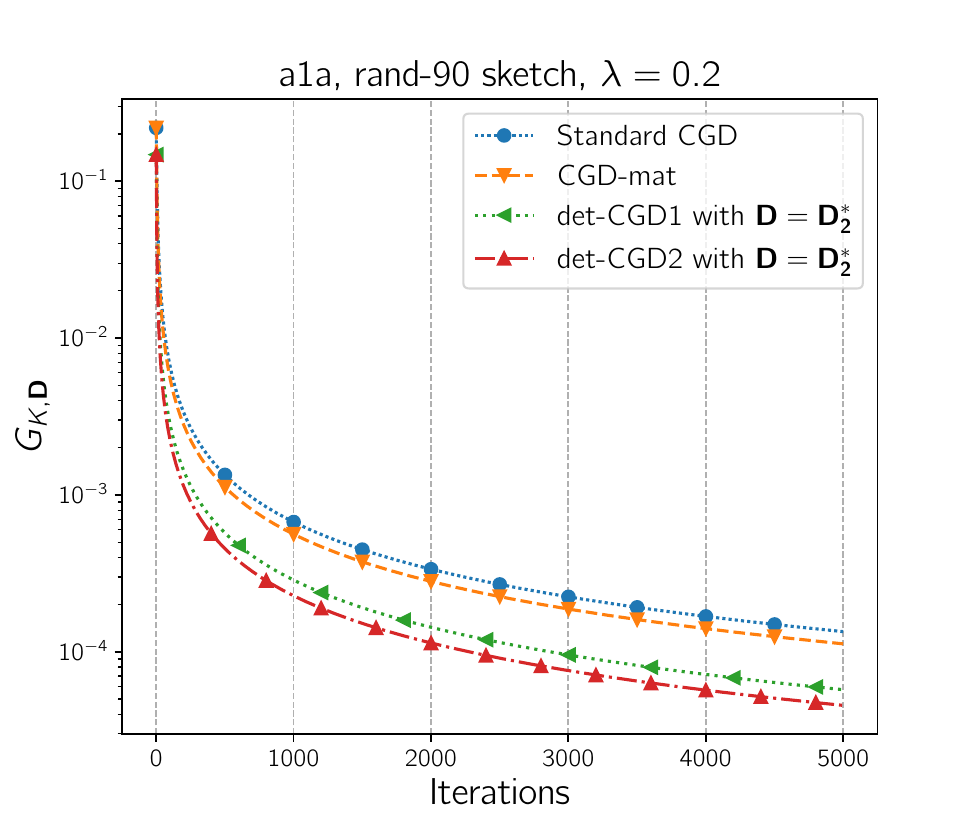}
	\end{minipage}
    }

    \subfigure{
	\begin{minipage}[b]{0.98\textwidth}
		\includegraphics[width=0.32\textwidth]{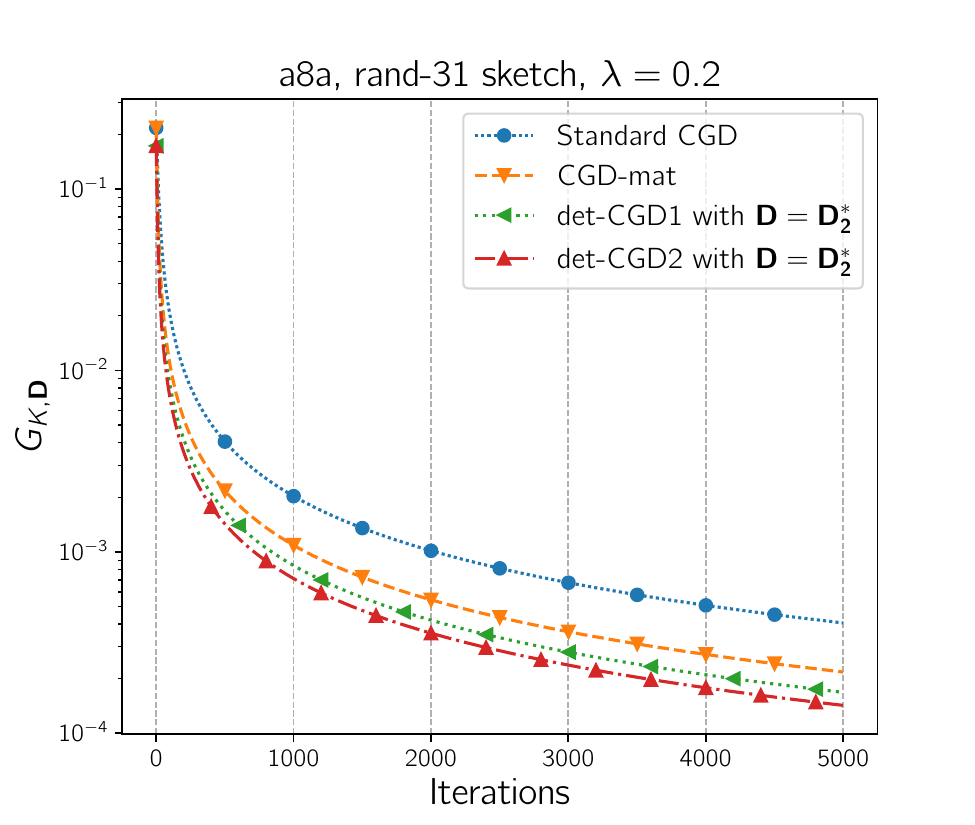} 
		\includegraphics[width=0.32\textwidth]{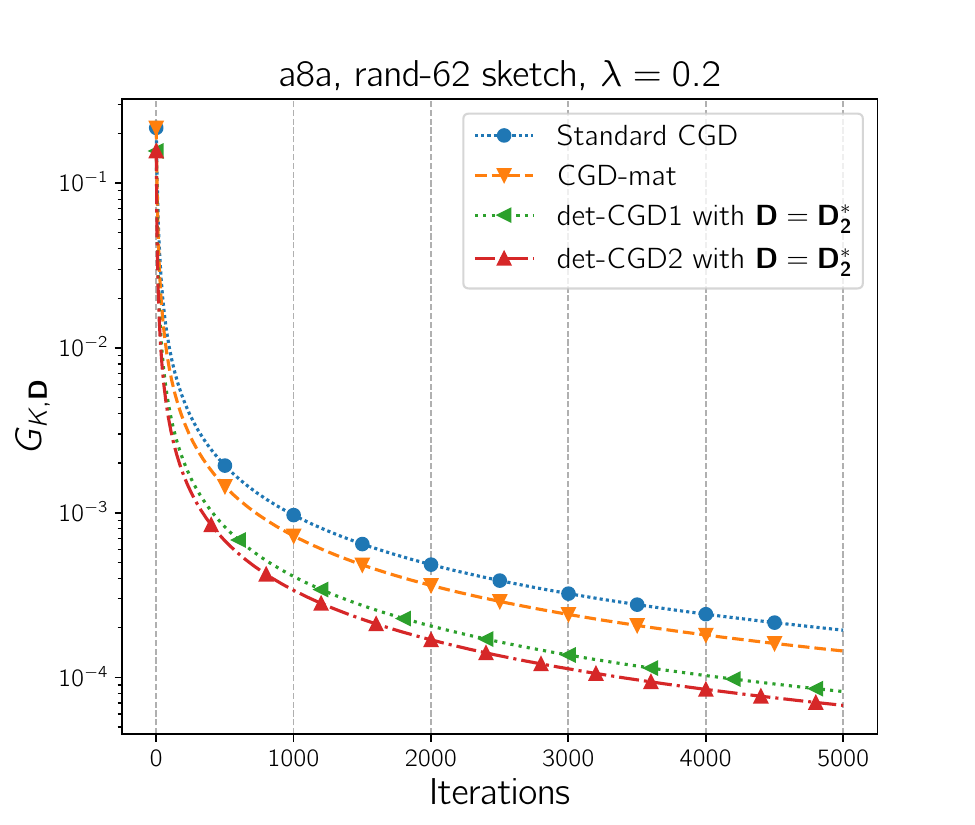}
        \includegraphics[width=0.32\textwidth]{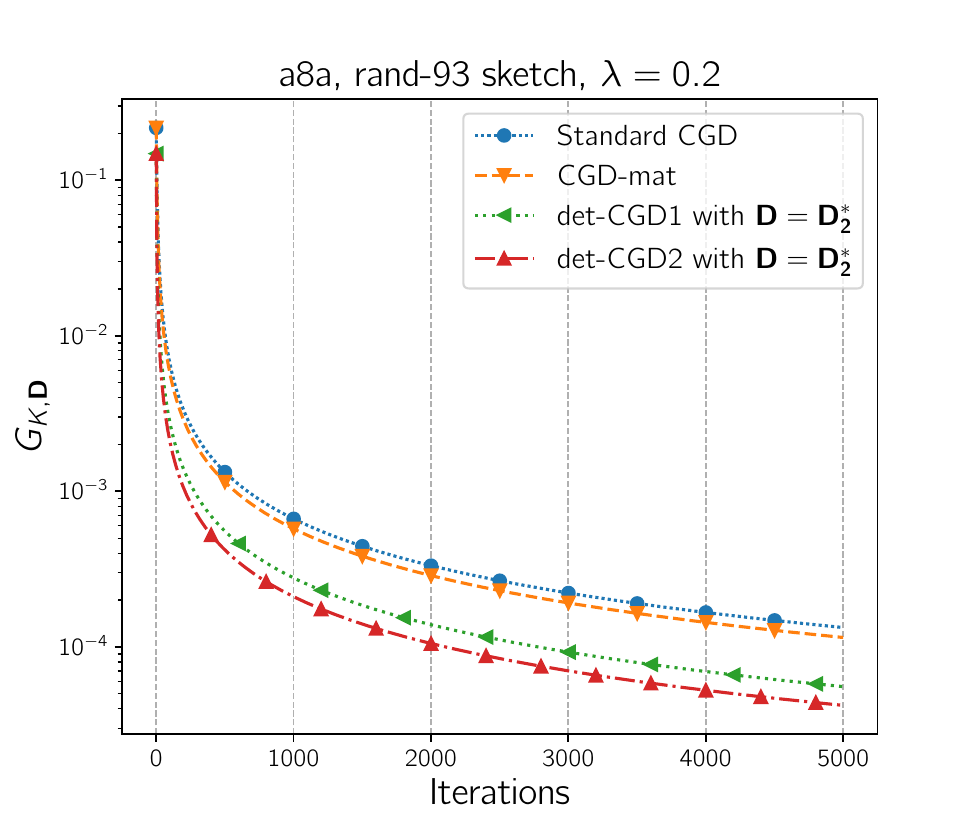}
	\end{minipage}
    }

    \subfigure{
	\begin{minipage}[b]{0.98\textwidth}
		\includegraphics[width=0.32\textwidth]{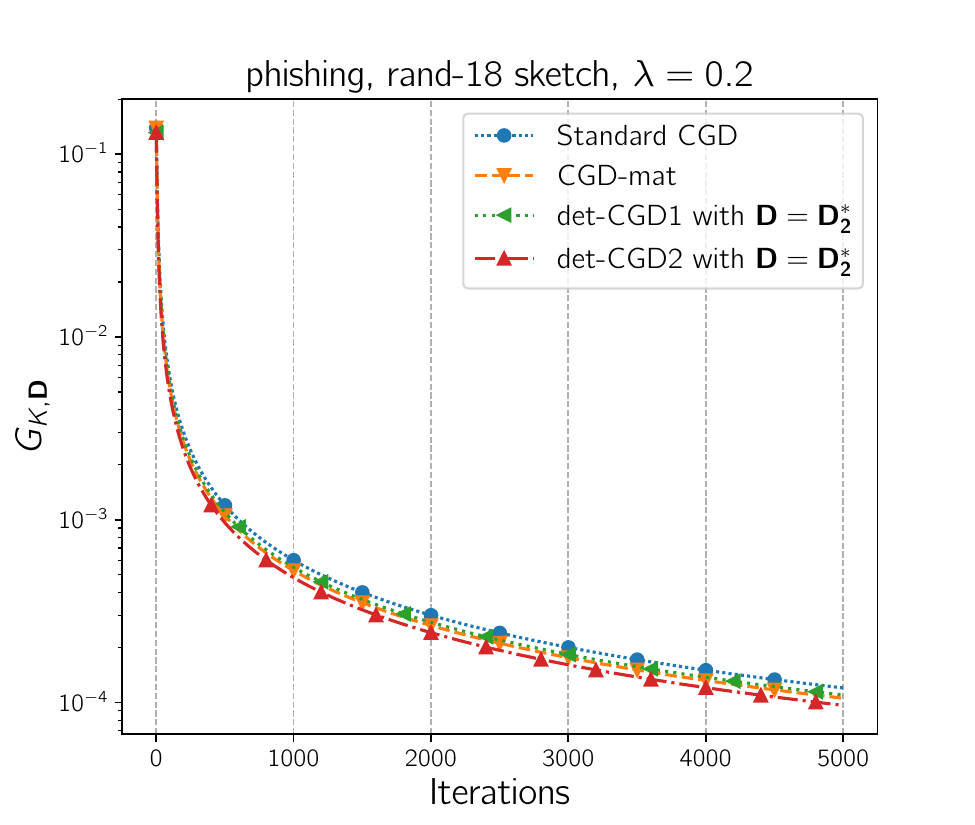} 
		\includegraphics[width=0.32\textwidth]{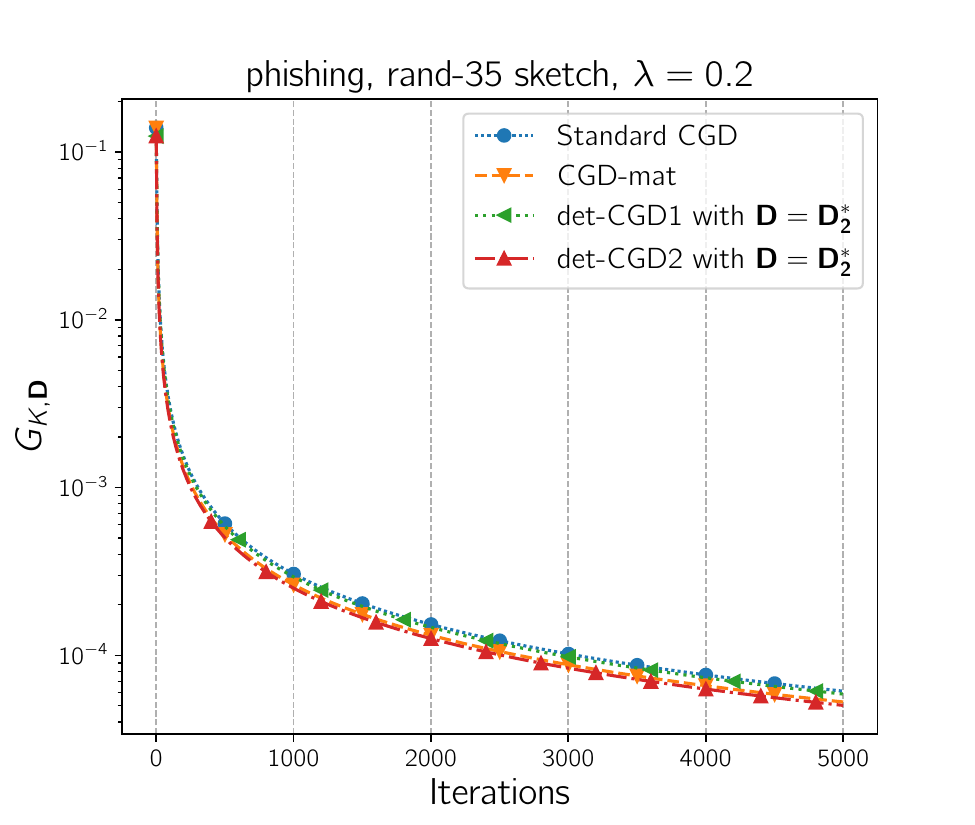}
        \includegraphics[width=0.32\textwidth]{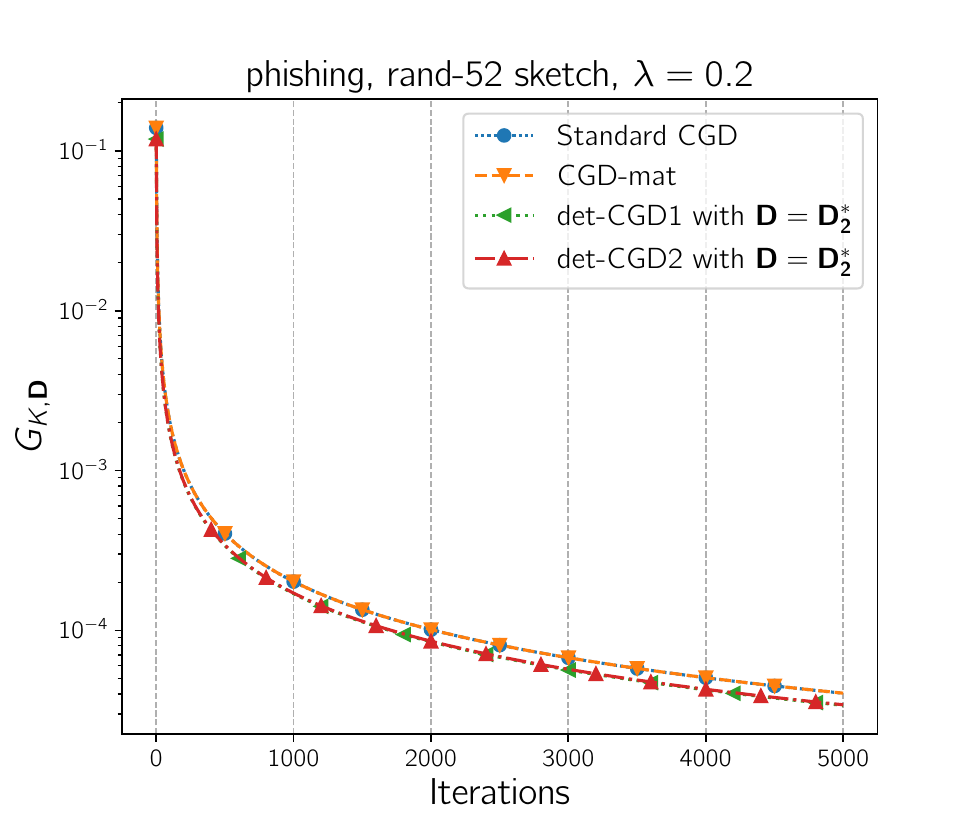}
	\end{minipage}
    }
    
    \subfigure{
	\begin{minipage}[b]{0.98\textwidth}
		\includegraphics[width=0.32\textwidth]{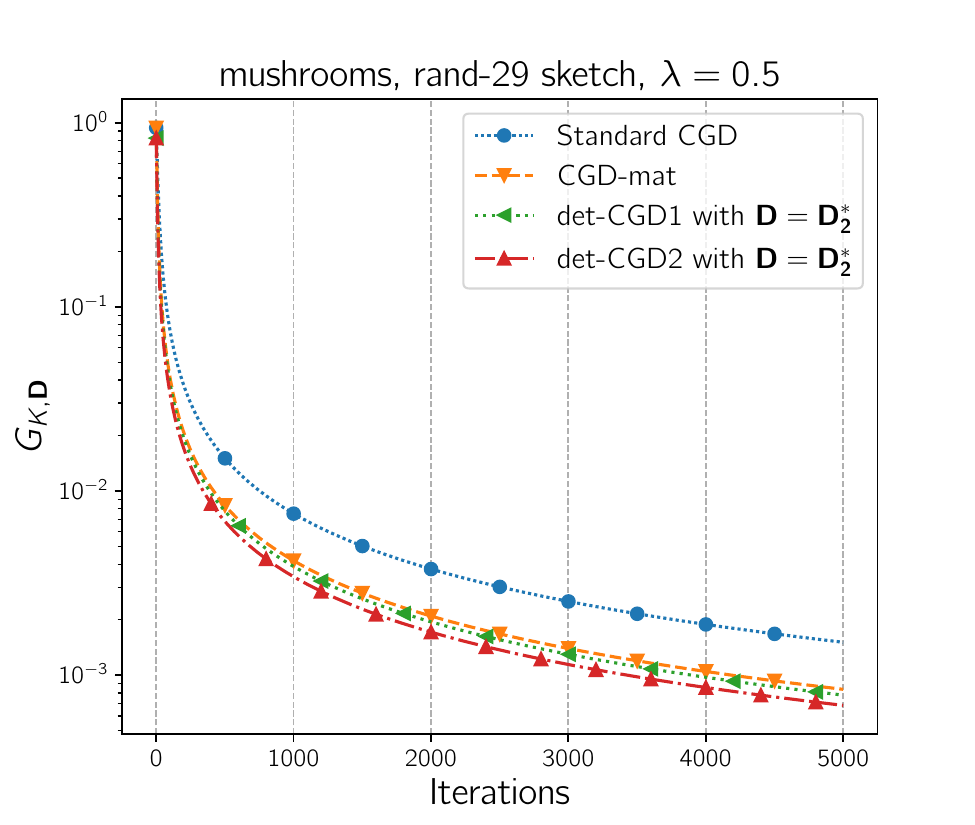} 
		\includegraphics[width=0.32\textwidth]{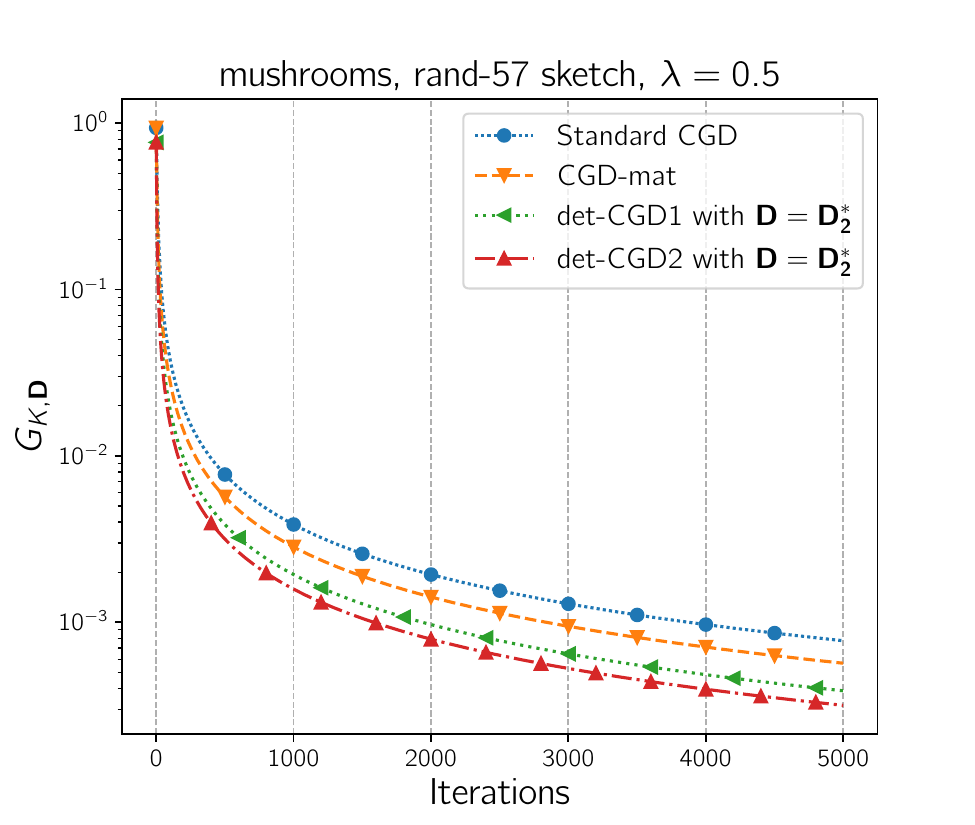}
        \includegraphics[width=0.32\textwidth]{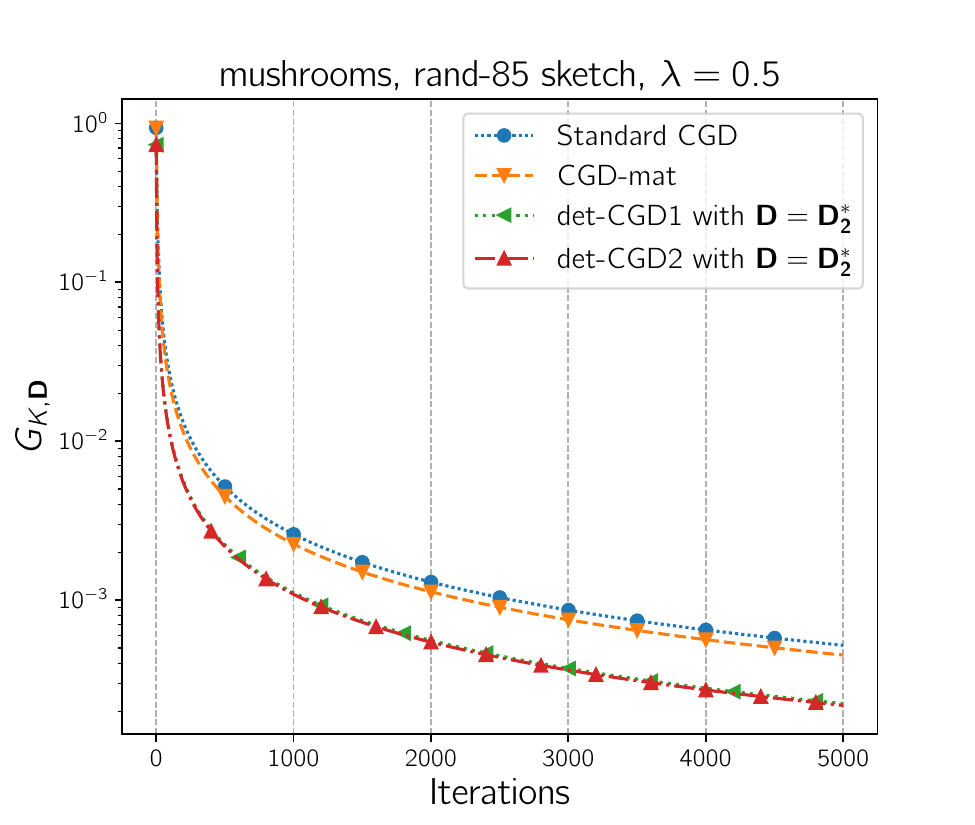}
	\end{minipage}
    }

    \caption{Comparison of standard CGD, CGD-mat \ref{eq:alg1} with stepsize $\mD = \mD_2^*$ and \ref{eq:alg2} with stepsize $\mD = \mD_2^*$, where $\mD_2^*$ is the optimal stepsize matrix for \ref{eq:alg2} and the optimal diagonal stepsize matrix for \ref{eq:alg1}. Rand-$\tau$ sketch is used in all the algorithms throughout the experiments. The notation $G_{K, \mD}$ in the $y$-axis is defined in \eqref{eq:def-G-var}.}
	\label{fig:experiment-2-sub-1}
\end{figure}

It can be observed from the result presented in \Cref{fig:experiment-2-sub-1}, that in almost all cases in this experiment, \ref{alg:dist-alg2} with $\mD = \mD_2^*$ outperforms the other methods. 
Compared to standard CGD and CGD with matrix stepsize, \ref{eq:alg1} and \ref{eq:alg2} are always better. 
This provides numerical evidence in support of our theory. 
In this case, the stepsize matrix is not diagonal for \ref{eq:alg1} and \ref{eq:alg2}, so we do not expect them to perform similarly. 
Notice that in dataset \verb+phishing+, the four algorithms behave very similarly, this is because the smoothness matrix $\mL$ here has a concentrated spectrum.

\subsection{Distributed case}
For the distributed case, we again use the logistic regression problem with a non-convex regularizer as our experiment setting. The objective is given similarly as 
\begin{eqnarray*}
    f(x) = \frac{1}{n}\sum_{i=1}^{n} f_i(x); \quad f_i(x) = \frac{1}{m_i}\sum_{j=1}^{m_i}\log\left(1 + e^{-b_{i, j}\cdot\langle a_{i, j}, x \rangle}\right) + \lambda \cdot \sum_{t=1}^{d}\frac{x_t^2}{1 + x_t^2},
\end{eqnarray*}
where $x \in \R^d$ is the model, $(a_{i,j}, b_{i, j}) \in \R^d \times \left\{-1, +1\right\}$ is one data point in the dataset of client $i$ whose size is $m_i$. $\lambda > 0$ is a constant associated with the regularizer. For each dataset used in the distributed setting, we randomly reshuffled the dataset before splitting it equally to each client. We estimate the smoothness matrices of function $f$ and each individual function $f_i$ here as 
\begin{eqnarray*}
    \mL_i &=& \frac{1}{m_i}\sum_{i=1}^{m_i}\frac{a_ia_i^{\top}}{4} + 2\lambda \cdot \mI_d; \\
    \mL &=& \frac{1}{n}\sum_{i=1}^{n} \mL_i. 
\end{eqnarray*}
The value of $\Delta^{\inf}$ here is determined in the following way, we first perform gradient descent on $f$ and record the minimum value in the entire run, $f^{\inf}$, as the estimate of its global minimum, then we do the same procedure for each $f_i$ to obtain the estimate of its global minimum $f^{\inf}_i$. After that we estimate $\Delta^{\inf}$ using its definition.

\subsubsection{Comparison to standard DCGD in the distributed case}
To ease the reading of this section we use D-\ref{eq:alg1} (resp. D-\ref{eq:alg2}) to refer to  \Cref{alg:dist-alg1} (resp. \Cref{alg:dist-alg2}).
This experiment is designed to show that D-\ref{eq:alg1} and D-\ref{eq:alg2} will have better iteration  and communication complexity compared to standard DCGD \citep{khirirat2018distributed} and DCGD with scalar stepsize, smoothness matrix. 
We will use the standard DCGD here to refer to DCGD with a scalar stepsize and a scalar smoothness constant, and DCGD-mat to refer to the DCGD with a scalar stepsize with smoothness.
The Rand-$1$ sparsifier is used in all the algorithms throughout the experiment. The error level is fixed as $\varepsilon^2 = 0.0001$, the conditions for the standard DCGD to converge can be deduced using Proposition 4 in \cite{khaled2020better}, we use the largest possible scalar stepsize here for standard DCGD. 
The optimal scalar stepsize for DCGD-mat, optimal diagonal matrix stepsize $\mD_1$ for D-\ref{eq:alg1} and $\mD_2$ for D-\ref{eq:alg2} can be determined using \Cref{cor:dist-cond-conv}.

\begin{figure}
	\centering
    \subfigure{
	\begin{minipage}[t]{0.98\textwidth}
		\includegraphics[width=0.32\textwidth]{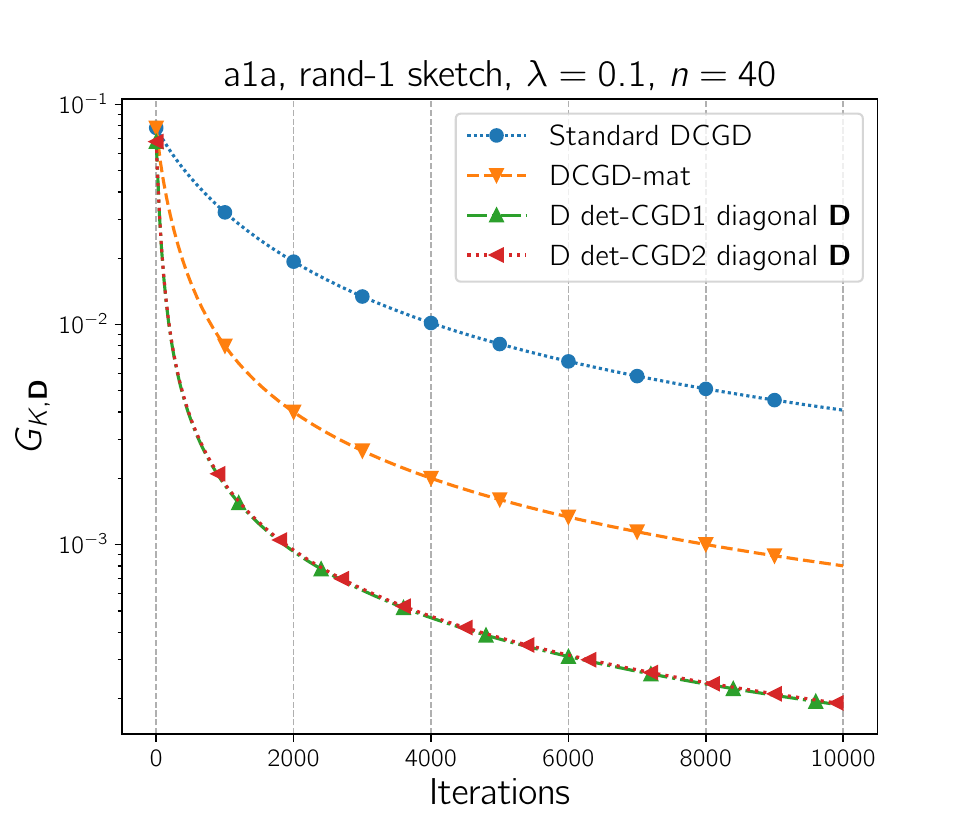} 
		\includegraphics[width=0.32\textwidth]{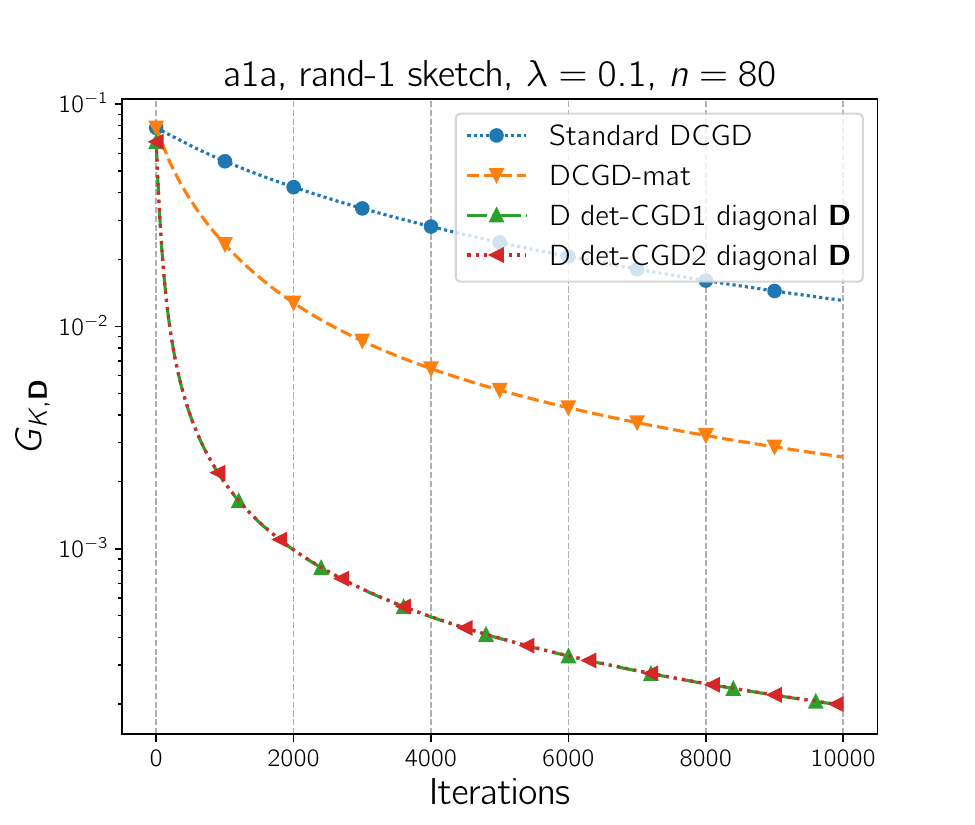}
        \includegraphics[width=0.32\textwidth]{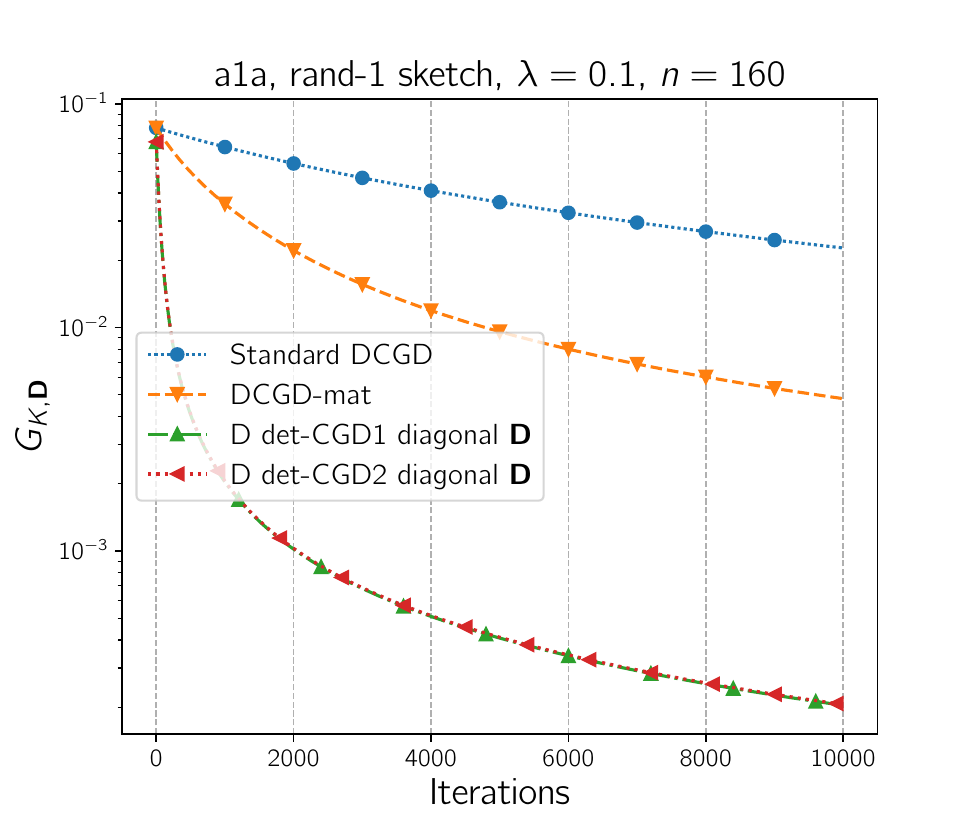}
	\end{minipage}
    }

    \subfigure{
	\begin{minipage}[t]{0.98\textwidth}
		\includegraphics[width=0.32\textwidth]{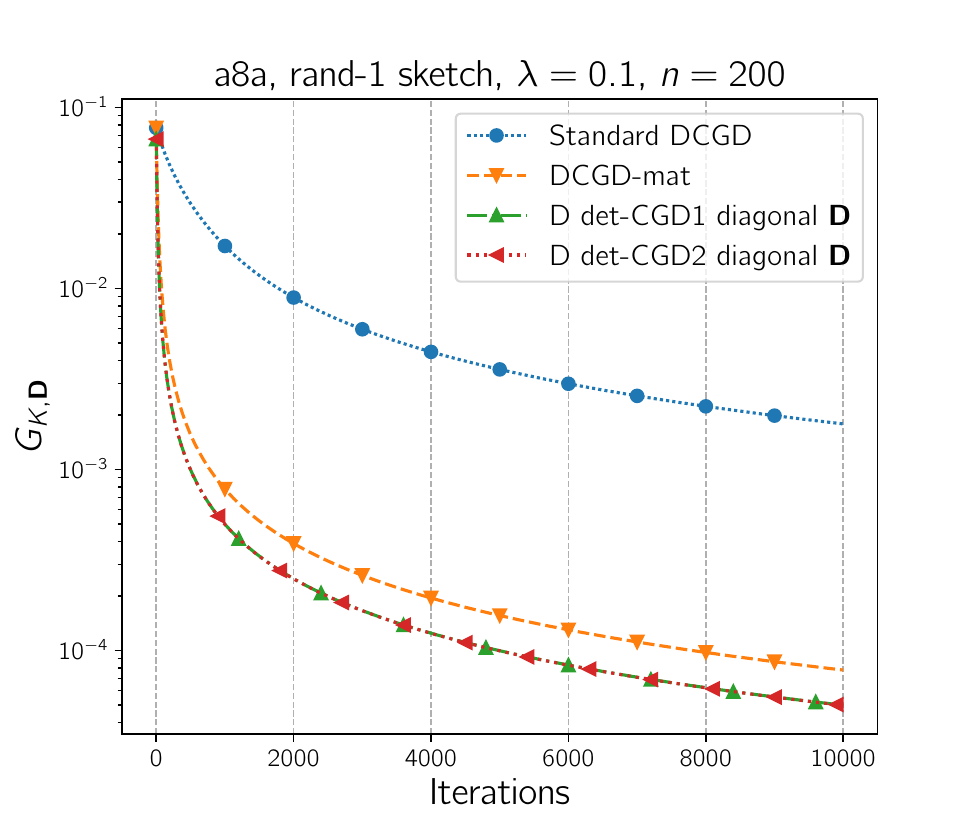} 
		\includegraphics[width=0.32\textwidth]{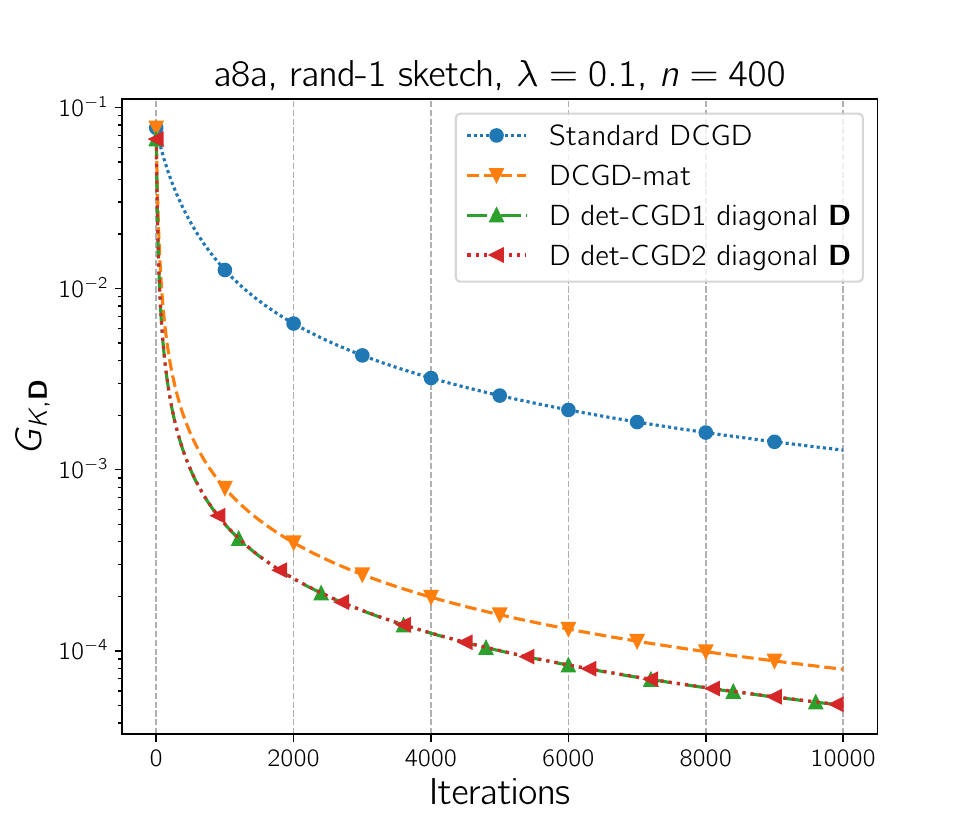}
        \includegraphics[width=0.32\textwidth]{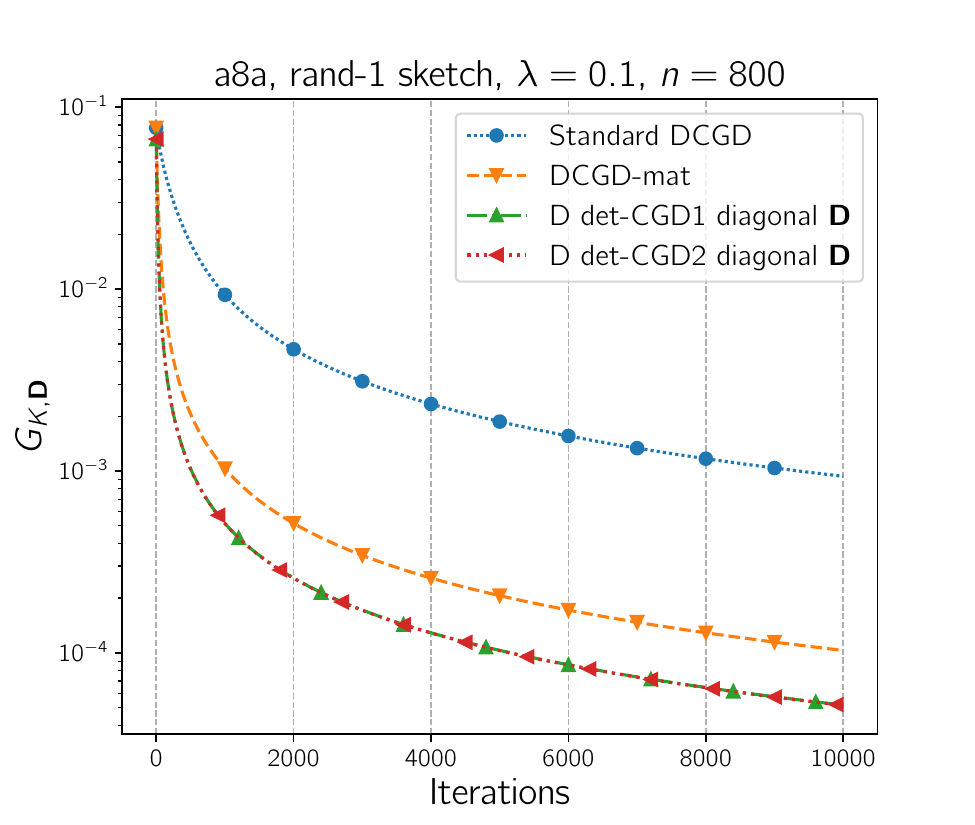}
	\end{minipage}
    }

    \subfigure{
	\begin{minipage}[t]{0.98\textwidth}
		\includegraphics[width=0.32\textwidth]{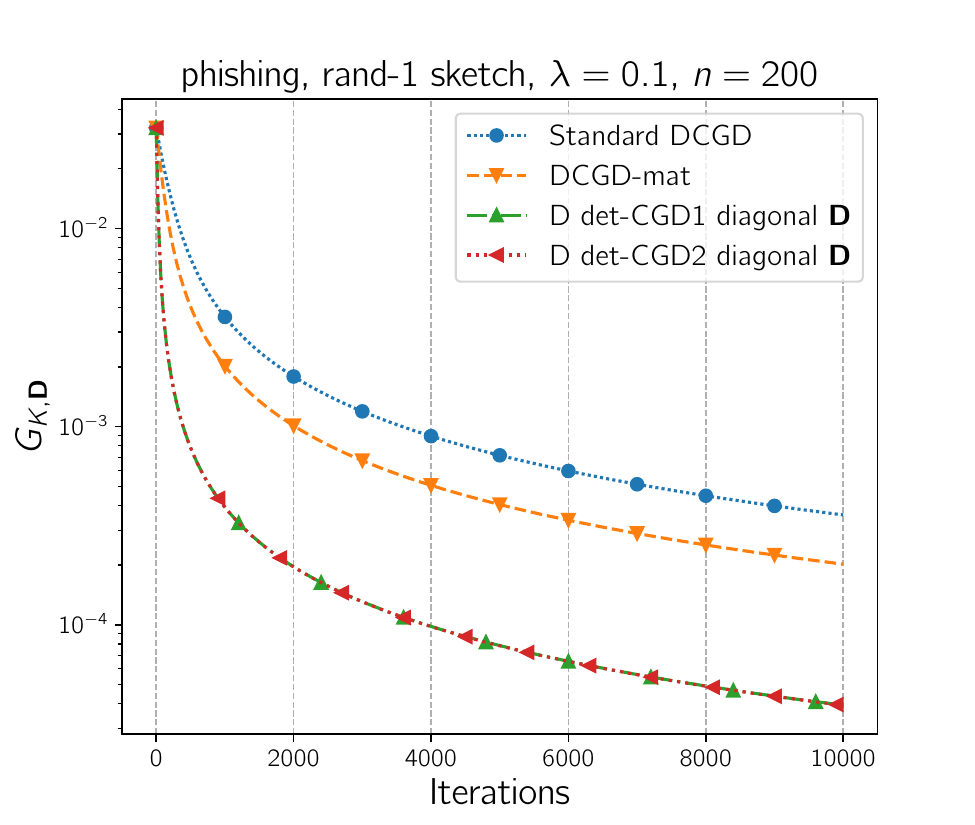} 
		\includegraphics[width=0.32\textwidth]{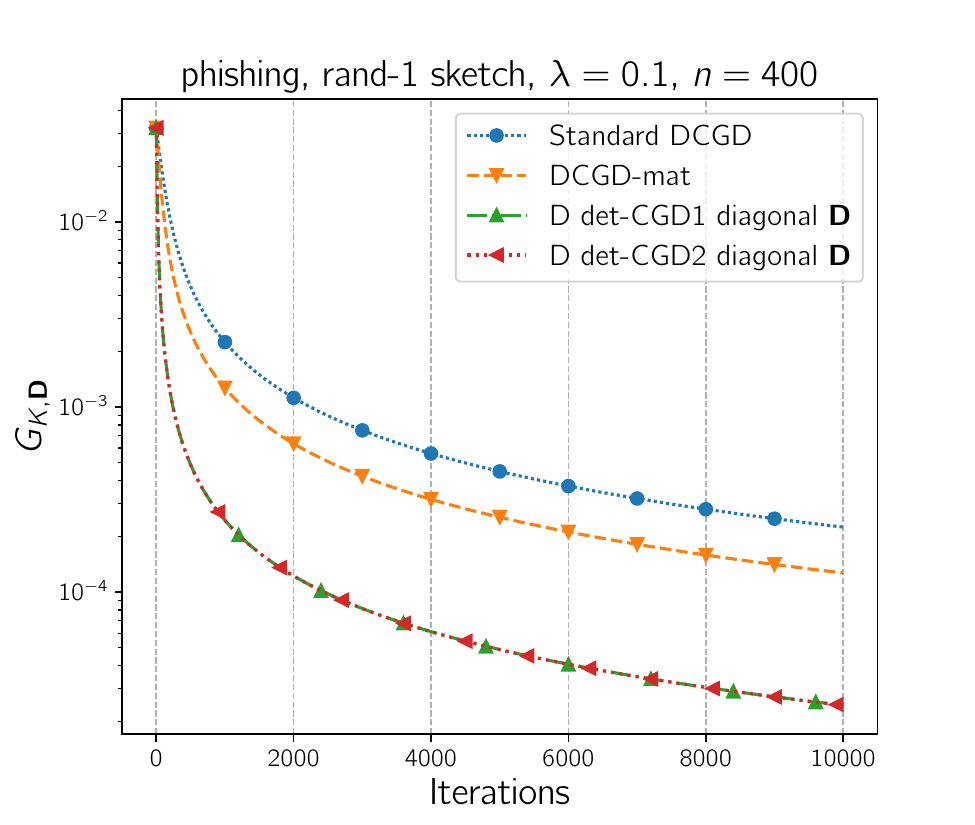}
        \includegraphics[width=0.32\textwidth]{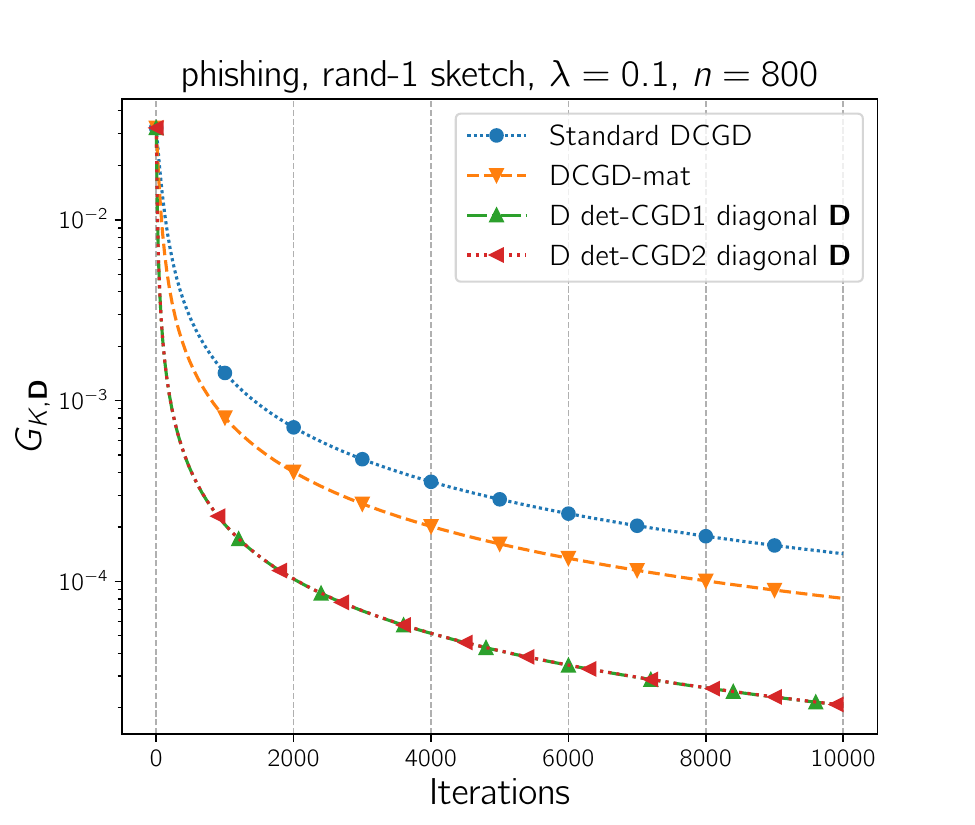}
	\end{minipage}
    }
    
    \caption{Comparison of standard DCGD, DCGD-mat, D-\ref{eq:alg1} with matrix stepsize $\mD_1$ and D-\ref{eq:alg2} with matrix stepsize $\mD_2$, where $\mD_1, \mD_2$ are the optimal diagonal matrix stepsizes for D-\ref{eq:alg1} and D-\ref{eq:alg2} respectively. Rand-$1$ sketch is used in all the algorithms throughout the experiment. The notation $G_{K, \mD}$ in the $y$-axis is defined in \eqref{eq:def-G-var}.}
    \label{fig:experiment-4-sub-1}
\end{figure}

From the result of \Cref{fig:experiment-4-sub-1}, we are able to see that both D-\ref{eq:alg1} and D-\ref{eq:alg2} outperform standard DCGD and DCGD-mat in terms of iteration complexity and communication complexity, which confirms our theory. Notice that D-\ref{eq:alg1}, D-\ref{eq:alg2} are expected to perform very similarly because the stepsize matrix and sketches are diagonal which means that they are commutable. We also plot the corresponding standard Euclidean norm of iterates of D-\ref{eq:alg1} and D-\ref{eq:alg2} in \Cref{fig:experiment-4-sub-2}, the $E_K$ here appears in the $y$-axis is defined as,
\begin{equation}
    \label{eq:standard-euc}
    E_K := \frac{1}{K}\sum_{k=0}^{K-1}\norm{\nabla f(x^k)}^2.
\end{equation}

\begin{figure}
	\centering
    \subfigure{
	\begin{minipage}[t]{0.98\textwidth}
		\includegraphics[width=0.32\textwidth]{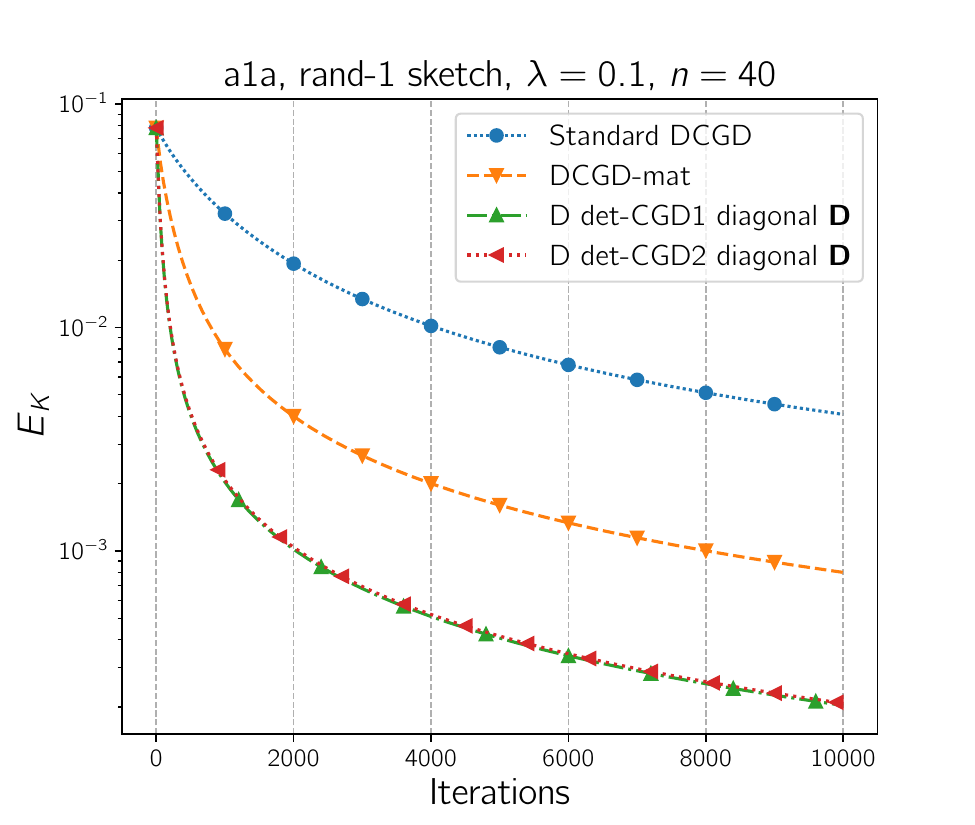} 
		\includegraphics[width=0.32\textwidth]{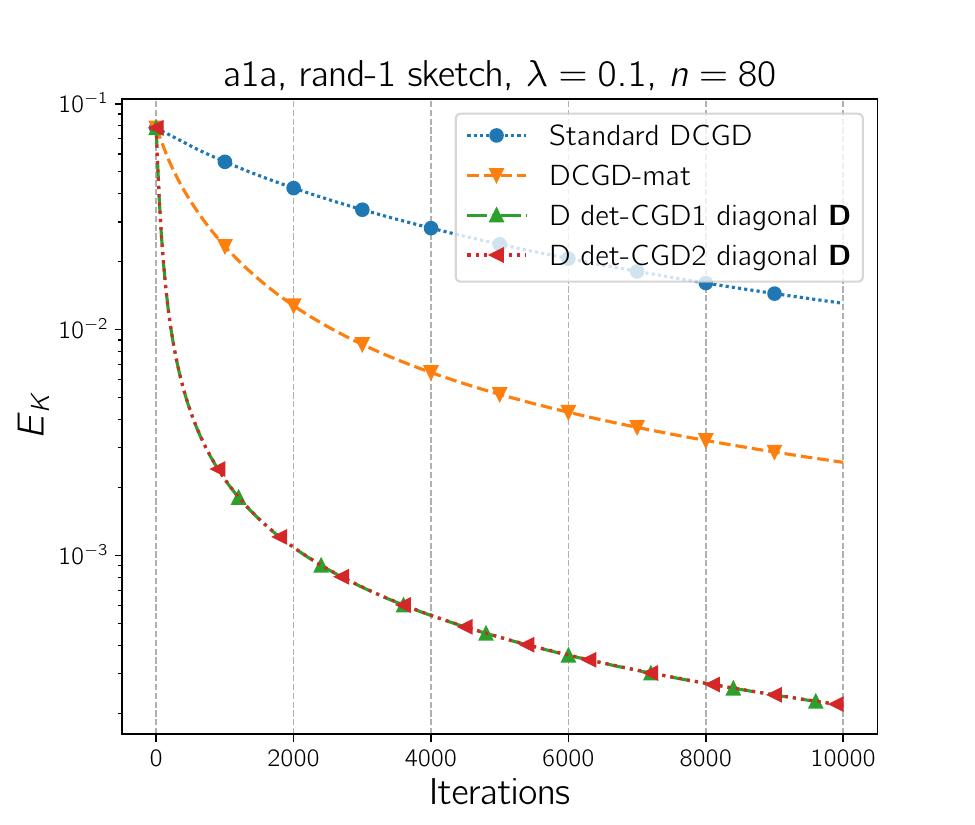}
        \includegraphics[width=0.32\textwidth]{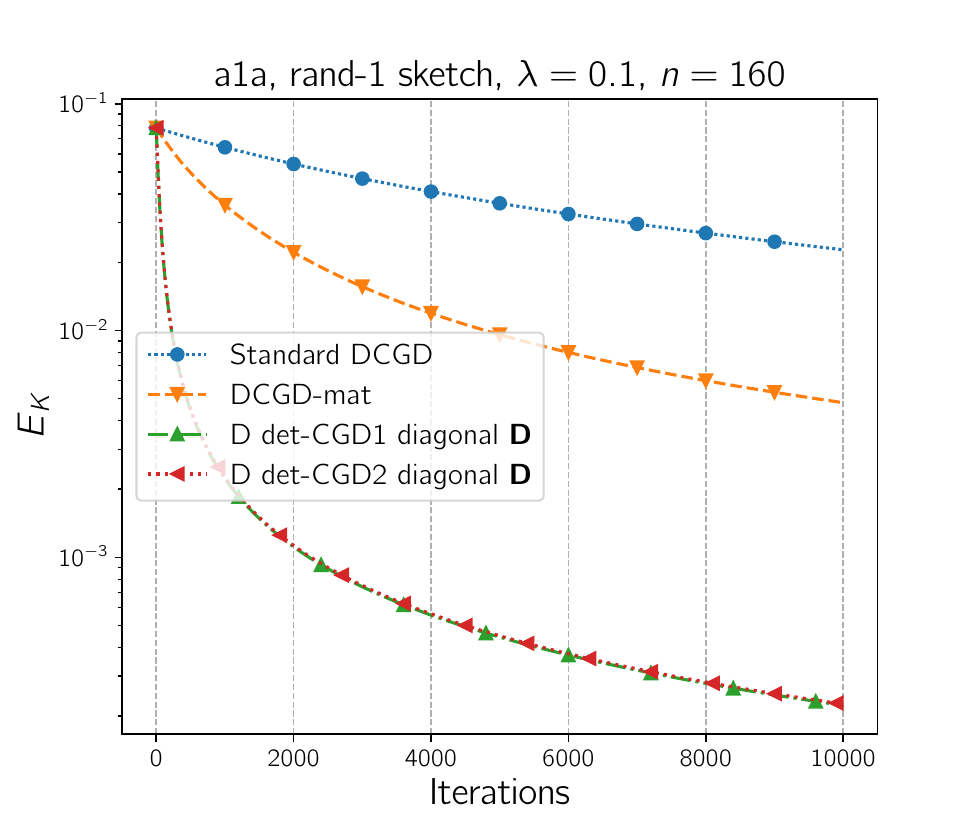}
	\end{minipage}
    }

    \subfigure{
	\begin{minipage}[t]{0.98\textwidth}
		\includegraphics[width=0.32\textwidth]{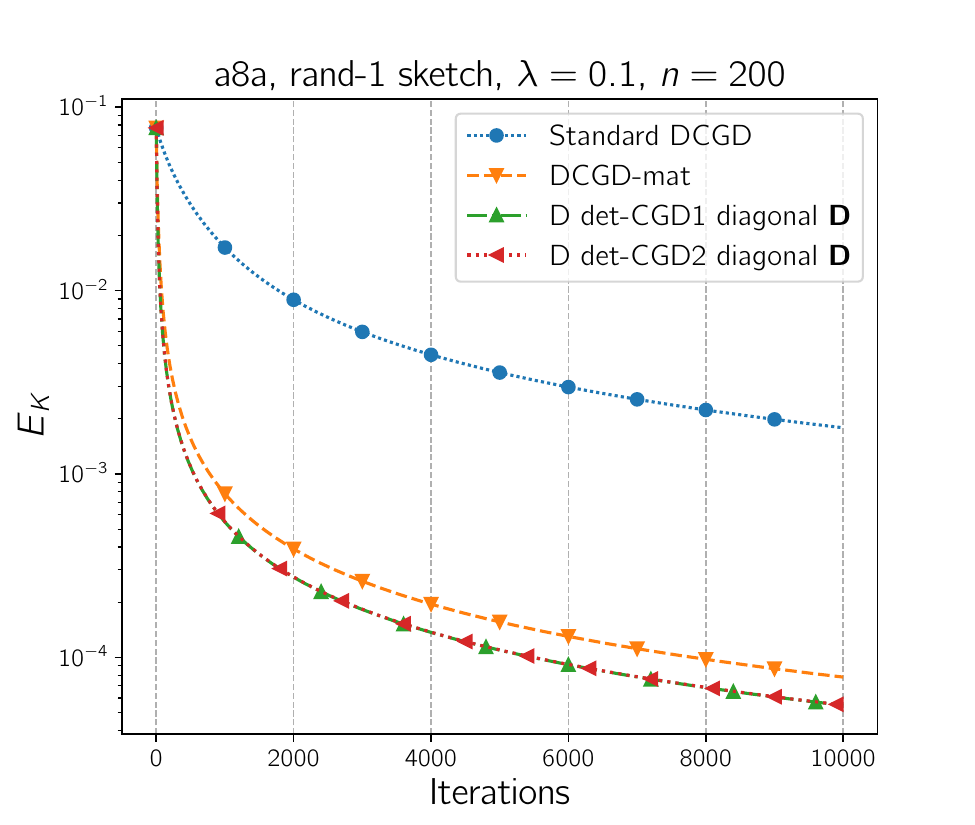} 
		\includegraphics[width=0.32\textwidth]{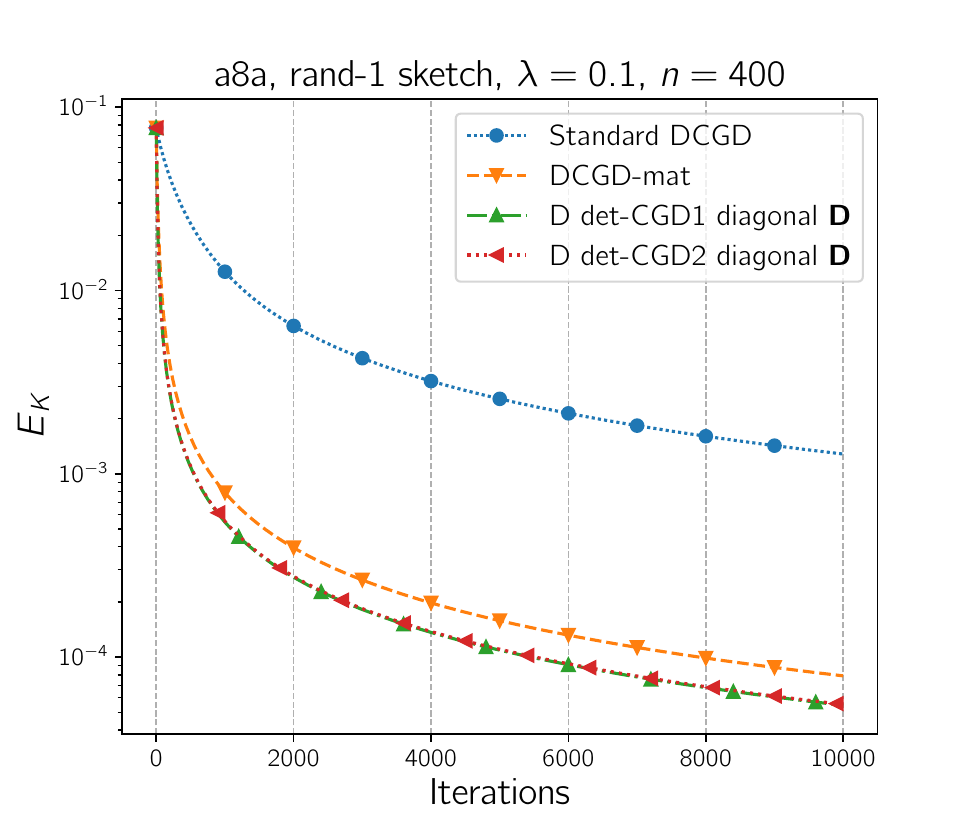}
        \includegraphics[width=0.32\textwidth]{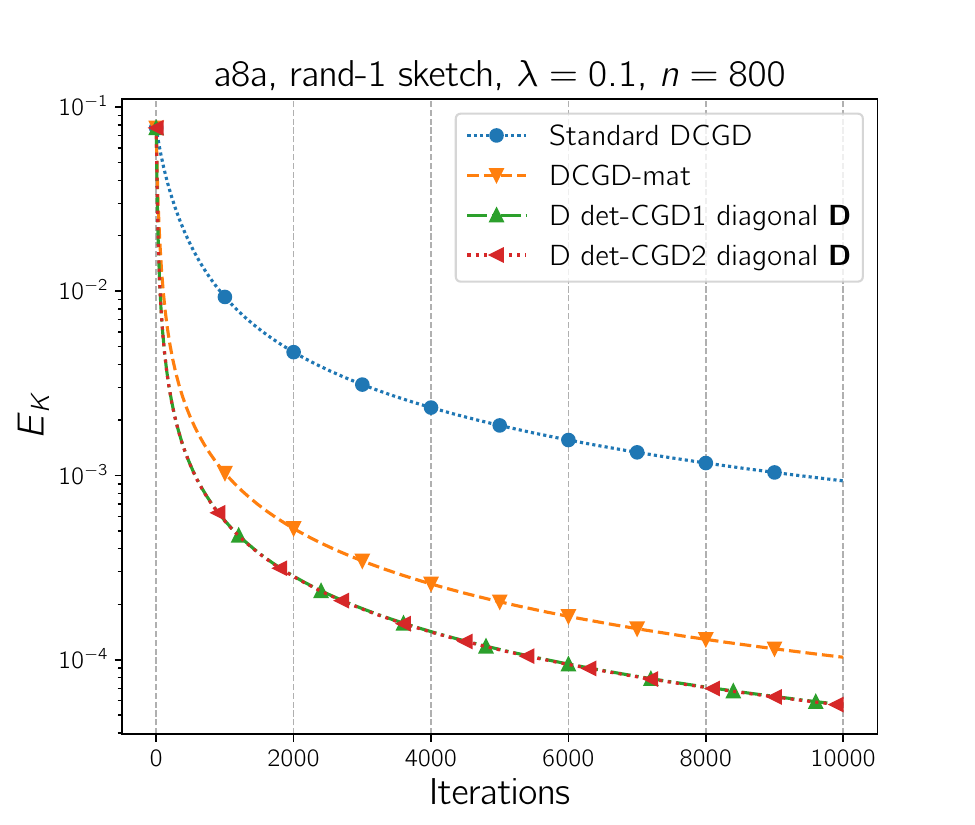}
	\end{minipage}
    }

    \subfigure{
	\begin{minipage}[t]{0.98\textwidth}
		\includegraphics[width=0.32\textwidth]{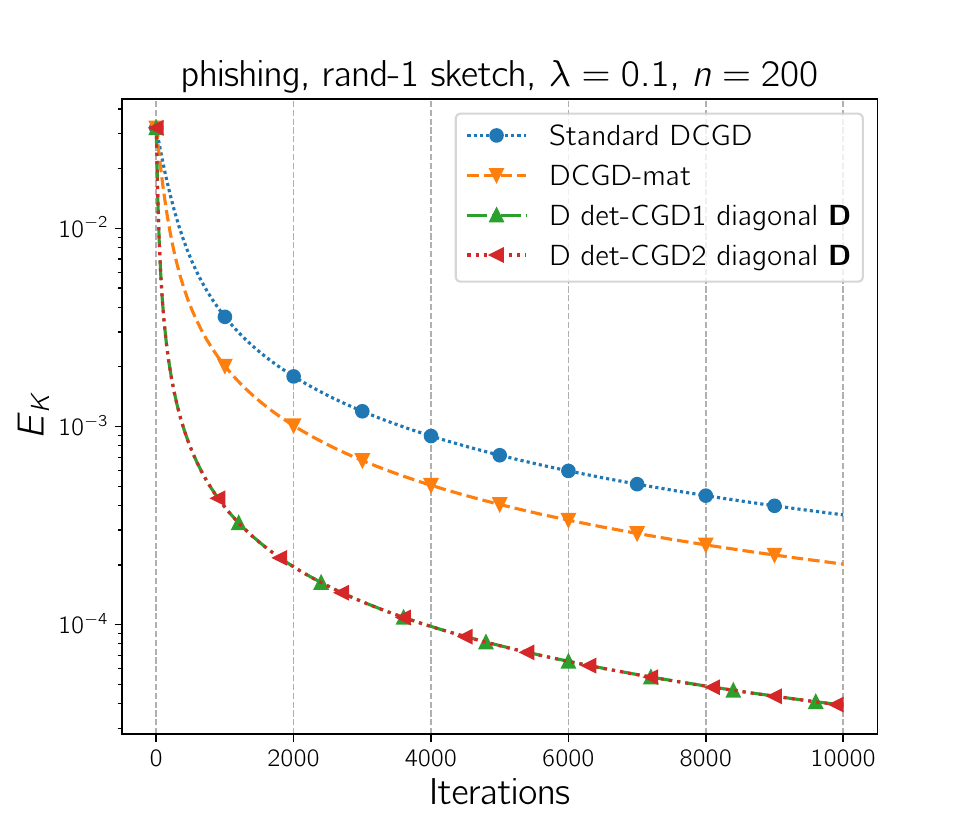} 
		\includegraphics[width=0.32\textwidth]{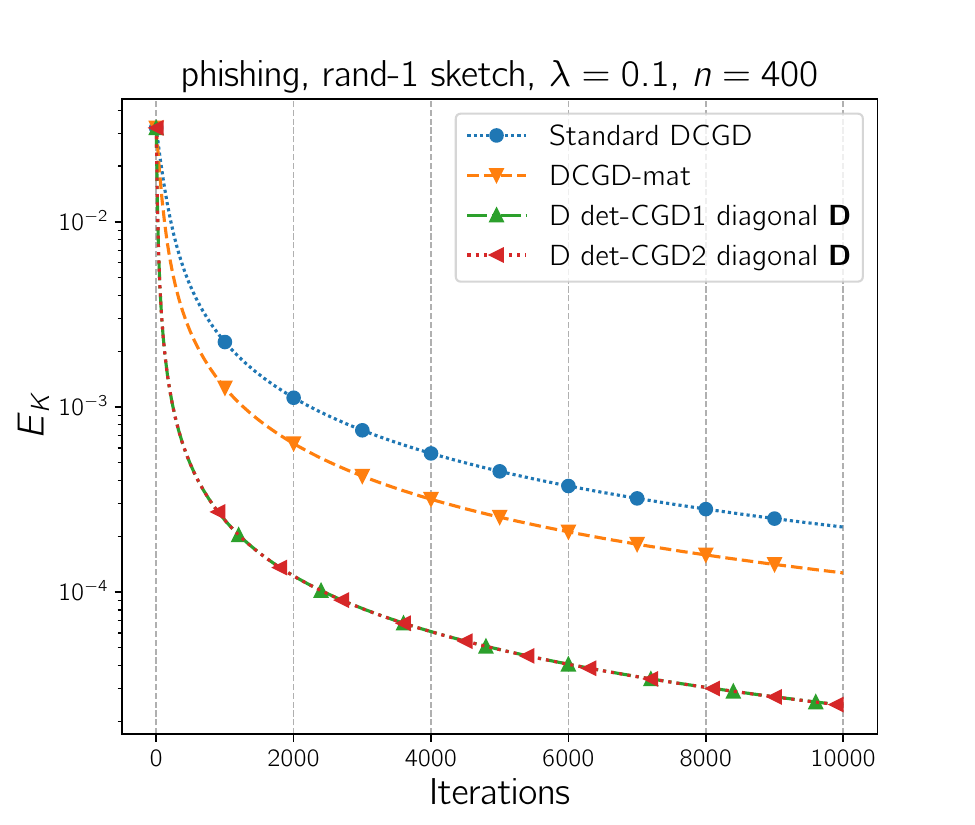}
        \includegraphics[width=0.32\textwidth]{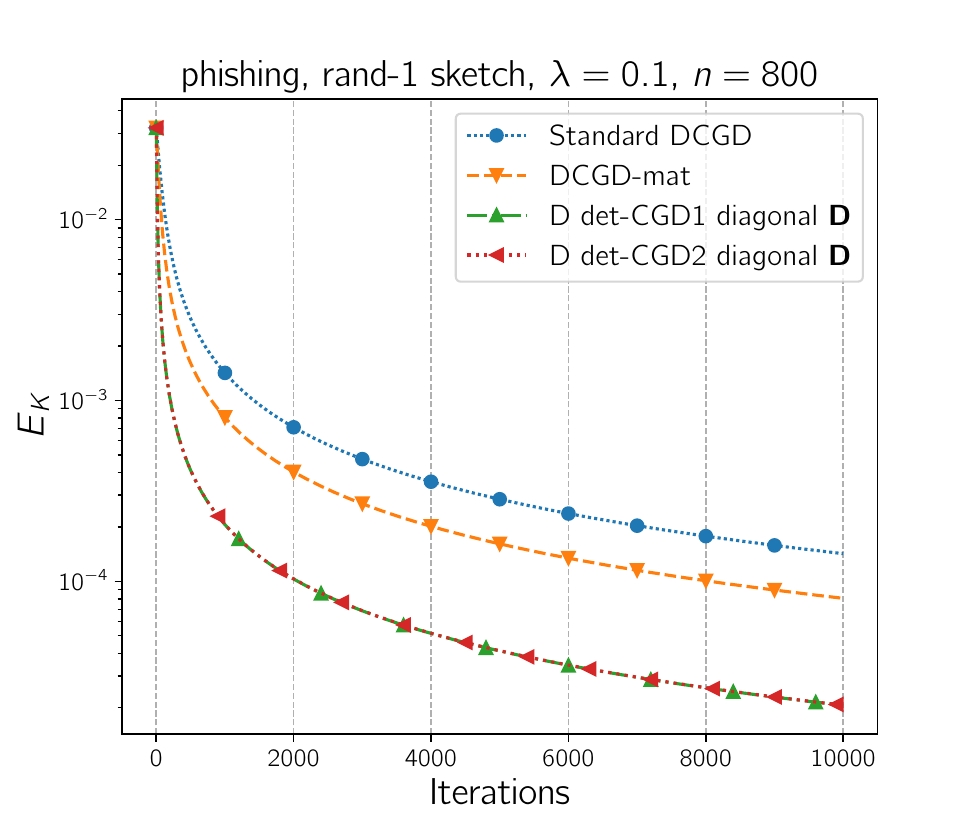}
	\end{minipage}
    }
    
    \caption{Comparison of standard DCGD, DCGD-mat, D-\ref{eq:alg1} with matrix stepsize $\mD_1$ and D-\ref{eq:alg2} with matrix stepsize $\mD_2$, where $\mD_1, \mD_2$ are the optimal diagonal matrix stepsizes for D-\ref{eq:alg1} and D-\ref{eq:alg2} respectively. Rand-$1$ sketch is used in all the algorithms throughout the experiment. The $y$-axis is now standard Euclidean norm defined in \eqref{eq:standard-euc}.}
    \label{fig:experiment-4-sub-2}
\end{figure}
\medskip

\end{document}